%% file: CIME_YAML.tex
\documentclass[graybox]{svmult}

\usepackage{type1cm}        
\usepackage{makeidx}         
\usepackage{graphicx}        
\usepackage{multicol}        
\usepackage[bottom]{footmisc}

\usepackage{newtxtext}       %
\usepackage[varvw]{newtxmath}

\usepackage{amssymb}
\usepackage{amsmath}
\usepackage{dsfont}
\usepackage{color}
\usepackage{hyperref}
\usepackage{fullpage} 
\usepackage{xcolor}
\usepackage{graphicx}
\usepackage{empheq}
\usepackage{ dsfont}
\usepackage{color}
 \usepackage{colordvi} 
\usepackage{xcolor}

\usepackage{caption}
\usepackage{subcaption} 
\usepackage{graphics} 
\usepackage{graphicx}
\usepackage{capt-of} 
\usepackage{algorithm}
\usepackage[noend]{algpseudocode}
 \usepackage{tikz}

\newcommand{\EE}{\mathbb E}
\newcommand{\PP}{\mathbb P}

\newcommand{\RR}{\mathbb R}

\newcommand{\TT}{\mathbb T}

\newcommand{\R}{\RR}

\newcommand{\cC}{\mathcal C}

\newcommand{\cF}{\mathcal F}
\newcommand{\cG}{\mathcal G}
\newcommand{\cH}{\mathcal H}

\newcommand{\cJ}{\mathcal J}

\newcommand{\cL}{\mathcal L}
\newcommand{\cM}{\mathcal M}

\newcommand{\cP}{\mathcal P}

\newcommand{\cT}{\mathcal T}
\newcommand{\cU}{\mathcal U}

\newcommand{\cW}{\mathcal W}

\newcommand{\diver}{\mathrm{div}}
\renewcommand{\div}{\mathrm{div}}
\newcommand{\grad}{\nabla}
\newcommand{\ds}{\displaystyle}

\newcommand{\dt}{{\Delta t}}

\usepackage{xcolor}
\definecolor{mypurple}{RGB}{140,0,255}
\definecolor{myred}{rgb}{255,0,0}

\definecolor{mydarkturquoise}{RGB}{0,206,209}
\definecolor{mydeeppink}{RGB}{255,20,147}
\definecolor{darkblue}{RGB}{0,0,140}
\definecolor{blue2}{RGB}{0,0,0}
\definecolor{middleblue}{RGB}{0,0,71}
\definecolor{light-gray}{gray}{0.9}

\definecolor{ProcessBlue}{cmyk}{1,0,0,0.40}
\definecolor{Black}{cmyk}{0,0,0,1}
\definecolor{Red}{cmyk}{0,1,1,0.2}
\definecolor{Green}{cmyk}{0.9,0,1,0}
\definecolor{Orange}{cmyk}{0,0.61,0.87,0.5}
\definecolor{Fuchsia}{cmyk}{0.47,0.91,0,0.06}
\definecolor{PineGreen}{cmyk}{0.92,0,0.59,0.30}

\DeclareMathOperator*{\argmax}{arg\,max}
\DeclareMathOperator*{\argmin}{arg\,min}

\renewcommand{\Im}{\mathrm{Im}}
\newcommand{\Ker}{\mathrm{Ker}}

\makeindex             

\begin{document}

\title*{Mean Field Games and Applications: Numerical Aspects}
\author{Yves Achdou and Mathieu Lauri{\`e}re}
\institute{Yves Achdou  \at Universit{\'e} de Paris, Laboratoire Jacques-Louis Lions (LJLL), F-75013 Paris, France, \email{achdou@ljll-univ-paris-diderot.fr}
\and Mathieu Lauri{\`e}re  \at Princeton University, Operations Research and Financial Engineering (ORFE) department, Sherrerd Hall, Princeton, New Jersey, U.S.A., \email{lauriere@princeton.edu}}

%
%
\maketitle

\abstract*{No more than 200 words}


\abstract{
The theory of  mean field games  aims at studying deterministic or stochastic differential
games (Nash equilibria) as the number of agents tends to infinity. Since  very few mean field games have explicit or semi-explicit solutions, 
numerical simulations  play a crucial role in obtaining quantitative information from this class of models.
They may lead to  systems of  evolutive partial differential equations coupling a backward Bellman equation and a forward Fokker-Planck equation.
In the present survey, we focus on such systems.  The forward-backward structure is an important feature of 
this system, which makes it necessary to design unusual strategies for mathematical analysis and   numerical approximation.
In this survey,  several aspects of a finite difference method used to approximate the previously mentioned system of PDEs are discussed, including convergence,
variational aspects and algorithms for solving the resulting  systems of nonlinear equations.
Finally, we discuss in details two applications of mean field games  to the study of crowd motion and to macroeconomics,  a comparison with mean field type control, and present numerical simulations.
}

\section{Introduction}
\label{sec:intro}

The theory of  mean field games ({\sl MFGs} for short) aims at studying deterministic or stochastic differential
games (Nash equilibria) as the number of agents tends to infinity. It has  been introduced in the pioneering  works of J-M. Lasry and P-L. Lions~\cite{MR2269875,MR2271747,MR2295621} and a similar approach has been formulated independently by Caines, Huang, and Malham{\'e} under the name of Nash Certainty Equivalence principle~\cite{MR2346927}. In MFGs, one supposes that the rational agents are  indistinguishable and  individually have a negligible influence on the game, and that each individual strategy is influenced by some averages of quantities depending on the states (or the controls) of the other agents. The applications of MFGs are numerous, from economics to the study of crowd motion. On the other hand, very few MFG problems have explicit or semi-explicit solutions. Therefore, numerical simulations of MFGs play a crucial role in obtaining quantitative information from this class of models.

The paper is organized as follows: in section \ref{finite_difference_scheme}, we discuss finite difference schemes for the system of forward-backward PDEs. In section \ref{sec:varMFG}, we focus on the case when the latter system can be interpreted as the system of optimality of a convex control problem driven by a linear PDE (in this case, the terminology {\sl variational MFG} is often used), and put the stress on primal-dual optimization methods that may be used in this case. Section \ref{sec:mult-prec} is devoted to multigrid preconditioners that prove to be very important in particular as ingredients of the latter primal-dual methods for  stochastic variational MFGs (when the volatility is positive). In Section \ref{sec:other-algor-solv}, we address numerical algorithms that may be used also for non-variational MFGs.  Sections \ref{sec:an-example-with}, \ref{MFTC} and \ref{sec:mfgs-macroeconomics} are devoted to examples of applications of mean field games and their numerical simulations. Successively, we consider a model for a pedestrian flow with two populations, a comparison between mean field games and mean field control still for a model of  pedestrian flow, and applications of MFGs to the field of macroeconomics.

In what follows, we suppose for simplicity that the state space is the $d$-dimensional torus $\TT^d$,  and we fix a finite time horizon $T>0$. The periodic setting makes it possible to avoid the discussion on non periodic boundary conditions which always bring additional difficulties. The results stated below may be generalized to other boundary conditions, but this would lead us too far. Yet, sections~\ref{sec:an-example-with}, \ref{MFTC} and \ref{sec:mfgs-macroeconomics} below, which are devoted to some applications of MFGs and to numerical simulations, deal with realistic boundary conditions. In particular, boundary conditions linked to state constraints play a key role in section~\ref{sec:mfgs-macroeconomics}.

We will use the notation $Q_T = [0,T] \times \TT^d$, and $\langle \cdot, \cdot \rangle$ for the inner product of two vectors (of compatible sizes).

Let $f: \TT^d \times \RR \times \RR^d \to \RR, (x,m,\gamma) \mapsto f(x,m,\gamma)$ and $\phi: \TT^d \times \RR \to \RR, (x,m) \mapsto \phi(x,m)$ be respectively a running cost and a terminal cost, on which assumptions will be made later on. Let $\nu>0$ be a constant parameter linked to the volatility. Let $b:  \TT^d \times \RR \times \RR^d \to \RR^d, (x,m,\gamma) \mapsto b(x,m,\gamma)$ be a drift function.

We consider the following MFG: find a flow of probability densities $\hat m: Q_T \to \RR$ and a feedback control $\hat v: Q_T\to \RR^d$ satisfying the following two conditions:
\begin{enumerate}
	\item $\hat v$ minimizes
\begin{align*}
	J_{\hat m}: v \mapsto J_{\hat m}(v) = \EE \left[\int_0^T f(X_t^v, \hat m(t,X_t^v), v(t,X_t^v) ) dt + \phi(X_T^v, \hat m(T,X_T^v)) \right]
\end{align*}
under the constraint that the  process $X^v = (X_t^v)_{t \ge 0}$ solves the stochastic differential equation (SDE)
\begin{equation}
\label{eq:dyn-X-general}
	d X_t^v = b(X_t^v, \hat m(t,X_t^v), v(t, X_t^v)) dt + \sqrt{2 \nu} d W_t, \qquad t \ge 0,
\end{equation}
and $X_0^v$ has distribution with density $m_0$;
	\item For all $t \in [0,T]$, $\hat m(t,\cdot)$ is the law of $X_t^{\hat v}$.
\end{enumerate}

It is useful to note that for a given feedback control $v$, the density $m^v_t$ of the law of $X^v_t$ following~\eqref{eq:dyn-X-general} solves the Kolmogorov-Fokker-Planck (KFP) equation:
\begin{equation}
\label{eq:intro-MFG-KFP}
\begin{cases}
	\displaystyle
	\frac{\partial m^{v}}{\partial t}(t,x)  - \nu \Delta m^{v}(t,x) + \div\left( m^{v}(t,\cdot) b(\cdot, \hat m(t,\cdot), v(t,\cdot)) \right)(x) = 0, 
	&\hbox{ in } (0,T] \times \TT^d ,
	\\
	m^{v}(0,x) = m_0(x), 
	&\hbox{ in } \TT^d.
\end{cases}
\end{equation}

Let $H: \TT^d \times \RR \times \RR^d \ni (x,m,p) \mapsto H(x,m,p) \in \RR$ be the Hamiltonian of the control problem faced by an infinitesimal agent in the first point above, which is defined by
$$
	H: \TT^d \times \RR \times \RR^d \ni (x,m,p) \mapsto H(x,m,p) = \max_{\gamma \in \RR^d} -L(x,m,\gamma,p)\in \RR,
$$
where $L$ is the Lagrangian, defined by
$$
L: \TT^d \times \RR \times \RR^d \times \RR^d \ni (x,m,\gamma,p) \mapsto L(x,m,\gamma,p) =
f(x,m,\gamma) + \langle b(x,m,\gamma) , p \rangle \in \RR.
$$
In the sequel, we will assume that the running cost $f$ and the drift $b$ are such that $H$ is well-defined, $\cC^1$ with respect to $(x,p)$, and strictly convex with respect to $p$.

From standard optimal control theory, one can characterize the best strategy through the value function $u$ of the above optimal control problem for a typical agent, which satisfies a Hamilton-Jacobi-Bellman (HJB) equation. Together with the equilibrium condition on the distribution, we obtain that the equilibrium best response $\hat v$ is characterized by
$$
	\hat v(t,x) = \argmax_{a \in \RR^d} \big\{ -f(x, m(t,x), a) - \langle  b(x, m(t,x), a) , \nabla u(t,x) \rangle  \big\},
$$
where $(u,m)$ solves the following forward-backward PDE system:
\begin{subequations}\label{eq:PDE-system-MFG}
     \begin{empheq}[left=\empheqlbrace]{align}
     	\displaystyle
	&-\frac{\partial u}{\partial t}(t,x)  - \nu \Delta u(t,x) + H(x, m(t,x), \grad u(t,x))=0, 
	&&\hbox{ in } [0,T) \times \TT^d,
	\label{eq:PDE-system-MFG-HJB}
	\\
	\displaystyle
	&\frac{\partial m}{\partial t}(t,x)  - \nu \Delta m(t,x) - \div\left( m(t,\cdot) H_p (\cdot, m(t,\cdot), \grad u(t,\cdot)) \right)(x) = 0, 
	&&\hbox{ in } (0,T] \times \TT^d,
	\label{eq:PDE-system-MFG-KFP}
	\\
	&u(T,x) = \phi(x, m(T,x)), \qquad m(0,x) = m_0(x), 
	&&\hbox{ in } \TT^d.
	\label{eq:PDE-system-MFG-IT}
     \end{empheq}
\end{subequations}

\begin{example}[$f$ depends separately on $\gamma$ and $m$]
\label{ex:special-case}
Consider the case where the drift is the control, i.e. $b(x,m,\gamma) = \gamma$, and the running cost is of the form $f(x,m,\gamma) = L_0(x,\gamma) + f_0(x,m)$ where $L_0(x,\cdot): \RR^d \ni \gamma \mapsto L_0(x,\gamma) \in \RR$ is strictly convex and such that $\lim_{|\gamma|\to \infty} \min_{x\in \TT^d} \frac {L_0(x,\gamma)} {|\gamma|}=+\infty$. We set $H_0(x,p)= \max_{\gamma\in \RR^d} \langle -p,\gamma\rangle - L_0(x,\gamma)$, which is convex with respect to $p$. Then
	$$
	H(x,m,p)
	 = \max_{\gamma \in \RR^d} \{- L_0(x,\gamma) -  \langle \gamma , p  \rangle  \} - f_0(x,m)
	 =  H_0(x, p)  - f_0(x,m).
	$$
	In particular, if $H_0(x,\cdot) = H^*_0(x,\cdot) = \tfrac{1}{2}|\cdot|^2$, then the maximizer in the above expression is
        $ \hat \gamma (p) = - p$,  the Hamiltonian reads $H(x,m,p) = \tfrac{1}{2} |p|^2 -f_0(x,m)$
        and the equilibrium best response is $\hat v(t,x) = - \nabla u(t,x)$ where $(u,m)$ solves the PDE system
	\begin{subequations}
     \begin{empheq}[left=\empheqlbrace]{align*}
     	\displaystyle
	&-\frac{\partial u}{\partial t}(t,x)  - \nu \Delta u(t,x) + \frac{1}{2} |\grad u(t,x)|^2 = f_0(x, m(t,x)), 
	&&\hbox{ in } [0,T) \times \TT^d,
	\\
	\displaystyle
	&\frac{\partial m}{\partial t}(t,x)  - \nu \Delta m(t,x) - \div\left( m(t,\cdot) \grad u(t,\cdot) \right)(x) = 0, 
	&&\hbox{ in } (0,T] \times \TT^d,
	\\
	&u(T,x) = \phi(x, m(T,x)), \qquad m(0,x) = m_0(x), 
	&&\hbox{ in } \TT^d.
     \end{empheq}
\end{subequations}
\end{example}

\begin{remark}\label{sec:introduction}
The setting presented above is somewhat restrictive and does not cover the case when the Hamiltonian depends non locally on $m$.
 Nevertheless, the latter situation makes a lot of sense from the modeling viewpoint. The case when the Hamiltonian depends in a separate manner on $\nabla u$ and $m$ and when the coupling cost 
 continuously  maps probability measures on $\TT^d$ to smooth functions plays an important role in the theory of mean field games, because it permits to obtain
 the most elementary existence results of strong solutions of the system of PDEs, see~\cite{MR2295621}. Concerning the numerical aspects, all what follows may be easily adapted to the latter case, see the numerical simulations at the end of  Section \ref{finite_difference_scheme}. Similarly, the case when the volatility  $\sqrt{2\nu}$ is a function of the state variable $x$ can also be dealt with by finite difference schemes. 
\end{remark}

\begin{remark}
  \label{sec:introduction-1}
Deterministic mean field games  (i.e. for $\nu=0$) are also quite meaningful. One may also consider volatilities as functions of  the state variable that may vanish.
 When the Hamiltonian depends separately on $\nabla u$ and $m$ and the coupling cost is a smoothing map, see Remark~\ref{sec:introduction}, the Hamilton-Jacobi equation (respectively the Fokker-Planck equation) should be understood in viscosity  sense (respectively in the sense of of distributions), see the notes of P. Cardaliaguet~\cite{cardaliaguet2010}. When the coupling costs depends locally on $m$, then under some assumptions on the growth at infinity of the coupling cost as a function of $m$, it is possible to propose a notion of weak solutions to the system of PDEs, see~\cite{MR2295621,MR3399179,MR3305653}. The numerical schemes discussed below can also be applied to these situations, even 
if the convergence results in the available literature are obtained under the hypothesis that $\nu$ is bounded from below by a positive constant. The results and methods presented in Section~\ref{sec:varMFG} below for variational MFGs, i.e. when the system of PDEs can be seen as the optimality conditions of an optimal control problem driven by a PDE, hold if $\nu$ vanishes.
In Section \ref{sec:mfgs-macroeconomics} below, we present some applications and numerical simulations for which the viscosity is zero.
\end{remark}
\begin{remark}
  \label{sec:introduction-2}
The study of the so-called {\sl master equation} plays a key role in the theory of MFGs, see~\cite{PLL-CDF,MR3967062}: the mean field game is described by a single first or second order time dependent partial differential equation in the set of probability measures, thus in an infinite dimensional space in general.  When the state space is finite (with say $N$ admissible states), the distribution of states is a linear combination of $N$ Dirac masses, and the master equation is  posed in $\R^N$;  numerical simulations are then possible, at least if $N$ is not too large. We will not discuss this aspect  in the present notes.
\end{remark}

\section{Finite difference schemes}
\label{finite_difference_scheme}

In this section, we present a finite-difference scheme first introduced in~\cite{MR2679575}. We consider the special case described in Example~\ref{ex:special-case}, with $H(x,m,p) = H_0(x,p) - f_0(x,m)$, although similar methods have been proposed, applied and at least partially analyzed in situations when the Hamiltonian does not depend separately on $m$ and $p$, for example models addressing congestion, see for example~\cite{YAJML}. 

To alleviate the notation, we present the scheme in the one-dimensional setting, i.e. $d=1$ and the domain is $\TT$.

\begin{remark}
  Although we focus on the time dependent problem, a similar scheme has also been studied for the ergodic MFG system, see~\cite{MR2679575}.
\end{remark}

\paragraph{\textbf{Discretization.} }

Let $N_T$ and $N_h$ be two positive integers. We consider $(N_T+1)$ and $(N_h+1)$ points in  time and space respectively. Let $\Delta t = T/N_T$ and $h = 1/N_h$, and $t_n = n \times \Delta t, x_i = i \times h$ for $(n,i) \in \{0,\dots,N_T\}\times\{0,\dots,N_h\}$. 

We approximate $u$ and $m$ respectively by vectors $U$ and $M \in \RR^{(N_T+1)\times(N_h+1)}$, that is, $u(t_n,x_i) \approx U^{n}_{i}$ and $m(t_n,x_i) \approx M^{n}_{i}$ for each $(n,i) \in \{0,\dots,N_T\}\times\{0,\dots,N_h\}$. We use a superscript and a subscript respectively for the time and space indices. Since we consider periodic boundary conditions, we slightly abuse notation and for $W \in \RR^{N_h+1}$, we identify $W_{N_h+1}$ with $W_1$, and $W_{-1}$ with $W_{N_h-1}$.

We introduce the finite difference operators
\begin{align*}
	(D_t W)^n &= \frac{1}{\Delta t}(W^{n+1} - W^n), \qquad &&n \in \{0, \dots N_T-1\}, \qquad W \in \RR^{N_T+1},
	\\
	(D W)_i &= \frac{1}{h} (W_{i+1} - W_{i}), \qquad &&i  \in \{0, \dots N_h\}, \qquad W \in \RR^{N_h+1},
	\\
	(\Delta_h W)_i &= -\frac{1}{h^2} \left(2 W_i - W_{i+1} - W_{i-1}\right), \qquad &&i \in \{ 0, \dots N_h\}, \qquad W \in \RR^{N_h+1},
	\\
	[\grad_h W]_i &= \left( (D W)_{i}, (D W)_{i-1} \right)^\top, \qquad &&i \in \{ 0, \dots N_h\}, \qquad W \in \RR^{N_h+1}.
\end{align*}

\paragraph{\textbf{Discrete Hamiltonian.} }
Let $\tilde H: \TT \times \RR \times \RR \to \RR, (x,p_1,p_2) \mapsto \tilde H(x,p_1,p_2)$ be  a discrete Hamiltonian, assumed to satisfy the following properties:
\begin{enumerate}
	\item ($\mathbf{\tilde H_1}$) Monotonicity:  for every $x \in \TT$, $\tilde H$ is nonincreasing in $p_1$ and nondecreasing in $p_2$.
	\item  ($\mathbf{\tilde H_2}$)  Consistency:  for every $x \in \TT, p \in \RR$, $\tilde H(x,p,p) = H_0(x,p)$.
	\item  ($\mathbf{\tilde H_3}$)  Differentiability:  for every $x \in \TT$, $\tilde H$ is of class $\mathcal C^1$ in $p_1,p_2$
	\item  ($\mathbf{\tilde H_4}$)  Convexity:   for every $x \in \TT$, $(p_1,p_2) \mapsto \tilde H(x,p_1,p_2)$ is convex. 
\end{enumerate}

\begin{example}
\label{eq-quadratic-discrete}
For instance, if $H_0(x,p) = \frac{1}{2} |p|^2$, then one can take $\tilde H(x, p_1, p_2) = \frac{1}{2}|P_K(p_1,p_2)|^2$ where $P_K$ denotes the projection on $K = \RR_- \times \RR_+$.
\end{example}

\begin{remark}
  Similarly, for $d$-dimensional problems, the discrete Hamiltonians that we consider are real valued functions defined on $\TT^d \times (\RR^2)^d $.
\end{remark}

\paragraph{\textbf{Discrete HJB equation.} }
We consider the following discrete version of the HJB equation~\eqref{eq:PDE-system-MFG-HJB}:
\begin{subequations}\label{eq:discrete-HJB}
     \begin{empheq}[left=\empheqlbrace]{align}
     	\displaystyle
	\,\, &- (D_t U_{i})^{n} - \nu (\Delta_h U^{n})_{i}
	+ \tilde H(x_i, [\grad_h U^n]_i) 
	= f_0(x_i,M^{n+1}_i) \, , 
	&& i \in \{0,\dots,N_h\} \, , \,  n \in \{0, \dots, N_T-1\} \, ,
	\\
	\,\, &U^{n}_{0} = U^{n}_{N_h} \, ,  && n \in \{0, \dots,  N_T-1\} \, ,
	\\
	\,\, &U^{N_T}_{i} = \phi(M^{N_T}_i) \, ,  && i \in \{ 0, \dots , N_{h} \} \, .
	\end{empheq}
\end{subequations}

Note that it is an implicit scheme since the equation is backward in time.

\paragraph{\textbf{Discrete KFP equation.} } \label{sec:textbfd-kfp-equat}
To define an appropriate discretization of the KFP equation~\eqref{eq:PDE-system-MFG-KFP}, we consider the weak form. For a smooth test function $w \in \mathcal{C}^{\infty}([0,T] \times \TT)$, it involves, among other terms,  the expression
\begin{align}
	&- \int_{\TT} \partial_x \big( H_p(x, \partial_x u(t,x)) m(t,x)\big) w(t,x) dx
	=
	\int_{\TT} H_p(x, \partial_x u(t,x)) m(t,x) \, \partial_x w(t,x) dx  \, ,
	\label{eq:HJB-weak}
\end{align}
where we used an integration by parts and the periodic boundary conditions. In view of what precedes, it is quite natural to propose the following discrete version of the right hand side of \eqref{eq:HJB-weak}:
\begin{displaymath}
 h \sum_{i=0}^{N_h-1} M_i^{n+1} \left(  \tilde H_{p_1}(x_i, [\grad_h U^n]_i)  \frac { W^n_{i+1}-W^n_i} h+
    \tilde H_{p_2}(x_i, [\grad_h U^n]_i)  \frac { W^n_{i}-W^n_{i-1}} h   \right),
\end{displaymath}
and performing a discrete integration by parts, we obtain the discrete counterpart  of the left hand side of \eqref{eq:HJB-weak} as follows:
$ -\ds
  h \sum_{i=0}^{N_h-1}  	\cT_i(U^n , M^{n+1}) W_i^n$,
where $\cT_i$ is the following discrete transport operator:
\begin{align*}
	\cT_i(U, M)
	= \frac{1}{h} 
	&\Big( M_i  \tilde H_{p_1}(x_i, [\grad_h U]_i) 
		- M_{i-1} \tilde H_{p_1}(x_{i-1}, [\grad_h U]_{i-1})
		\\
		&\quad +
		M_{i+1} \tilde H_{p_2}(x_{i+1}, [\grad_h U]_{i+1})
		- M_i \tilde H_{p_2}(x_i, [\grad_h U]_i)
	\Big) \, .
\end{align*}

Then, for the discrete version of~\eqref{eq:PDE-system-MFG-KFP}, we consider
\begin{subequations}\label{eq:discrete-KFP}
     \begin{empheq}[left=\empheqlbrace]{align}
	\,\, &(D_t M_{i})^{n} - \nu (\Delta_h M^{n+1})_{i}
	- \cT_i(U^{n}, M^{n+1})
	= 0 \, , 
	&& i \in \{0,\dots,N_h\}, n \in \{0, \dots, N_T-1\} \, ,
	\label{eq:sysH-FP-scheme-discrete-edo}
	\\
	\,\, &M^{n}_{0} = M^{n}_{N_h} \, , &&n \in \{ 1, \dots, N_T\} \, , 
	\label{eq:sysH-FP-scheme-discrete-bc}
	\\
	\,\, &M^{0}_{i} = \bar m_0(x_i) \, , \,  && i \in \{0, \dots, N_h\} \, ,
	\label{eq:sysH-FP-scheme-discrete-initialc}
\end{empheq}
\end{subequations}
where, for example,
\begin{equation}
  \label{eq:1}
	\bar m_0(x_i) = \int_{|x - x_i| \le h/2} m_0(x) dx.
\end{equation}
Here again, the scheme is implicit since the equation is forward in time.

\begin{remark}[Structure of the discrete system]
The finite difference system~\eqref{eq:discrete-HJB}--\eqref{eq:discrete-KFP} preserves the structure of the PDE system~\eqref{eq:PDE-system-MFG} in the following sense: The operator $M \mapsto - \nu (\Delta_h M)_{i}
	- \cT_i(U, M)$
	is the adjoint of the linearization of the operator
	$U \mapsto - \nu (\Delta_h U)_{i}
	+ \tilde H(x_i, [\grad_h U]_i)$ since 
	\begin{align*}
		& \sum_{i} \cT_i(U, M) W_i
		= - \sum_{i} M_i  \left\langle \tilde H_p(x_i, [\grad_h U]_i) , [\grad_h W]_i \right\rangle \, .
	\end{align*}
\end{remark}

\paragraph{\bf Convergence results.}

Existence and uniqueness for the discrete system has been proved in~\cite[Theorems 6 and 7]{MR2679575}.
The monotonicity properties ensure that the grid function $M$ is nonnegative. By construction of $\cT$, the scheme preserves the total mass $h \sum_{i=0}^{N_h-1} M_i^n$. Note that there is no restriction on the time step since the scheme is implicit. Thus, this method may be used for long horizons and
the scheme can be very naturally adapted to ergodic MFGs, see~\cite{MR2679575}.

Furthermore, convergence results are available. A first type of convergence theorems see ~\cite{MR2679575,MR2888257,MR3097034} (in particular~\cite[Theorem 8]{MR2679575} for finite horizon problems) make the assumption that the MFG system of PDEs has a unique classical solution and strong versions of Lasry-Lions monotonicity assumptions, see~\cite{MR2269875,MR2271747,MR2295621}. Under such assumptions, 
the solution to the discrete system converges towards the classical solution as the grid parameters tend to zero.

In what follows, we discuss a second type of results that were obtained in~\cite{MR3452251}, namely, the  convergence of the solution of the discrete problem to weak solutions of the system of forward-backward PDEs.

\begin{theorem}\label{th:conv_weak} \cite[Theorem 3.1]{MR3452251}
  We make the following assumptions:
  \begin{itemize}
\item $\nu>0$;
\item $\phi(x,m)= u_T(x)$ where $u_T$ is a continuous function on $\TT^d$;
\item $m_0$  is a  nonnegative function in $L^\infty (\TT^d)$ such that $\int_{\TT^d} m_0(x)dx=1$;
\item $f_0$ is a continuous function on $\TT^d\times \R^+$, which is bounded from below;
\item The Hamiltonian $(x,p)\mapsto H_0(x,p)$ is assumed to be continuous, convex and  
 $\cC^1$ regular w.r.t.  $p$;
\item The discrete Hamiltonian $\tilde H$ satisfies  ($\mathbf{\tilde H_1}$)-($\mathbf{\tilde H_4}$) and
  the following  further assumption: \\
\begin{description}
\item ($\mathbf{\tilde H_5}$) \emph{growth conditions}: there exist positive constants $c_1, c_2,c_3,c_4$ such that
  \begin{eqnarray}
    \label{eq:27}
\langle \tilde H_q(x,q), q \rangle -\tilde H(x,q)&\ge& c_1 |\tilde H_q(x,q)|^2 - c_2,\\
\label{eq:31}
|\tilde H_q(x,q)|&\le &c_3 |q| + c_4.
  \end{eqnarray}
\end{description}
\end{itemize}
Let $(U^n), (M^n)$ be  a solution of the discrete system (\ref{eq:sysH-FP-scheme-discrete-edo})-(\ref{eq:sysH-FP-scheme-discrete-initialc})  (more precisely
of its $d$-dimensional counterpart).
Let $u_{h,\Delta t}$, $m_{h,\Delta t}$ be the piecewise constant functions which take  the values $U_{I}^{n+1} $ and $M_{I}^{n}$, respectively,
in  $(t_n,t_{n+1})\times \omega_I$, where $\omega_I$ is the $d$-dimensional cube centered at $x_I$ of side $h$ and $I$ is any multi-index in $\{0,\dots,N_h-1\}^d$. \\
Under the assumptions above, there exists a subsequence of $h$ and $\Delta t$ (not relabeled) and  functions $\tilde u$, $\tilde m$,
which belong to $L^\alpha(0,T;W^{1,\alpha}(\TT^d))$ for any $\alpha\in [1,\frac{d+2} {d+1})$, such that $u_{h,\Delta t} \to \tilde u$ and $m_{h,\Delta t} \to \tilde m$ in   $L^\beta(Q_T)$ for all $\beta\in [1,\frac {d+2} d )$,  and $(\tilde u, \tilde m)$ is a weak solution to the system
\begin{subequations}\label{eq:conv:0}
     \begin{empheq}[left=\empheqlbrace]{align}
    \displaystyle
    &-\frac{\partial \tilde u}{\partial t}(t,x)  - \nu \Delta \tilde u(t,x) + H_0\left( x, \grad \tilde u(t,x)\right) = f_0(x, \tilde m(t,x)), 
    &&\hbox{ in } [0,T) \times \TT^d,
    \label{eq:conv:1}
    \\
    \displaystyle
    &\frac{\partial \tilde m}{\partial t}(t,x)  - \nu \Delta \tilde m(t,x) - \div\left( \tilde m(t,\cdot) D_p H_0 \grad \tilde u(t,\cdot) \right)(x) = 0, 
    &&\hbox{ in } (0,T] \times \TT^d,
    \label{eq:conv:2}
    \\
    &\tilde u(T,x) = u_T( x), \qquad \tilde m(0,x) = m_0(x), 
    &&\hbox{ in } \TT^d,
    \label{eq:conv:3}
  \end{empheq}
\end{subequations}
in the following sense:
\begin{itemize}
\item[(i)] $H_0(\cdot, \nabla\tilde u) \in L^1(Q_T)$,  $\tilde mf_0(\cdot,\tilde m)\in L^1(Q_T)$,
  $\tilde m[D_pH_0(\cdot ,\nabla\tilde u)\cdot \nabla\tilde u-H_0(\cdot,\nabla\tilde u)]\in L^1(Q_T)$
\item[(ii)] $(\tilde u, \tilde m)$ satisfies (\ref{eq:conv:1})-(\ref{eq:conv:2})  in the sense of distributions
\item[(iii)]  $\;\tilde u, \tilde m \in C^0([0,T]; L^1(\TT^d))$ and $\tilde u|_{t=T}= u_T$, $\tilde m|_{t=0}=m_0$ .
\end{itemize}
\end{theorem}

\begin{remark}\label{sec:convergence-results}
  \begin{itemize}
  \item Theorem \ref{th:conv_weak} does not suppose the existence of a (weak) solution to
    (\ref{eq:conv:1})-(\ref{eq:conv:3}), nor  Lasry-Lions monotonicity assumptions, see~\cite{MR2269875,MR2271747,MR2295621}. It can thus be seen as an alternative proof of existence of weak solutions of (\ref{eq:conv:1})-(\ref{eq:conv:3}).
  \item The assumption made on the coupling cost $f_0$ is very general. 
  \item If uniqueness holds for the weak solutions of  (\ref{eq:conv:1})-(\ref{eq:conv:3}), which is the case if Lasry-Lions monotonicity assumptions hold, then the whole sequence of discrete solutions converges.
    \item
      It would not be difficult to generalize Theorem \ref{th:conv_weak} to the case when the terminal condition at $T$ is of the form (\ref{eq:PDE-system-MFG-IT}).
    \item Similarly, non-local coupling cost of the type $F[m] (x)$ could be addressed,
where for instance, $F$ maps bounded measures on $\TT^d$ to functions in $C^0(\TT^d)$.
  \end{itemize}
 
\end{remark}

\paragraph{\bf An example illustrating the robustness of the scheme in the deterministic limit}
To illustrate the robustness of the finite difference scheme described above, let us consider the following ergodic problem 
\begin{displaymath}
\left\{  \begin{array}[c]{rcl}
    \ds    \rho - \nu \Delta u    + H(x,\nabla  u) & =& F[m](x), \quad \hbox{in }\TT^2, \\
    \ds -\nu \Delta  m -\diver\left(  m\frac {\partial H} {\partial p} (x,\nabla  u)\right)  &=& 0,\quad \hbox{in }\TT^2,
  \end{array}\right.
  \end{displaymath}
where the ergodic constant $\rho\in \R$ is also an unknown of the problem, and
\begin{displaymath}
  \begin{array}[c]{l}
    \nu =0.001,\quad \quad H(x,p)= \sin(2\pi x_2)+\sin(2\pi x_1)+\cos(4\pi x_1)+ |p|^{3/2},
\\ F[m](x)= \left((1-\Delta)^{-1} (1-\Delta)^{-1} m\right)  (x).
  \end{array}
\end{displaymath}
Note the very small value of the viscosity parameter $\nu$.
Since the coupling term is a smoothing nonlocal operator, nothing prevents $m$ from becoming singular in the deterministic limit when $\nu\to 0$.
In Figure~\ref{fig:exemple1}, we display the solution. We see that the value function tends to be singular (only Lipschitz continuous) and that the density of the distribution 
of states tends to be  the sum of two Dirac masses located at the minima of the value function. The monotonicity of the scheme has made it possible to capture the singularities of the solution.
\begin{figure}[H]
  \centering
 \includegraphics[width=8.0cm]{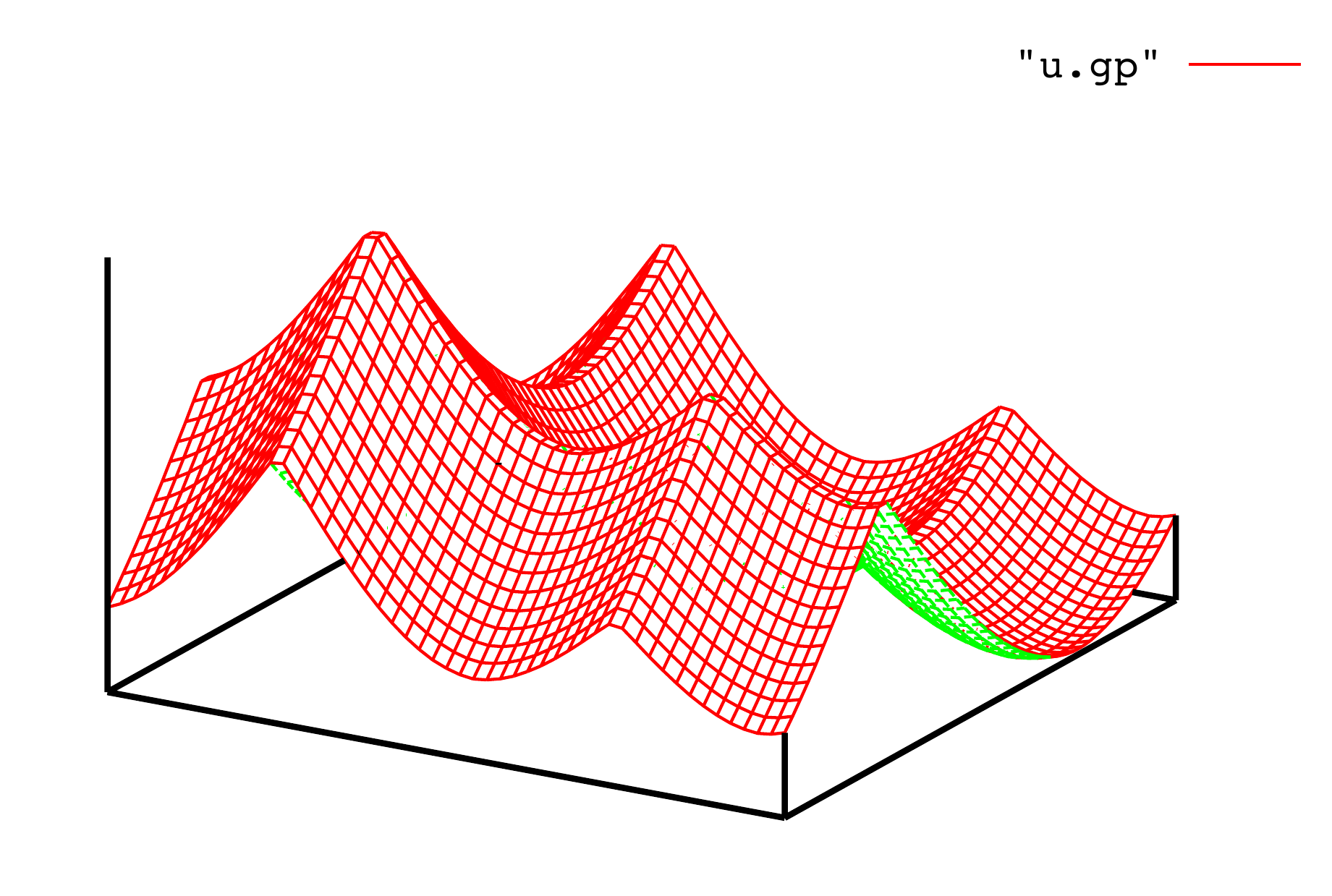} \includegraphics[width=8.0cm]{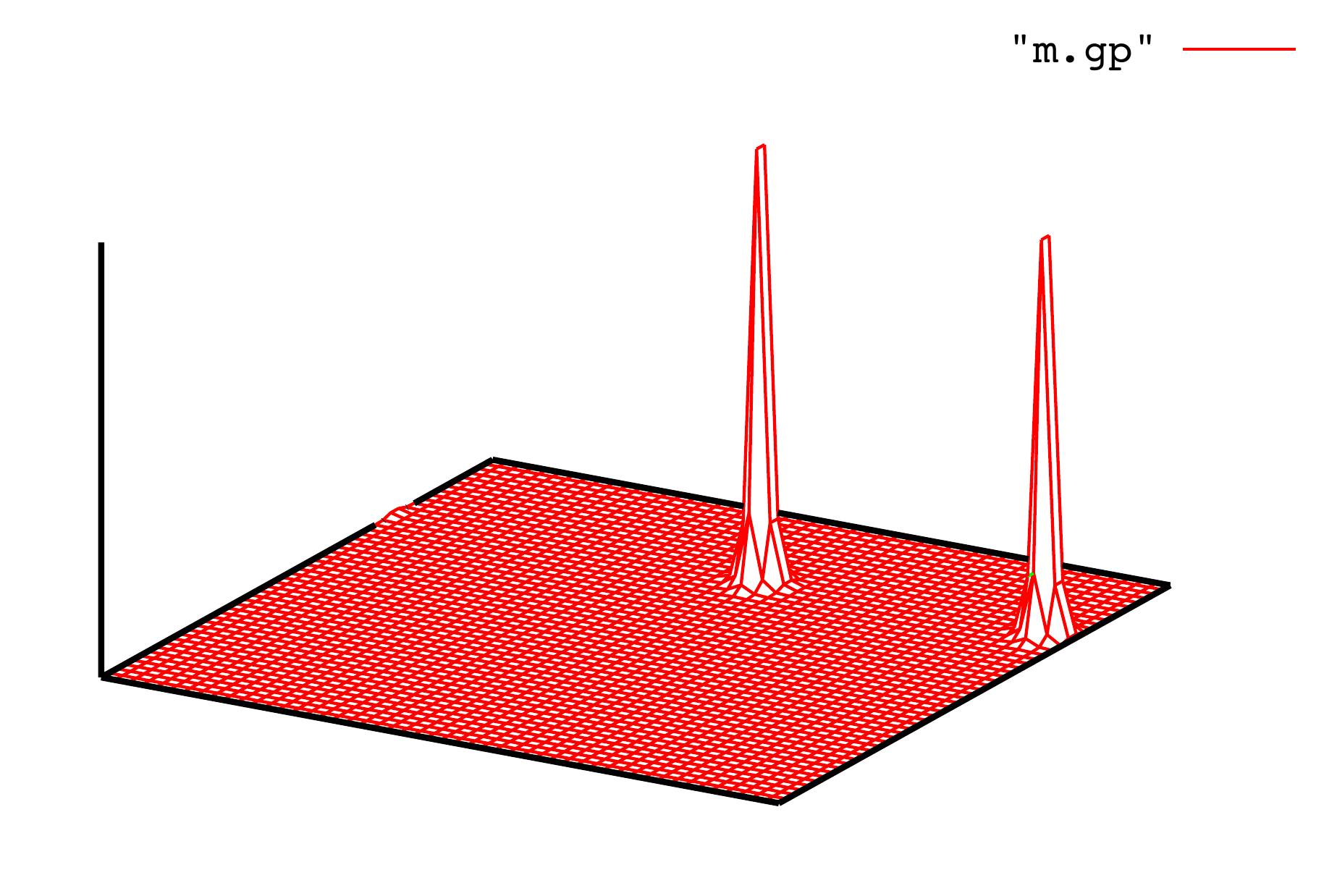}
  \caption{Left: the value function. Right: the distribution of states}
  \label{fig:exemple1}
\end{figure}

\section{Variational aspects of some MFGs and numerical applications}
\label{sec:varMFG}

In this section, we restrict our attention to the setting described in the Example~\ref{ex:special-case}.
We assume  that $L_0$ is smooth, strictly convex in the variable $\gamma$ and such that there exists positive constants $c$ and $C$, and $r>1$ such that
\begin{equation}
  \label{eq:var:0}
-c+ \frac 1 C |\gamma|^r \le L_0(x,\gamma) \le c+ C|\gamma|^r, \quad \forall x\in \TT^d, \gamma\in \RR^d,
\end{equation}
 and that $H_0$ is smooth and has a superlinear growth in the gradient variable.
  We also assume that the function $f_0$ is smooth and non-decreasing with respect to $m$. We make the same assumption on $\phi$.

\subsection{An optimal control  problem  driven by a PDE }\label{sec:variational-problem-2}
Let us consider the functions
\begin{equation}\label{eq:def-F-variational}
	F: \TT^d \times \RR \to \RR, \quad (x,m) \mapsto F(x,m) = 
	\begin{cases}
		\ds
		\,\, \int_0^m f_0(x,s) ds, &\hbox{ if } m \ge 0,
		\\
		\,\, +\infty, &\hbox{ if } m < 0,
	\end{cases}
\end{equation}
and
\begin{equation}\label{eq:def-Phi-variational}
	\Phi: \TT^d \times \RR \to \RR, \quad (x,m) \mapsto \Phi(x,m) = 
	\begin{cases}
		\ds
		\,\, \int_0^m \phi(x,s) ds, &\hbox{ if } m \ge 0,
		\\
		\,\, +\infty, &\hbox{ if } m < 0.
	\end{cases}
\end{equation}
Let us also introduce
\begin{equation}
\label{eq:def-L}
	L: \TT^d \times \RR \times \RR^d \to \RR \cup \{+\infty\}, \quad (x,m,w) \mapsto L(x,m,w) = 
	\begin{cases}
		\ds
		\,\, m L_0\left(x,\frac{w}{m}\right), &\hbox{ if } m > 0,
		\\
		\,\, 0, &\hbox{ if } m=0, \hbox{ and } w= 0,
		\\
		\,\, +\infty, &\hbox{ otherwise.}
	\end{cases}
\end{equation}
Note that since $f_0(x,\cdot)$ is non-decreasing,  $F(x,\cdot)$ is convex and l.s.c. with respect to $m$. Likewise, since $\phi(x,\cdot)$ is non-decreasing, then $\Phi(x,\cdot)$ is convex and l.s.c. with respect to $m$. Moreover,
it can be checked that the assumptions made on $L_0$ imply that  $L(x,\cdot,\cdot)$ is convex and l.s.c. with respect to $(m,w)$.

We now consider the following minimization problem expressed in a formal way: Minimize, when it is possible, $\cJ$ 
defined by
\begin{equation}
\label{eq:variational-form-J}
	(m,w) \mapsto \cJ(m,w) = \int_0^T \int_{\TT^d} \left[ L(x, m(t,x), w(t,x)) + F(x, m(t,x)) \right] dx dt + \int_{\TT^d} \Phi(x, m(T,x)) dx
\end{equation}
on pairs $(m,w)$
such that $m\ge 0$ and 
\begin{equation}
\label{eq:variational-form-linearcons}
\begin{cases}
	\displaystyle
	\,\, \frac{\partial m}{\partial t}  - \nu \Delta m + \div_x w = 0, 
	&\hbox{ in } (0,T] \times \TT^d ,
	\\
	\,\, m|_{t=0} = m_0, 
\end{cases}
\end{equation}
holds in the sense of distributions.

Note that, thanks to our assumptions and to the structure in terms of the variables $(m,w)$, this problem is the minimization of a convex functional under a linear constraint. The key insight is that the optimality conditions of this minimization problem yield a weak version of the MFG system of PDEs~\eqref{eq:PDE-system-MFG}.

\begin{remark}\label{sec:variational-problem-4}
If  $(m,w)\in L^1(Q_T)\times L^1(Q_T; \R^d)$ is such that $\int_0^T \int_{\TT^d}  L (x,m(t,x), w(t,x)) dx dt <+\infty$, then it is possible to write that $w=m \gamma$ with $\int_0^T \int_{\TT^d}  m(t,x)|\gamma(t,x)|^r  dx dt <+\infty$ thanks to the assumptions on $L_0$. From this piece of information and (\ref{eq:variational-form-linearcons}), we deduce that  $t\mapsto m(t)$ for $t<0$  is H{\"o}lder continuous a.e. for the weak$*$ topology of $\cP(\TT^d)$, see~\cite[Lemma 3.1]{MR3399179}.
 This gives a meaning to $\int_{\TT^d} \Phi(x, m(T,x)) dx$ in the case when $\Phi(x, m(T,x))  = u_T(x)m(T,x)$ for a smooth function $u_T(x)$. Therefore, at least in this case, the minimization problem   is defined rigorously in $L^1(Q_T)\times L^1(Q_T; \R^d)$.  
\end{remark}

\begin{remark}\label{sec:variational-problem-1}
Although we consider here the case of a non-degenerate diffusion, similar problems have been considered in the  first-order case or with degenerate diffusion in~\cite{MR3399179,MR3358627}.
\end{remark}

\begin{remark}\label{sec:variational-problem-12}
Note that, in general, the system of forward-backward PDEs cannot be seen as the optimality conditions of a minimization problem. This is never the case if the Hamiltonian does not depend separately on $m$ and $\nabla u$.
\end{remark}

\subsection{Discrete version of the PDE driven optimal control problem}\label{sec:discr-vers-vari}
We turn our attention to a discrete version of the minimization problem introduced above. To alleviate notation, we restrict our attention to the one dimensional case, i.e., $d=1$. Let us introduce the following spaces, respectively for the discrete counterpart of $m,w,$ and $u$:
$$
	\cM = \RR^{(N_T+1) \times N_h}, 
	\qquad 
	\cW = (\RR^2)^{N_T \times N_h}, 
	\qquad
	\cU = \RR^{N_T \times N_h}.
$$
Note that at each space-time grid point, we will have two variables for the value of $w$, which will be useful to define an upwind scheme.

We consider a discrete Hamiltonian $\tilde H$ satisfying  the assumptions ($\mathbf{\tilde H_1}$)-($\mathbf{\tilde H_4}$), and introduce its convex conjugate w.r.t. the $p$ variable:
\begin{equation}
\label{eq:var:1}  
	\tilde H^*: \TT^d \times (\RR^2)^d \to \RR \cup \{+\infty\},
	(x,\gamma) \mapsto 
	\tilde H^*(x, \gamma)
	=
	\max_{p \in \RR^2} \left\{ \langle \gamma , p \rangle - \tilde H(x, p) \right\}.
\end{equation}

\begin{example}\label{sec:discr-vers-vari-2}
	If $\tilde H(x, p_1, p_2) = \tfrac{1}{2} |P_K(p_1,p_2)|^2$ with $K = \RR_- \times \RR_+$, then
	$$
		\tilde H^*(x, \gamma_1, \gamma_2)
		=
		\begin{cases}
			\ds
			\,\, \tfrac{1}{2}|\gamma|^2 , &\hbox{ if } \gamma \in K,
			\\
			\,\, +\infty, &\hbox{ otherwise.}
		\end{cases}
	$$
\end{example}
A discrete counterpart $\Theta$ to the functional $\cJ$ introduced in~\eqref{eq:variational-form-J} can be defined as
\begin{equation}
\label{eq:var:2}  
	\Theta:
	\cM \times \cW \to \RR, 
	\qquad
	(M,W) \mapsto
	\sum_{n=1}^{N_T} \sum_{i=0}^{N_h-1} \left[ \tilde L(x_i, M^n_i, W^{n-1}_i) + F(x_i, M^n_i) \right] + \frac{1}{\Delta t}  \sum_{i=0}^{N_h-1}  \Phi(x_i, M^{N_T}_i),
      \end{equation}
      where $\tilde L$ is a discrete version of $L$ introduced in~\eqref{eq:def-L} defined as
\begin{equation}
\label{eq:var:3}  
	\tilde L:
	\TT \times \RR \times \RR^2 \to \RR \cup \{+\infty\},
	\qquad
	 (x,m,w) \mapsto \tilde L(x,m,w)
	 =
	\begin{cases}
		\ds
		\,\, m \tilde H^*\left(x,-\frac{w}{m}\right), &\hbox{ if } m > 0 \hbox{ and } w \in K,
		\\
		\,\, 0, &\hbox{ if } m=0, \hbox{ and } w= 0,
		\\
		\,\, +\infty, &\hbox{ otherwise.}
	\end{cases}
      \end{equation}
      Furthermore, a discrete version of the linear constraint~\eqref{eq:variational-form-linearcons} can be written as
      \begin{equation}
        \label{eq:var:5}  
	\Sigma(M,W) = (0_\cU, \bar M^0)
      \end{equation}
      where $0_\cU \in \cU$ is the vector with $0$ on all coordinates, $\bar M^0 = (\bar m_0(x_0), \dots, \bar m_0(x_{N_h-1})) \in \RR^{N_h}$, see (\ref{eq:1}) for the definition of $\bar m_0$, and 
      \begin{equation}
\label{eq:var:6}  
	\Sigma: \cM \times \cW \to \cU \times \RR^{N_h},
	\qquad
	(M,W) \mapsto \Sigma(M,W) = (\Lambda(M,W), M^0),
      \end{equation}
      with $\Lambda(M,W) = AM + BW$, $A$ and $B$ being discrete versions of respectively the heat operator and the divergence operator, defined as follows:
\begin{equation}
\label{eq:def-op-A}
	A: \cM \to \cU,
	\qquad
	(AM)^{n}_i = \frac{M^{n+1}_i - M^{n}_i}{\Delta t} - \nu (\Delta_h M^{n+1})_i, \qquad 0 \le n < N_T, 0 \le i < N_h,
\end{equation}
and
\begin{equation}
\label{eq:def-op-B}
	B: \cW \to \cU,
	\qquad
	(BW)^{n}_i = \frac{W^{n}_{i+1,2} - W^{n}_{i,2}}{h} + \frac{W^{n}_{i,1} - W^{n}_{i-1,1}}{h} 
              \qquad 0 \le n \le N_T, 0 \le i < N_h.
\end{equation}
We define the transposed operators $A^*: \cU \to \cM$ and $B^*: \cU \to \cW$, i.e. such that
$$
\langle M, A^*U \rangle = \langle A M, U \rangle,
\qquad
\langle W, B^*U \rangle = \langle B W, U \rangle.
$$
This yields
\begin{equation}
  \label{eq:def-op-A*}
  (A^* U)^{n}_i = \left\{
    \begin{array}[c]{ll}
      \ds  -\frac{U^{0}_i }{\Delta t} , \qquad &n =0,  \; 0\le i < N_h,\\ \\
      \ds       \frac{U^{n-1}_i - U^{n}_i}{\Delta t} - \nu (\Delta_h U^{n-1})_i, \qquad &0 < n < N_T-1,  \; 0\le i < N_h,\\ \\
       \ds       \frac{U^{N_T-1}_i }{\Delta t} - \nu (\Delta_h U^{N_T-1})_i, \qquad & n = N_T-1, \; 0 \le i < N_h,
    \end{array}
\right.
\end{equation}
and
\begin{equation}
\label{eq:var:7}  
	(B^* U)^n_i
	=- [\nabla_h U^n]_i
	= -\left( \frac{U^n_{i+1} - U^n_{i}}{h}, \frac{U^n_{i} - U^n_{i-1}}{h} \right),
	\qquad
	0 \le n < N_T, 0 \le i < N_h.
      \end{equation}
      This implies that 
\begin{equation}
\label{eq:var:8}  
	\Im(B) = \Ker(B^*)^\bot = \left\{ U \in \cU \,:\, \forall \, 0 \le n < N_T, \sum_i U^n_i = 0 \right\}.
      \end{equation}
      
The discrete counterpart of the variational problem (\ref{eq:variational-form-J})-(\ref{eq:variational-form-linearcons}) is therefore
\begin{equation}
\label{eq:discrete-pb-short}
	\inf_{\substack{(M,W) \in \cM \times \cW, \\ \Sigma(M,W) = (0_\cU, \bar M^0)}} \Theta(M,W).
\end{equation}

The following results provide a constraint qualification property and the existence of a minimizer.  See e.g.~\cite[Theorem 3.1]{BricenoAriasetalCEMRACS2017} for more details on a special case (see also~\cite[Section 6]{MR3135339} and~\cite[Theorem 2.1]{MR3772008}  for similar results respectively in the context of the planning problem and in the context of an ergodic MFG).

\begin{proposition}(Constraint qualification)\label{sec:discr-vers-vari-1}
	There exists a pair $(\tilde M, \tilde W) \in \cM \times \cW$ such that
	$$
		\Theta(\tilde M, \tilde W) < \infty,
		\qquad
		\hbox{ and } 
		\qquad 
		\Sigma(\tilde M, \tilde W) = (0_\cU, \bar M^0).
	$$
\end{proposition}

\begin{theorem}\label{sec:discr-vers-vari-3}
There exists a minimizer $(M,W) \in \cM \times \cW$ of~\eqref{eq:discrete-pb-short}. Moreover, if $\nu>0$, $M^n_i >0$ for all $i \in \{0, \dots, N_h\}$ and all $n>0$.
\end{theorem}

\subsection{Recovery of the finite difference scheme  by convex duality}
\label{sec:recov-finite-diff}
We will now show that the finite difference scheme introduced in the previous section corresponds to the optimality conditions of the discrete minimization problem~\eqref{eq:discrete-pb-short} introduced above. To this end, we first rewrite the discrete problem by embedding the constraint in the cost functional as follows:
\begin{equation}
\label{eq:discrete-pb-primal-chi0}
	\min_{(M,W) \in \cM \times \cW}
	\left\{ \Theta(M,W) + \chi_0(\Sigma(M,W)) \right\},
\end{equation}
where 
$
	\chi_0: \cU \times \RR^{N_h} \to \RR \cup \{+\infty\}
$
is the characteristic function (in the convex analytical sense) of the singleton $\{(0_\cU, \bar M^0)\}$, i.e., $\chi_0(U, \varphi)$ vanishes if $(U, \varphi) = (0_\cU, \bar M^0)$ and otherwise it takes the value $+\infty$.

By Fenchel-Rockafellar duality theorem, see~\cite{MR1451876}, we obtain that
\begin{equation}
\label{eq:discrete-pb-duality}
	\min_{(M,W) \in \cM \times \cW}
	\left\{ \Theta(M,W) + \chi_0(\Sigma(M,W)) \right\}
	=
	- \min_{(U,\varphi) \in \cU \times \RR^{N_h}} 
	\left\{ \Theta^*(\Sigma^*(U,\varphi)) + \chi_0^*(-U, -\varphi) \right\},
\end{equation}
where $\Sigma^*$ is the adjoint operator of $\Sigma$, and $\chi_0^*$ and $\Theta^*$ are the convex conjugates of $\chi_0$ and $\Theta$ respectively, to wit,
\begin{equation}
  \label{eq:var:10}
  \Bigl\langle	\Sigma^*(U,\varphi) , (M,W) \Bigr\rangle=  \langle A^* U, M\rangle + \langle B^* U, W\rangle + \sum_{i=0}^{N_h-1} \varphi_i M^0_i 
\end{equation}
and
\begin{equation}
  \label{eq:var:11}
	\chi_0^*(-U, -\varphi) = - \sum_{i=0}^{N_h-1} \bar M^0_i \varphi_i .
      \end{equation}
Therefore
\begin{displaymath}
  \begin{split}
   & \Theta^* (\Sigma^*(U,\varphi)) \\
    =	&\max_{(M,W) \in \cM\times \cW }
    \Bigl(  \langle A^* U, M\rangle + \langle B^* U, W\rangle + \sum_{i=0}^{N_h-1} \varphi_i M^0_i
    -	\sum_{n=1}^{N_T} \sum_{i=0}^{N_h-1} \left[ \tilde L(x_i, M^n_i, W^{n-1}_i) + F(x_i, M^n_i) \right] - \frac{1}{\Delta t}  \sum_{i=0}^{N_h-1}  \Phi(x_i, M^{N_T}_i)    \Bigr)
  \end{split}
\end{displaymath}
We use the fact that
\begin{displaymath}
\max_{W\in \cW} \left(\langle B^* U, W\rangle  -	\sum_{n=1}^{N_T} \sum_{i=0}^{N_h-1} \tilde L(x_i, M^n_i, W^{n-1}_i)\right)=\left\{
  \begin{array}[c]{ll}
    \ds \sum_{n=1}^{N_T} \sum_{i=0}^{N_h-1} M_i^n \tilde H\left(x_i, -(B^*U)^{n-1}_i\right),\quad &\hbox{if } M^n \ge 0 \; \hbox{for all } n\ge 1,\\
    -\infty,\quad& \hbox{otherwise.}
  \end{array}
\right.
\end{displaymath}
We deduce the following:
\begin{equation}
  \label{eq:var:12}
    \begin{split}
   & \Theta^* (\Sigma^*(U,\varphi)) \\
    =& \max_{M\in \cM}   \Bigl(  \langle A^* U, M\rangle  - \sum_{i=0}^{N_h-1} \varphi_i M^0_i +
    \sum_{n=1}^{N_T} \sum_{i=0}^{N_h-1}  \left[M_i^n \tilde H\left(x_i, -(B^*U)^{n-1}_i\right)- F(x_i, M^n_i) \right] - \frac{1}{\Delta t}  \sum_{i=0}^{N_h-1}  \Phi(x_i, M^{N_T}_i)    \Bigr)\\
    =& \sum_{n=1}^{N_T-1} \sum_{i=0}^{N_h-1} F^* \Bigl( x_i,  (A^* U)^n_i +\tilde H\left (x_i ,[\nabla_h U^{n-1}]_i\right)  \Bigr)
     +  \sum_{i=0}^{N_h-1} \left(F+\frac 1 {\Delta t} \Phi \right)^*  \Bigl(x_i,   (A^* U)^{N_T}_i +\tilde H\left (x_i ,[\nabla_h U^{N_T-1}]_i\right)  \Bigr)
    \\& +\max_{ M\in \cM} M^0_i \left( (A^* U)^0 _i +\varphi_i \right) ,
  \end{split}
\end{equation}
where for every $x \in \TT$, $F^*(x, \cdot)$ and $\left(F+\frac 1 {\Delta t} \Phi\right)^*(x, \cdot)$ are the convex conjugate of $F(x, \cdot)$ and $F(x,\cdot)+\frac 1 {\Delta t}\Phi(x, \cdot)$ respectively.
Note that from the definition of $F$, it has not been  not necessary to impose the constraint  $ M^n \ge 0 \; \hbox{for all } n\ge 1$ in (\ref{eq:var:12}).
Note also that the last term in  (\ref{eq:var:12}) can be written $\max_{ M\in \cM} M^0_i \left( -\frac 1 {\Delta t} U^0 _i +\varphi_i \right) $.      
Using (\ref{eq:var:12}), in the dual minimization problem given in the right hand side of~\eqref{eq:discrete-pb-duality}, 
we see that the minimization with respect to $\varphi$ yields that $M_i^0=\tilde M_i^0$. Therefore, the 
the dual problem can be rewritten as
\begin{align}
	&- \min_{U \in \cU} 
	\Bigg\{ \sum_i \left(F + \frac{1}{\Delta t} \Phi\right)^* \left( x_i, \frac{U^{N_T-1}_i}{\Delta t} - \nu(\Delta_h U^{N_T-1})_i + \tilde H(x_i, [\nabla_h U^{N_T-1}]_i) \right)
	\notag
	\\
	&\qquad\qquad 
	+ \sum_{n=0}^{N_T-2} \sum_i F^*  \left( x_i, \frac{U^{n}_i - U^{n+1}_i}{\Delta t} - \nu(\Delta_h U^{n})_i + \tilde H(x_i, [\nabla_h U^{n}]_i) \right)
	- \frac{1}{\Delta t} \sum_i \bar M^0_i U^0_i 
	\Bigg\},
	\label{eq:discrete-pb-dual}
\end{align}
 which is the discrete version of
$$
	-\min_u \left\{ \int_0^T \int_\TT F^*\Big( x, -\partial_t u(t,x) - \nu \Delta u(t,x) + H_0(x, \nabla u(t,x)) \Big) dx dt + \int_\TT \Phi^* (x, u(T,x)) dx - \int_\TT m_0(x) u(0,x) dx \right\}.
$$

The arguments above lead to the following:
\begin{theorem}
	The solutions $(M,W) \in \cM \times \cW$ and $U \in \cU$ of the primal/dual problems are such that: 
	\begin{itemize}
		\item if $\nu>0$, $M^n_i>0$ for all $0\le i \le N_h-1$ and $0< n \le N_T$;  
		\item $W^n_i = M^{n+1}_i P_K([\nabla_h U^n]_i)$ for all $0\le i \le N_h-1$ and $0 \le n < N_T$;  
		\item $(U,M)$ satisfies the discrete MFG system~\eqref{eq:discrete-HJB}--\eqref{eq:discrete-KFP} obtained by the finite difference scheme.
	\end{itemize}
	Furthermore, the solution is unique if $f_0$ and $\phi$ are strictly increasing.
\end{theorem}
\begin{proof}
  The proof follows naturally for the arguments given above. See e.g.~\cite{MR2888257} in the case of planning problems, \cite[Theorem 3.1]{BricenoAriasetalCEMRACS2017}, or~\cite[Theorem 2.1]{MR3772008} in the stationary case.
\end{proof}

In the rest of this section, we describe two numerical methods to solve the discrete optimization problem.

\subsection{Alternating Direction Method of Multipliers}
\label{sec:ADMM}

Here we describe the Alternating Direction Method of Multipliers (ADMM) based on an Augmented Lagrangian approach for the dual problem. The interested reader is referred to e.g. the monograph~\cite{MR724072} by Fortin and Glowinski and the references therein for more details. This idea was made popular  by Benamou and Brenier  for optimal transport,  see\cite{MR1738163}, and  first used in the context of MFGs by Benamou and Carlier in~\cite{MR3395203}; see also~\cite{MR3731033} for an application to second order MFG using multilevel preconditioners and~\cite{MR3575615} for an application to mean field type control.

The dual problem~\eqref{eq:discrete-pb-dual} can be rewritten as
$$
	\min_{U \in \cU} \left\{ \Psi(\Lambda^* U) + \Gamma(U) \right\}
$$
where
$$
	\Lambda^*: \cU \to \cM \times \cW,
	\qquad
	\Lambda^* U = (A^* U, B^* U),
$$
$$
	\Psi(M, W) = \sum_i \left(F + \frac{1}{\Delta t} \Phi\right)^* \left(x_i, M^{N_T}_i + \tilde H(x_i, W^{N_T-1}_i) \right)
	+ \sum_{n=0}^{N_T-2} \sum_i F^*  \left( x_i, M^n_i + \tilde H(x_i, W^{n-1}_i) \right)
$$
and 
$$
	\Gamma(U) = - \frac{1}{\Delta t} \sum_i \bar M^0_i U^0_i .
$$
Note that $\Psi$ is convex and l.s.c., and $\Gamma$ is linear. Moreover, $\Psi$ depend in a separate manner on the pairs $(M^n_i ,W^{n-1}_i )_{n,i}$, and a similar observation can be made for $\Lambda$.

 We can artificially introduce an extra variable $Q \in \cM \times \cW$ playing the role of $\Lambda^* U$. Introducing a Lagrange multiplier $\sigma = (M,W)$ for the constraint $Q = \Lambda^* U$, we obtain the following equivalent problem:
\begin{equation}
\label{eq:pb-Lagrangian}
	\min_{U \in \cU, Q \in \cM \times \cW} \,\, \sup_{\sigma \in \cM \times \cW} \,\, \cL(U, Q, \sigma),
\end{equation}
where the Lagrangian $\cL: \cU \times \cM \times \cW \times \cM \times \cW \to \RR$ is defined as
$$
	\cL(U, Q, \sigma) = \Psi(Q) + \Gamma(U)  - \langle \sigma, Q - \Lambda^* U \rangle .
$$
Given a parameter $r>0$, it is convenient to  introduce an equivalent problem involving the augmented Lagrangian $\cL_r$
 obtained by adding to $\cL$  a quadratic penalty term in order to obtain strong convexity: 
$$
	\cL_r(U, Q, \sigma) = \cL(U, Q, \sigma) + \frac{r}{2} \| Q - \Lambda^* U \|_2^2.
$$
Since the saddle-points of $\cL$ and $\cL_r$ coincide, equivalently to~\eqref{eq:pb-Lagrangian}, we will seek to solve
$$
	\min_{U \in \cU, Q \in \cM \times \cW} \,\, \sup_{\sigma \in \cM \times \cW} \,\, \cL_r(U, Q, \sigma).
$$
In this context, the Alternating Direction Method of Multipliers (see e.g. the method called ALG2 in~\cite{MR724072}) is 
described in Algorithm~\ref{admm-algo} below. We use $\partial \Gamma$ and $\partial \Psi$ to denote the subgradients of $\Gamma$ and $\Psi$ respectively.

\begin{algorithm} 
\caption{Alternating Direction Method of Multipliers }\label{admm-algo}
\begin{algorithmic}
\Function {\texttt{ADMM}}{$U, Q, \sigma$}
	\State Initialize $(U^{(0)}, Q^{(0)}, \sigma^{(0)}) \leftarrow (U, Q, \sigma)$
	\For{$k=0, \dots, K-1$}
		\vspace{0.1cm}
		\State Find $U^{(k+1)} \in \argmin_{U \in \cU} \cL_r(U, Q^{(k)}, \sigma^{(k)})$; the first order optimality condition yields:
		$$
			-r \Lambda \left( \Lambda^* U^{(k+1)} - Q^{(k)} \right) - \Lambda \sigma^{(k)} \in \partial \Gamma\left( U^{(k+1)} \right).
		$$
		\State Find $Q^{(k+1)} \in \argmin_{Q \in \cM \times \cW} \cL_r \left( U^{(k+1)}, Q, \sigma^{(k+1)} \right)$; the first order optimality condition yields:
		$$
			\sigma^{(k)} + r \left( \Lambda^* U^{(k+1)} - Q^{(k+1)} \right) \in \partial \Psi\left( Q^{(k+1)} \right).
		$$
		\State Perform a gradient step: $\sigma^{(k+1)} = \sigma^{(k)} + r \left( \Lambda^* U^{(k+1)} - Q^{(k+1)} \right)$.
		\vspace{0.1cm}
	\EndFor
	\Return $(U^{(K)}, Q^{(K)}, \sigma^{(K)})$
\EndFunction
\end{algorithmic}
\end{algorithm}

\begin{remark}[preservation of the sign]
An important consequence of this method is that, thanks to the second and third steps above, the non-negativity of the discrete approximations to the  density is preserved. Indeed, $\sigma^{(k+1)} \in  \partial \Psi\left( Q^{(k+1)} \right)$ and $ \Psi\left( Q^{(k+1)} \right)< +\infty$, hence $\Psi^*\left( \sigma^{(k+1)} \right) = \langle \sigma^{(k+1)}, Q^{(k+1)} \rangle - \Psi\left( Q^{(k+1)} \right)$, which yields $\Psi^*\left( \sigma^{(k+1)} \right) < +\infty$. In particular, denoting $(M^{(k+1)},W^{(k+1)}) = \sigma^{(k+1)} \in \cM \times \cW$, this implies that we have $(M^{(k+1)})^n_i \ge 0 $ for every $0 \le i < N_h, 0 \le n \le N_T$, and $(W^{(k+1)})^n_i \in K$ vanishes if $(M^{(k+1)})^{n+1}_i = 0$  for every $0 \le i < N_h, 0 \le n < N_T$.
\end{remark}

Let us provide more details on the numerical solution of the first and second steps in the above method. For the first step, since $\Gamma$ only acts on $U^0$, the first order optimality condition amounts to solving the discrete version of a boundary value problem on $(0,T) \times \TT$ or more generally $(0,T) \times \TT^d$ when the state is in dimension $d$, with a degenerate fourth order linear elliptic operator (of order four in the state variable if $\nu$ is positive and two in time), and with a Dirichlet condition at $t=T$ and a Neumann like condition at $t=0$. If $d>1$, it is generally not possible to use direct solvers; instead, one has to use iterative methods such as the conjugate gradient algorithm. An important difficulty is that, since the equation is fourth order with respect to the state variable if $\nu>0$, the condition number grows like $h^{-4}$, and a careful preconditionning is mandatory (we will come back to this point in the next section). In the deterministic case, i.e. if $\nu=0$, the first step consists of solving the discrete version of a second order $d+1$ dimensional elliptic equation and preconditioning is also useful for reducing the computational complexity.

As for the second step, the first order optimality condition amounts to solving, at each space-time grid node $(i,n)$, a non-linear minimization problem in $\RR^{1+2d}$ of the form 
$$
	\min_{a \in \RR, b \in \RR^2} F^*(x_i, a + \tilde H(x_i, b) ) - \langle \sigma, (a,b) \rangle + \frac{r}{2} \| (a,b) - (\bar a, \bar b)\|_2^2,
$$
where $(a, b)$ plays the role of $(M^n_i, W^{n-1}_i)$, whereas $(\bar a, \bar b)$ plays the role of $(\Lambda^* U)^n_i$.

Note that the quadratic term, which comes from the fact that we consider the augmented Lagrangian $\cL_r$ instead of the original Lagrangian $\cL$, provides coercivity and strict convexity. The two main difficulties in this step are the following. First, in general, $F^*$ does not admit an explicit formula and is itself obtained by solving a maximization problem. Second, for general $F$ and $\tilde H$, computing the minimizer explicitly may be impossible. In the latter case, Newton iterations may be used, see e.g.~\cite{MR3575615}. 

The following result ensures convergence of the numerical method; see~\cite[Theorem 8]{MR1168183} for more details.

\begin{theorem}[Eckstein and Bertsekas~\cite{MR1168183}]
	Assume that $r>0$, that $\Lambda \Lambda^*$ is symmetric and positive definite and that there exists a solution to the primal-dual extremality . 
Then the sequence $(U^{(k)}, Q^{(k)}, \sigma^{(k)})$ converges as $k\to \infty$ and
	$$
		\lim_{k \to +\infty} (U^{(k)}, Q^{(k)}, \sigma^{(k)}) = (\bar U, \Lambda^* \bar U, \bar \sigma),
	$$
where $\bar U$ solves the dual problem and $\bar \sigma$ solves the primal problem.
\end{theorem}

\subsection{Chambolle and Pock's algorithm}
\label{sec:CPalgo}

We now consider a primal-dual algorithm introduced by Chambolle and Pock in~\cite{MR2782122}. The application to stationary MFGs was first investigated by Brice{\~n}o-Arias, Kalise and Silva in~\cite{MR3772008}; see also~\cite{BricenoAriasetalCEMRACS2017} for an extension to the dynamic setting.

Introducing the notation $\Pi = \chi_0 \circ \Sigma : \cM \times \cW \ni (M,W) \mapsto \chi_0(\Sigma(M,W)) \in \RR \cup \{+\infty\}$, the primal problem~\eqref{eq:discrete-pb-primal-chi0} can be written as
$$
	\min_{\sigma \in \cM \times \cW}
	\left\{ \Theta(\sigma) + \Pi(\sigma) \right\},
$$
and the dual problem can be written as
$$
	\min_{Q \in \cM \times \cW}
	\left\{ \Theta^*(-Q) + \Pi^*(Q) \right\}.
$$
For $r,s >0$, the first order optimality conditions at $\hat\sigma$ and $\hat Q$ can be equivalently written as
$$
	\begin{cases}
	\,\, - \hat Q &\in \partial \Theta(\hat \sigma) 
	\\
	\,\, \hat \sigma &\in \partial \Pi^*(\hat Q)
	\end{cases}
	\quad
	\Leftrightarrow
	\quad
	\begin{cases}
	\,\, \hat \sigma - r \hat Q &\in \hat\sigma + r \partial \Theta(\hat \sigma) 
	\\
	\,\, \hat Q + s \hat \sigma &\in \hat Q + s \partial \Pi^*(\hat Q)
	\end{cases}
	\quad
	\Leftrightarrow
	\quad
	\begin{cases}
	\,\, \hat \sigma &\in \argmin_\sigma \left\{ \Theta(\sigma)  + \frac{1}{2r} \| \sigma - (\hat \sigma - r \hat Q) \|^2 \right\}
	\\
	\,\, \hat Q &\in \argmin_Q \left\{ \Pi^*(Q) + \frac{1}{2s} \| Q - (\hat Q + s \hat \sigma )\|^2 \right\}.
	\end{cases}
$$

Given some parameters $r>0, s>0, \tau \in [0,1]$, the Chambolle-Pock method is described in Algorithm~\ref{chambolle-pock-algo}.  The algorithm has been proved to converge if $rs < 1$.

\begin{algorithm} 
\caption{Chambolle-Pock Method}\label{chambolle-pock-algo}
\begin{algorithmic}
\Function {\texttt{ChambollePock}}{$\sigma, \tilde \sigma, Q$}
	\State Initialize $(\sigma^{(0)}, \tilde \sigma^{(0)}, Q^{(0)}) \leftarrow (\sigma, \tilde \sigma, Q)$
	\For{$k=0, \dots, K-1$}
		\vspace{0.1cm}
		\State Find
		\begin{equation}
		\label{eq:CPalgo-step1}
			Q^{(k+1)} \in \argmin_Q \left\{ \Pi^*(Q) + \frac{1}{2s} \| Q - (Q^{(k)} + s \tilde \sigma^{(k)} )\|^2 \right\}
		\end{equation}
		\State Find
		\begin{equation}
		\label{eq:CPalgo-step2}
			\sigma^{(k+1)} \in \argmin_\sigma \left\{ \Theta(\sigma)  + \frac{1}{2r} \| \sigma - (\sigma^{(k)} - r Q^{(k+1)}) \|^2 \right\}
		\end{equation}
		\State Set
		$$
			\tilde \sigma^{(k+1)} = \sigma^{(k+1)} + \tau \left( \sigma^{(k+1)} - \sigma^{(k)} \right).
		$$ 
		\vspace{0.1cm}
	\EndFor
	\Return $(\sigma^{(K)}, \tilde \sigma^{(K)}, Q^{(K)})$
\EndFunction
\end{algorithmic}
\end{algorithm}

Note that the first step is similar to the first step in the ADMM method described in subsection~\ref{sec:ADMM} and amounts to solving a linear fourth order PDE. The second step is easier than in ADMM because $\Theta$ admits an explicit formula.

\section{Multigrid preconditioners}
\label{sec:mult-prec}

\subsection{General considerations on multigrid methods}
\label{sec:gener-cons-mult}

Before explaining how multigrid methods can be applied in the context of numerical solutions for MFG, let us recall the main ideas. We refer to~\cite{MR1807961} for an introduction to multigrid methods and more details. In order to solve a linear system which corresponds to the discretisation of an equation on a given grid, we can use coarser grids in order to get approximate solutions. Intuitively, a multigrid scheme should be efficient because solving the system on a coarser grid is computationally easier and the solution on this coarser grid should provide a good approximation of the solution on the original grid. Indeed, using a coarser grid should help capturing quickly the low frequency modes (i.e., the modes corresponding to the smallest eigenvalues of the differential operator from which the linear system comes), which takes more iterations on a finer grid. 

When the linear system stems from a well behaved second order elliptic operator for example, one can find simple iterative methods (e.g. Jacobi or Gauss-Seidel algorithms) such that a few iterations of these methods are enough to damp the higher frequency components of the residual, i.e. to make the error smooth. Typically, these iterative methods have bad convergence properties, but they have good smoothing properties and are hence called \textit{smoothers}. The produced residual is smooth on the grid under consideration (i.e., it has small fast Fourier components on the given grid), so it is well represented on the next coarser grid. This suggests to transfer  the residual to the next coarser grid (in doing so, half of the low frequency components on the finer grid become high frequency components on the coarser one, so they will be damped by the smoother on the coarser grid). These principles are the basis for a recursive algorithm. Note that in such an algorithm, using a direct method for solving systems of linear equations is required only on the coarsest grid, which contains much fewer nodes than the initial grid. On the grids of intermediate sizes, one only needs to perform matrix multiplications.

To be more precise, let us consider a sequence of nested grids $(\cG_\ell)_{\ell = 0,\dots,L}$, i.e. such that $\cG_\ell \subseteq \cG_{\ell+1}, \ell=0,\dots,L-1$. Denote the corresponding number of points by $\tilde N_\ell = \tilde N 2^{d\ell}$, $\ell = 0,\dots,L$, where $\tilde N$ is a positive integer representing the number of points in the coarsest grid. Assume that the linear system to be solved is
\begin{equation}
\label{eq:sys-multigrid-L}
	M_L x_L = b_L
\end{equation}
where the unknown is $x_L \in \RR^{\tilde N_L}$ and with $b_L \in \RR^{\tilde N_L}$, $M_L \in \RR^{\tilde N_L \times \tilde N_L}$. In order to perform intergrid communications, we introduce 
\begin{itemize}
	\item Prolongation operators, which represent a grid function on the next finer grid: $P_\ell^{\ell+1}: \cG_\ell \to \cG_{\ell+1}$.
	\item Restriction operators, which interpolate a grid function on the next coarser grid: $R_\ell^{\ell-1}: \cG_\ell \to \cG_{\ell-1}$.
\end{itemize}
Using these operators, we can define on each grid $\cG_\ell$ a matrix corresponding to an approximate version of the linear system to be solved: 
$$
	M_{\ell} = R_{\ell+1}^{\ell} M_{\ell+1} P_\ell^{\ell+1}.
$$
Then, in order to solve $M_\ell x_\ell = b_\ell$, the method is decomposed into three main steps. First, a pre-smoothing step is performed: starting from an initial guess $\tilde x_\ell^{(0)}$, a few smoothing iterations, say $\eta_1$, i.e. Jacobi or Gauss-Seidel iterations  for example. This produces an estimate $\tilde x_\ell = \tilde x_\ell^{(\eta_1)}$. Second, an (approximate) solution $x_{\ell-1}$ on the next coarser grid is computed for the equation $M_{\ell-1} x_{\ell-1} = R_\ell^{\ell-1}(b_\ell - M_\ell \tilde x_\ell)$. This is performed either by calling recursively the same function, or by means of a direct solver (using Gaussian elimination) if it is on the coarsest grid.  Third, a post-smoothing step is performed: $\tilde x_\ell + P_{\ell-1}^{\ell} x_{\ell-1}$ is used as an initial guess, from which $\eta_2$ iterations of the smoother are applied, for the problem with right-hand side $b_\ell$. To understand the rationale behind this method, it is important to note that
\begin{align*}
	R_{\ell}^{\ell-1} M_\ell \left(\tilde x_\ell + P_{\ell-1}^{\ell} x_{\ell-1}\right)
	&= R_{\ell}^{\ell-1} M_\ell \tilde x_\ell + M_{\ell-1} x_{\ell-1}
	\\
	&\approx R_{\ell}^{\ell-1} M_\ell \tilde x_\ell + R_\ell^{\ell-1}(b_\ell - M_\ell \tilde x_\ell)
	\\
	&= R_{\ell}^{\ell-1} b_\ell.
\end{align*}
In words, the initial guess (namely, $\tilde x_\ell + P_{\ell-1}^{\ell} x_{\ell-1}$) for the third step above is a good candidate for a solution to the equation $M_\ell x_\ell = b_\ell$, at least on the coarser grid $\cG_{\ell-1}$.

Algorithm~\ref{multigrid-vcycle} provides a pseudo-code for the method described above. Here, $S_\ell(x, b, \eta)$ can be implemented by performing $\eta$ steps of \textsc{Gauss-Seidel} algorithm starting with $x$ and with $b$ as right-hand side. The method as presented uses once the multigrid scheme on the coarser grid, which is called a V-cycle. Other approaches are possible, such as W-cycle (in which the multigrid scheme is called twice) or F-cycle (which is intermediate between the V-cycle and the W-cycle).  See e.g.~\cite{MR1807961} for more details. 

\begin{algorithm} 
\caption{Multigrid Method for $M_L x_L = b_L$ with V-cycle}\label{multigrid-vcycle}
\begin{algorithmic}
\Function {\texttt{MultigridSolver}}{$\ell, \tilde x_\ell^{(0)}, b_\ell$}
\If{$\ell = 0$}
\State $x_0 \leftarrow M_0^{-1}b_0$ \hfill \textit{//  exact solver at level $0$}
\Else
\State $\tilde x_\ell \leftarrow S_\ell\Big(\tilde x_\ell^{(0)}, b_\ell, \eta_1\Big)$  \hfill \textit{//  pre-smoothing with $\eta_1$ steps of smoothing}
\State $\tilde x_{\ell-1}^{(0)} \leftarrow 0$
\State $x_{\ell-1} \leftarrow $ \texttt{MultigridSolver}$\Big(\ell-1, \tilde x_{\ell-1}^{(0)}, R_\ell^{\ell-1}(b_\ell - M_\ell \tilde x_\ell) \Big)$ \hfill \textit{//  coarser grid correction}
\State $x_\ell \leftarrow S_\ell\Big(x_\ell + P_{\ell-1}^{\ell} x_{\ell-1}, b_\ell, \eta_2\Big)$ \hfill \textit{//  post-smoothing with $\eta_2$ steps of smoothing}
\EndIf
\Return $x_\ell$
\EndFunction
\end{algorithmic}
\end{algorithm}

\subsection{Applications in the context of mean field games}
\label{sec:appl-cont-mfg}

Multigrid methods can be used for a linearized version of the MFG PDE system, see~\cite{MR2928376}, or as a key ingredient of the ADDM or the primal-dual algorithms, see~\cite{MR3731033,MR3772008}. In the latter case, it corresponds to taking $M_L = \Lambda \Lambda^*$ in~\eqref{eq:sys-multigrid-L}. A straightforward application of the multigrid scheme described above leads to coarsening the space-time grid which does not distinguish between the space and time dimensions. This is called {\sl full coarsening}. However, in the context of second-order MFG, this approach leads to poor performance. See e.g.~\cite{MR2928376}.  We reproduce here one table contained in~\cite{MR2928376}: the multigrid method is used as a preconditioner in a preconditioned \texttt{BiCGStab} iterative method, see~\cite{MR1149111}, in order to solve a linearized version of the MFG system of PDEs. 
 In Table \ref{tab:3}, we display the number of iterations of the preconditioned \texttt{BiCGStab}  method: we see that the number of iterations grows significantly as the number of nodes is increased.
 \begin{table}[H]
  \centering
  \caption{ Full coarsening multigrid preconditioner with 4 levels and several values of the viscosity $\nu$:  average  number of preconditioned \texttt{BiCGStab} iterations to decrease the residual by a factor $0.01$}
 {\small  \begin{tabular}[h]{|c||c|c|c|}
\hline
  $\nu\backslash$ grid& $32\times 32\times 32$  &  $64\times 64\times 64$ & $128\times 128\times 64$
 \\ \hline   \hline
 0.6 &  40 & 92  & \hbox{fail}
\\
\hline
 0.36 & 24  &61  & \hbox{fail}
\\
\hline
 0.2 & 21 &  45 &\hbox{fail}
 \\
 \hline
  \hline
\end{tabular}\label{tab:3}
}
\end{table}

The reason for this poor behavior can be explained by the fact that the usual smoothers actually make the error smooth in the hyperplanes $t = n \Delta t$, i.e. with respect to the variable $x$, but not with respect to the variable $t$. Indeed, the unknowns are more strongly coupled in the hyperplanes $\{(t,x) : t = n \Delta t\}, n = 0, \dots, N_T$ (fourth order operator w.r.t. $x$) than on the lines $\{(t,x) : x = i \, h\}, i=0,\dots,N_h$ (second order operator w.r.t. $t$). This leads to the idea of using semi-coarsening: the hierarchy of nested grids should be obtained by coarsening the grids in the $x$ directions only, but not in the $t$ direction. We refer the reader to~\cite{MR1807961} for  semi-coarsening multigrid methods in the context of  anisotropic operators.

In the context of the primal-dual method discussed in subsection~\ref{sec:CPalgo}, the first step~\eqref{eq:CPalgo-step1} amounts to solving a discrete version of the PDE with operator $-\partial_{tt}^2 + \nu^2 \Delta^2- \Delta$ where $\Delta^2$ denotes the bi-Laplacian operator. In other words, one needs to solve a system of the form~\eqref{eq:sys-multigrid-L} where $M_L$ corresponds to the discretization of this operator on the (finest) grid under consideration. One can thus use one cycle of the multigrid algorithm, which is a linear operator as a function of the residual on the finest grid, as a preconditioner for solving \eqref{eq:sys-multigrid-L} with the \texttt{BiCGStab} method. 

We now give details on the restriction and prolongation operators when $d=1$.
Using the notations introduced above, we consider that $N_h$ is of the form $N_h = n_0 2^{L}$  for some integer $n_0$. Remember that $\Delta t = T / N_T$ and $h = 1 / N_h$ since we are using the one-dimensional torus as the spatial domain. The number of points on the coarsest grid $\cG_0$ is $\tilde N = (N_T+1) \times n_0$ while on the $\ell$-th grid $\cG_\ell$, it is $\tilde N_\ell =  (N_T+1) \times n_0 \times 2^{\ell}$.

For the restriction operator  $R_\ell^{\ell-1}: \cG_\ell \to \cG_{\ell-1}$, following~\cite{MR2928376}, we can use the second-order full-weighting operator defined by
$$
(R_\ell^{\ell-1} x)^n_{i} := \frac1{4}
\left(
\begin{aligned}
2 X^n_{2i} + X^n_{2i+1}+X^n_{2i-1}
\end{aligned}
\right),
$$
for $n=0,\hdots, N_{T}$, $i=1, \hdots, 2^{\ell-1} n_0$.

As for the prolongation operator $P_\ell^{\ell+1}: \cG_\ell \to \cG_{\ell+1}$, one can take the standard linear interpolation which is second order accurate. An important aspect in the analysis of multigrid methods is that the sum of the accuracy orders of the two intergrid transfer operators should be not smaller than the order of the partial differential operator. Here, both are $4$. In this case, multigrid theory states that convergence holds even with a single smoothing step, i.e. it suffices to take $\eta_1, \eta_2$ such that $\eta_1 + \eta_2 = 1$.

\subsection{Numerical illustration}
\label{sec:numer-illustr}

In this paragraph, we borrow a numerical example from~\cite{BricenoAriasetalCEMRACS2017}. We assume that $d=2$ and that given $q >1$, with conjugate exponent denoted by $q'= q/(q-1)$, the Hamiltonian $H_0: \TT^2 \times \RR^2 \to \RR$ has the form 
$$
	H_0(x,p)= \frac{1}{q'} |p|^{q'}, \hspace{0.3cm} \forall \; x\in \TT^2, \; p \in \RR^2. 
$$
In this case the function $L$ defined in~\eqref{eq:def-L} takes the form
$$
	L(x,m,w) = 
	\begin{cases}
	 \frac{|w|^q}{qm^{q-1}},	&\text{ if } m>0,
	 \\
	 0, &\text{ if } (m,w)=(0,0), 
	 \\
	 +\infty, &\text{ otherwise.}
	\end{cases}
$$
Furthermore, recalling the notation~\eqref{eq:def-F-variational}--\eqref{eq:def-Phi-variational}, we consider $\phi \equiv 0$ and 
$$
	f(x, m) = m^2 - \overline H(x), \qquad \overline H(x) = \sin(2 \pi x_2) + \sin(2 \pi x_1) + \cos(2 \pi x_1),
$$
for all  $x = (x_1, x_2) \in \TT^2$ and $m \in \RR_+$. This means that in the underlying differential game, a typical agent aims to get closer to the maxima of $\bar{H}$ and, at the same time, she is adverse to crowded regions (because of the presence of the $m^2$ term in $f$).

Figure~\ref{fig:evolm} shows the evolution of the mass at four different time steps. Starting from a constant initial density, the mass converges to a steady state, and then, when $t$ gets close to the final time $T$, the mass is influenced by the final cost and converges to a final state. This behavior is referred to as \emph{turnpike phenomenon} in the literature on optimal control~\cite{MR3124890}. Theoretical results on the long time behavior of MFGs can be found in~\cite{MR2928380,MR3103242}. It is illustrated by Figure~\ref{fig:dist-statio}, which displays as a function of time $t$ the distance of the mass at time $t$ to the stationary state computed as in~\cite{MR3772008}. 
In other words, denoting by $M^\infty \in \mathbb{R}^{N_h\times N_h}$ the solution to the discrete stationary problem and by $M \in \cM$ the solution to the discrete evolutive problem, Figure~\ref{fig:dist-statio} displays the graph of 
$n \mapsto \|M^\infty - M^n\|_{\ell_2} = \left(h^2 \sum_{i,j} (M^\infty_{i,j} - M^n_{i,j})^2 \right)^{1/2}$, $n \in \{0,\dots,N_T\}$.

\begin{figure}	
	\centering
	\begin{subfigure}[t]{0.45\linewidth}
		\centering
		\includegraphics[width=\linewidth]{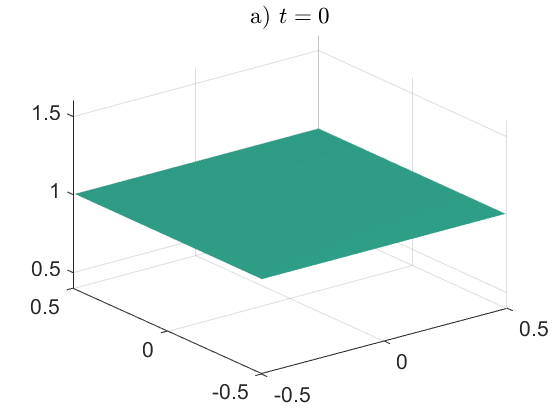}
	\end{subfigure}
	\quad
	\begin{subfigure}[t]{0.45\linewidth}
		\centering
	  	\includegraphics[width=\linewidth]{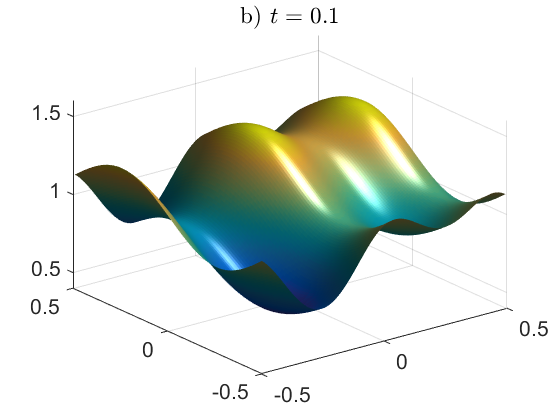}
	\end{subfigure}
	\vspace{1em}
	
	\begin{subfigure}[t]{0.45\linewidth}
		\centering
	  	\includegraphics[width=\linewidth]{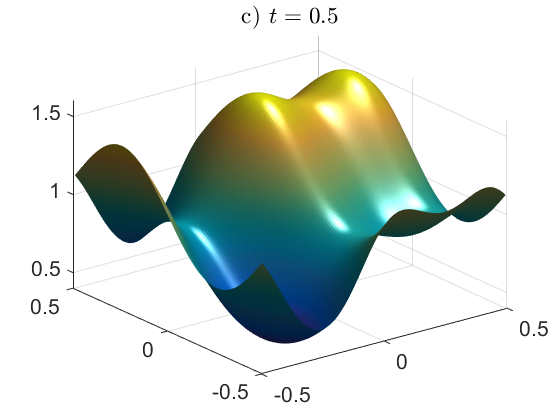}	
	\end{subfigure}
	\quad
	\begin{subfigure}[t]{0.45\linewidth}
		\centering
	  	\includegraphics[width=\linewidth]{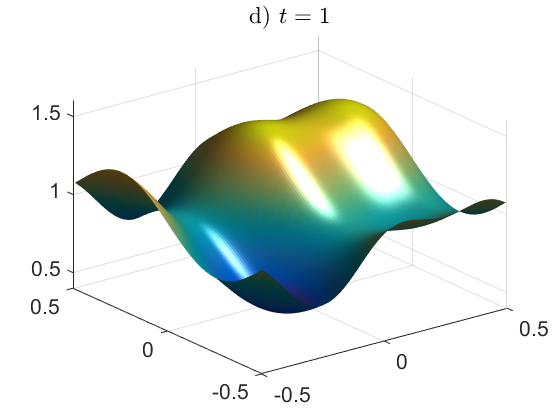}
	\end{subfigure}
	\caption{\label{fig:evolm} Evolution of the density $m$ obtained with the multi-grid preconditioner for $\nu = 0.5, T=1, N_T = 200$ and $N_h = 128$. At $t=0.12$ the solution is close to the solution of the associated  stationary MFG.}
\end{figure}

\begin{figure}	
		\centering
		\includegraphics[width=0.6\linewidth]{{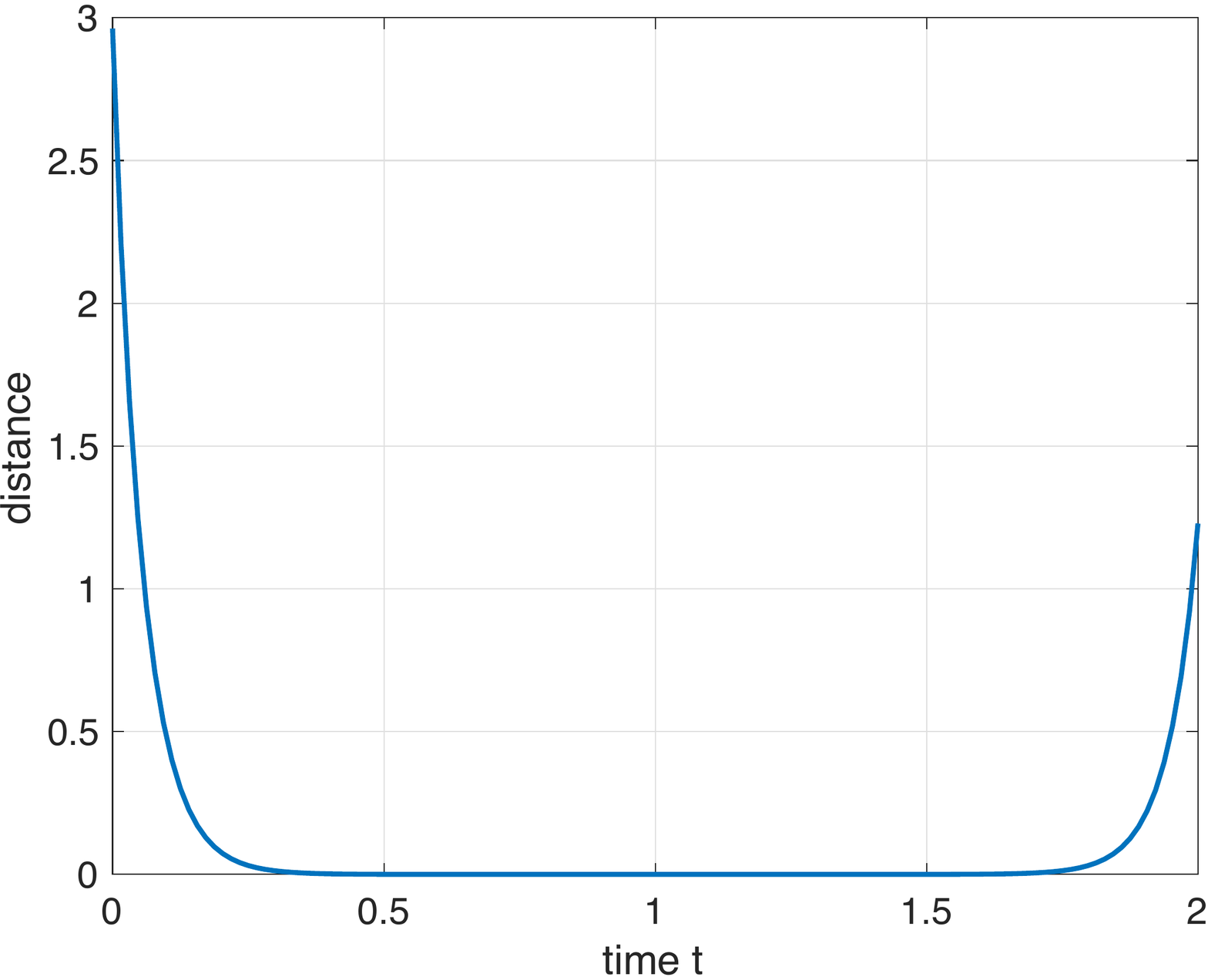}}	
	\caption{\label{fig:dist-statio} Distance to the stationary solution at each time $t \in [0,T]$, for $\nu = 0.5, T=2, N_T = 200$ and $N_h =128$. The distance is computed using the $\ell^2$ norm  as explained in the text. The turnpike phenomenon is observed as for a long time frame the time-dependent mass approaches the solution of the stationary MFG.}
\end{figure}

For the multigrid preconditioner, Table \ref{tab:MG_cvg_time} shows the computation times for different discretizations (i.e. different values of $N_h$ and $N_T$ in the coarsest grid). It has been observed in~\cite{MR3772008,BricenoAriasetalCEMRACS2017} that finer meshes with $128^3$ degrees of freedom are solvable within CPU times which outperfom several other methods such as Conjugate Gradient or \texttt{BiCGStab} unpreconditioned or preconditioned with modified incomplete Cholesky factorization. Furthermore, the method is robust with respect to different viscosities. 

From Table \ref{tab:MG_cvg_time} we observe that most of the computational time is used for solving~\eqref{eq:CPalgo-step2}, which does not use a multigrid strategy but which is a pointwise operator (see~\cite[Proposition 3.1]{MR3772008}) and thus could be fully parallelizable.

\begin{table}[H]
    \begin{minipage}{.5\linewidth}
      \caption*{(a) Grid with $64 \times 64 \times 64$ points.}
      \centering
        \begin{tabular}{| c | c | c | c |}
            \hline
            $\nu$ & Total time & Time first prox & Iterations \\
            \hline
            $0.6$ & $116.3$ [s] & $11.50$ [s] & $20$ \\
            \hline
            $0.36$ & $120.4$ [s] & $11.40$ [s] & $21$  \\
            \hline
            $0.2$  & $119.0$ [s]  & $11.26$ [s] & $22$ \\
            \hline
            $0.12$  &  $129.1$ [s] & $14.11$ [s] & $22$  \\
            \hline
            $0.046$ &  $225.0$ [s] & $23.28$ [s] & $39$  \\
            \hline
        \end{tabular}
    \end{minipage}%
    \begin{minipage}{.5\linewidth}
      \centering
        \caption*{(b) Grid with $128 \times 128 \times 128$ points.}
        \begin{tabular}{| c | c | c | c |}
            \hline
            $\nu$ & Total time & Time first prox & Iterations \\
            \hline
            $0.6$ &  $921.1$ [s] & $107.2$ [s] &  $20$ \\
            \hline
            $0.36$ &  $952.3$ [s] &  $118.0$ [s] & $21$ \\
            \hline
            $0.2$  &  $1028.8$ [s] &  $127.6$ [s] & $22$ \\
            \hline
            $0.12$  & $1036.4$ [s]  &  $135.5$ [s] & $23$ \\
            \hline
            $0.046$ &  $1982.2$ [s] & $260.0$ [s] & $42$ \\
            \hline
        \end{tabular}
    \end{minipage} 
  \vskip 2mm  	
    \caption{Time (in seconds) for the convergence of the Chambolle-Pock algorithm, cumulative time of the first proximal operator with the multigrid preconditioner, and number of iterations in the Chambolle-Pock algorithm, for different viscosity values $\nu$ and two types of grids. Here we used $\eta_1 = \eta_2 = 2$ in the multigrid methods, $T=1$.
      Instead of using a number of iterations $K$ fixed a priori, the iterations  have been stopped when the quantity $\|M^{(k+1)} - M^{(k)}\|_{\ell_2} = \left(\Delta t h^2  \sum_{n=0}^{N_T} \sum_{i,j} (M^{(k+1),n}_{i,j} - M^{(k),n}_{i,j})^2 \right)^{1/2}$ became smaller than $10^{-6}$, where  $M^{(k)}$ denotes the approximation of $M$ at iteration $k$, which is given by the first component of $\sigma^{(k)}$ in the notation used in Algorithm~\ref{chambolle-pock-algo}.}\label{tab:MG_cvg_time}
\end{table}

Table~\ref{tab:MG_bicg_ite} shows that the method is robust with respect to the viscosity since the average number of iterations of \texttt{BiCGStab} does not increase  much as the viscosity decreases. For instance, as shown in Table~\ref{tab:MG_bicg_ite}(b) for a grid of size $128 \times 128 \times 128$, the average number of iterations increases from $3.38$ to $4.67$ when $\nu$ is decreased from $0.6$ to $0.046$. On the other hand, Table~\ref{tab:MG_cvg_time} shows that the average number of  Chambolle-Pock iterations 
depends on the viscosity parameter, but this is not related to the use of the multigrid preconditioner.

\begin{table}[H]
    \begin{minipage}{.5\linewidth}
      \caption*{(a) Iterations to decrease the residual by a factor $10^{-3}$.}
      \centering
        \begin{tabular}{| c | c | c | c |}
            \hline
            $\nu$ & {\tiny $32\times32\times32$} & {\tiny $64\times64\times64$} & {\tiny $128\times128\times128$}\\
            \hline
            $0.6$ & $1.65$ & $1.86$ &  $2.33$\\
            \hline
            $0.36$ & $1.62$ & $1.90$ &  $2.43$ \\
            \hline
            $0.2$  & $1.68$ & $1.93$ & $2.59$ \\
            \hline
            $0.12$ & $1.84$ & $2.25$ &  $2.65$ \\
            \hline
            $0.046$ & $1.68$ & $2.05$ &  $2.63$ \\
            \hline
        \end{tabular}
    \end{minipage}%
    \begin{minipage}{.5\linewidth}
      \centering
        \caption*{(b) Iterations to solve the system with an error of $10^{-8}$.}
        \begin{tabular}{| c | c | c | c |}
            \hline
            $\nu$ & {\tiny $32\times32\times32$} & {\tiny $64\times64\times64$} & {\tiny $128\times128\times128$}\\
            \hline
            $0.6$ & $3.33$ & $3.40$ &  $3.38$\\
            \hline
            $0.36$ & $3.10$ & $3.21$ &  $3.83$ \\
            \hline
            $0.2$  & $3.07$ & $3.31$ & $4.20$ \\
            \hline
            $0.12$ & $3.25$ & $3.73$ &  $4.64$ \\
            \hline
            $0.046$ & $2.88$ & $3.59$ &  $4.67$ \\
            \hline
        \end{tabular}
    \end{minipage} 
    \caption{Average number of iterations of the preconditioned \texttt{BiCGStab} with $\eta_1 = \eta_2 = 2, T=1$ }   \label{tab:MG_bicg_ite}
\end{table}

\section{Algorithms for solving the system of non-linear equations}
\label{sec:other-algor-solv}

\subsection{Combining continuation methods with Newton iterations}
\label{sec:newton-algorithms}

Following previous works of the first author, see e.g.~\cite{MR3135339}, one may  use a continuation method (for example with respect to the viscosity parameter $\nu$) 
in which every system of nonlinear equations (given the parameter of the continuation method) is solved by means of Newton iterations.
With  Newton algorithm, it is important to have a good initial guess of the solution; for that,
 it is possible to take  advantage of the continuation method by  choosing the initial guess as the solution obtained with the previous value of  the parameter. 
Alternatively,  the initial guess can be obtained from the simulation of the same problem on a coarser grid, using interpolation. 
It is also important  to implement the Newton algorithm on a ``well conditioned''  system. Consider for example the system  (\ref{eq:discrete-HJB})-~(\ref{eq:discrete-KFP}): 
in this case, it proves more convenient to introduce  auxiliary unknowns, namely $\Bigl((f_i^n,)_{i,n},  (\Phi)_i\Bigr)$, 
 and see (\ref{eq:discrete-HJB})-(\ref{eq:discrete-KFP}) as a fixed point problem for  the map $\Xi:  \Bigl((f_i^n,)_{i,n},  (\Phi)_i\Bigr)\mapsto 
\Bigl((g_i^n,)_{i,n},  (\Psi)_i\Bigr) $ defined as follows: one solves first the discrete Bellman equation with data $\Bigl((f_i^n,)_{i,n},  (\Phi)_i\Bigr)$:
\begin{displaymath}
  \begin{array}[c]{rcll}
    - (D_t U_{i})^{n} - \nu (\Delta_h U^{n})_{i}
	+ \tilde H(x_i, [\grad_h U^n]_i) 
	&=&  f_i^{n} \quad  &i \in \{0,\dots,N_h\} \, , \,  n \in \{0, \dots, N_T-1\} \, ,\\
U^{n}_{0} &=& U^{n}_{N_h} \, ,\quad &n \in \{0, \dots,  N_T-1\} \, ,\\
U^{N_T}_{i} &=&\Phi_i & i \in \{ 0, \dots , N_{h} \} \,
\end{array}
\end{displaymath}
then the discrete Fokker-Planck equation
\begin{displaymath}
  \begin{array}[c]{rcll}
    (D_t M_{i})^{n} - \nu (\Delta_h M^{n+1})_{i}
    - \cT_i(U^{n}, M^{n+1})	&=& 0 \, , 
    \quad &i \in \{0,\dots,N_h\}, n \in \{0, \dots, N_T-1\} \, ,
    \\
    M^{n}_{0} &=& M^{n}_{N_h} \, , \quad &n \in \{ 1, \dots, N_T\} \, , 
    \\
    M^{0}_{i} &=& \bar m_0(x_i) \, , \, \quad & i \in \{0, \dots, N_h\} \, ,
  \end{array}
\end{displaymath}
and finally sets
\begin{displaymath}
  \begin{array}[c]{rcll}
    g_i^n &=& f_0(x_i,M^{n+1}_i),\quad \quad &i \in \{0,\dots,N_h\} \, , \,  n \in \{0, \dots, N_T-1\} \, ,\\
    \Psi_i&=& \phi(M^{N_T}_i) \, , \quad \quad &i \in \{0,\dots,N_h\}.
  \end{array}
\end{displaymath}
The Newton iterations are applied to the fixed point problem  $(I_d- \Xi) \Bigl((f_i^n,)_{i,n},  (\Phi)_i\Bigr)=0$. Solving the discrete Fokker-Planck guarantees that at each Newton iteration, the grid functions $M^n$ are non negative and have the same mass, which is not be the case if the Newton iterations are applied directly to (\ref{eq:discrete-HJB})-(\ref{eq:discrete-KFP}). Note also that Newton iterations consist of solving systems of linear equations involving the Jacobian of $\Xi$: 
for that, we use a nested iterative method (BiCGStab for example), which only requires a function that returns the matrix-vector product by the Jacobian, and not the construction of the Jacobian, which is a huge and dense  matrix.

The strategy consisting in  combining the continuation method with Newton iterations has the advantage to be very general: it requires neither a variational structure
 nor monotonicity: it has been successfully applied  in many simulations, for example for MFGs including congestion models, MFGs with two populations, 
see paragraph~\ref{sec:an-example-with} or MFGs  in which the coupling is through the control.  It is generally much faster than the methods presented in paragraph \ref{sec:varMFG}, when the latter can be used.

On the other hand, to the best of our knowledge, there is no general proof of convergence. Having these methods work efficiently is part of the know-how of the numerical analyst.  

Since  the strategy has been described in several articles of the first author, and since it is discussed thoroughly  in the paragraph \S~\ref{sec:an-example-with} devoted to  pedestrian flows, we shall not give any further detail here.

\subsection{A recursive algorithm based on elementary solvers on small time intervals}
\label{sec:recurs-algor-based}

In this paragraph, we consider a recursive method introduced by Chassagneux, Crisan and Delarue in~\cite{MR3914553} and further studied in~\cite{Angiulietal-2019}. It is based on the following idea. When the time horizon is small enough, mild assumptions allow one to apply Banach fixed point theorem and give a constructive proof of  existence and uniqueness for system~(\ref{eq:PDE-system-MFG}).  The unique  solution can be obtained by Picard iterations, i.e., by updating iteratively the flow of distributions and the value function. Then, when the time horizon $T$ is large, one can partition the time interval into intervals of duration $\tau$, with $\tau$ small enough. Let us consider the two adjacent intervals 
 $[0,T-\tau]$ and $[T-\tau, T]$.   The solutions in $[0,T-\tau]\times \TT^d$ and in $[T-\tau, T]\times \TT^d$  are coupled through their initial or terminal conditions: for the former interval $[0,T-\tau]$, the initial condition for the distribution of states  is given, but the terminal condition on the value function will come from the solution in $[T-\tau, T]$. 
The principle of the global solver is to use the elementary solver on $[T-\tau, T]$ (because $\tau$ is small enough) and recursive calls of the global solver on $[0,T-\tau]$, (which will in turn call the elementary solver on $[T-2\tau, T-\tau]$ and recursively the global solver on  $[0,T-2\tau]$, and so on so forth).

We present a version of this algorithm based on PDEs (the original version in~\cite{MR3914553} is based on forward-backward stochastic differential equations (FBSDEs for short) but the principle is the same). Recall that $T>0$ is the time horizon, $m_0$ is the initial density, $\phi$ is the terminal condition, and that we want to solve  system~\eqref{eq:PDE-system-MFG}. 

Let $K$ be a positive integer such that $\tau = T/K$ is small enough. Let 
$$
	\texttt{ESolver}:(\tau, \tilde m, \tilde \phi) \mapsto (m,u)
$$
 be an elementary solver which, given  $\tau$, an initial probability density $\tilde m$ defined on  $\TT^d$ 
and a terminal condition $\tilde \phi: \TT^d \to \RR$, returns the solution $(u(t),m(t))_{t \in [0,\tau]}$ to the MFG system of forward-backward PDEs corresponding to these data, i.e.,
\begin{subequations}
     \begin{empheq}[left=\empheqlbrace]{align*}
     	\displaystyle
	&-\frac{\partial u}{\partial t}(t,x)  - \nu \Delta u(t,x) = H(x, m(t,x), \grad u(t,x)), 
	&&\hbox{ in } [0,\tau) \times \TT^d,
	\\
	\displaystyle
	&\frac{\partial m}{\partial t}(t,x)  - \nu \Delta m(t,x) - \div\left( m(t,\cdot) H_p (\cdot, m(t,\cdot), \grad u(t,\cdot)) \right)(x) = 0, 
	&&\hbox{ in } (0,\tau] \times \TT^d,
	\\
	&u(\tau,x) = \tilde \phi(x), \qquad m(0,x) = \tilde m(x), 
	&&\hbox{ in } \TT^d.
     \end{empheq}
\end{subequations}
The solver \texttt{ESolver} may for instance consist of Picard or Newton iterations.

We then define the following recursive function,
 which takes as inputs the level of recursion $k$, the maximum recursion depth $K$, 
and an initial distribution $\tilde m: \TT^d \to \RR$, and returns an approximate solution of the system of PDEs in $[kT/K, T]\times \TT^d$ with initial condition $m(t,x)=\tilde m (x)$ and  terminal condition $u(T,x) = \phi(m(T,x),x)$. Calling $\texttt{RSolver}(0, K, m_0)$ then returns an approximate solution $(u(t),m(t))_{t \in [0,T]}$ to system~\eqref{eq:PDE-system-MFG} in $[0,T]\times \TT^d$, as desired.
\\
To compute the approximate solution on  $[kT/K, T]\times \TT^d$, the following is repeated $J$ times, say from $j=0$ to $J-1$, after some initialization step (see Algorithm~\ref{recursive-solver-algo} for a pseudo code):
\begin{enumerate}
\item  Compute the approximate solution on  $[(k+1)T/K, T]\times \TT^d$ by a recursive call of the algorithm, given the current approximation of $m((k+1)T/K,\cdot)$
(it will come from the next point if $j>0$).
\item Call the elementary solver on  $[kT/K,(k+1)T/K]\times \TT^d$  given $u((k+1)T/K,\cdot)$ coming from the previous point.
\end{enumerate}

\begin{algorithm}[H]
\caption{Recursive Solver for the system of forward-backward PDEs} \label{recursive-solver-algo}
\begin{algorithmic}
\Function {\texttt{RSolver}}{$k, K, \tilde m$}
\If{$k = K$}
		\vspace{0.1cm}
	\State $(u(t, \cdot),m(t, \cdot))_{t = T} = (\phi(\tilde m(\cdot), \cdot), \tilde m(\cdot))$ \hfill \textit{//  last level of recursion}
		\vspace{0.1cm}
\Else
		\vspace{0.1cm}
	\State $(u(t,\cdot), m(t,\cdot))_{t \in [kT/K, (k+1)T/K)]} \leftarrow (0, \tilde m(\cdot))$   \hfill \textit{//  initialization}
		\vspace{0.1cm}
	\For{$j=0, \dots, J$}
		\vspace{0.1cm}
		\State $(u(t, \cdot),m(t, \cdot))_{t \in [(k+1)T/K, T]} \leftarrow \texttt{RSolver}\big(k+1, K, m((k+1)T/K, \cdot) \big)$ \hfill \textit{// interval} $ [(k+1)T/K, T]$
		\vspace{0.1cm}
		\State $(u(t, \cdot),m(t, \cdot))_{t \in [kT/K, (k+1)T/K]} \leftarrow \texttt{ESolver}\big(T/K, m(kT/K, \cdot), u((k+1)T/K, \cdot) \big)$  \hfill \textit{// interval } $ [kT/K, (k+1)T/K]$
		\vspace{0.1cm}
	\EndFor
\EndIf
\Return $(u(t, \cdot),m(t, \cdot))_{t \in [kT/K, T]}$
\EndFunction
\end{algorithmic}
\end{algorithm}

In~\cite{MR3914553}, Chassagneux et al. introduced a version based on FBSDEs and proved, under suitable regularity assumptions on the decoupling field, the convergence of the algorithm, with a complexity that is  exponential in $K$. The method has been further tested in~\cite{Angiulietal-2019} with implementations relying on trees and grids to discretize the evolution of the state process.

\section{An   application to pedestrian flows}
\label{sec:an-example-with}
\subsection{An example with two populations, congestion effects and various boundary conditions}
\label{sec:an-example-with-1}
The numerical simulations discussed in this paragraphs somehow stand at the state of the art  because they combine the following  difficulties:
\begin{itemize}
	\item The MFG models includes congestion effects. Hence the Hamiltonian does not depend separately on $Du$ and $m$. In such cases, the MFG can never be interpreted as an optimal  control problem driven by a PDE; in other words, there is no variational interpretation, which makes it impossible to apply the methods discussed in \S~\ref{sec:varMFG}.
	\item There are two populations of agents which interact with each other, which adds a further layer of complexity. The now well known arguments due to Lasry and Lions and  leading to uniqueness do not apply
	\item The model will include different kinds of boundary conditions corresponding to walls, entrances and exits, which need careful discretizations.
	\item We are going to look for stationary equilibria, despite the fact that there is no underlying ergodicity: there should be a balance between exit and entry fluxes. A special numerical algorithm is necessary in order to capture such situations
\end{itemize}

We consider a two-population mean field game in a complex geometry. It models situations in which two crowds of pedestrians have to share a confined area.
In this case, the state space is a domain of $\R^2$.
The agents belonging to a given population are all identical, and differ from the agents belonging to the other population because they have for instance different objectives, and also because they feel uncomfortable in the regions where their own population is in minority (xenophobia). In the present example, there are several exit doors and the agents aim at reaching some of these doors, depending on which population they belong to. To reach their targets, the agents may have to cross regions in which their own population is in minority. 
More precisely, the running cost of each individual is made of different terms:
\begin{itemize}
\item a first term only depends of the state variable: it models the fact that a given agent more or less wishes to reach one or several exit doors. There is an exit  cost or reward at each doors, which 
 depends on which population the agents belong to. This translates the fact that the agents belonging to different populations have different objectives
\item the second term is a cost of motion. In order to model congestion effects, it depends on the velocity and on the distributions of states for both populations 
\item the third term models xenophobia and aversion to highly crowded regions.
\end{itemize}

\subsubsection{The system of partial differential equations}
\label{sec:syst-part-diff}
Labelling the two populations with the indexes $0$ and $1$, 
the model leads to a  system of four forward-backward partial differential equations as follows:
 \begin{align}\label{eq:ex:3}
\frac {\partial  u_0}{\partial t} +\nu \Delta  u_0 - H_0(\nabla  u_0; m_0,m_1)  &= - \Phi_0 (x,m_0,m_1),  \\
\label{eq:ex:4}
\frac {\partial  m_0}{\partial t}-\nu \Delta  m_0 - \diver\left(  m_0\frac {\partial H_0} {\partial p} (\nabla  u_0; m_0,m_1)\right)  &= 0, \\
\label{eq:ex:5}
\frac {\partial  u_1}{\partial t} +\nu \Delta  u_1 - H_1( \nabla  u_1; m_1,m_0)  &=  -\Phi_1(x,m_1,m_0), \\
\label{eq:ex:6}
\frac {\partial  m_1}{\partial t}- \nu \Delta  m_1 - \diver\left(  m_1\frac {\partial H_1} {\partial p} (\nabla  u_1; m_1,m_0)\right)  &= 0.
\end{align}
In the numerical simulations discussed below, we have chosen
\begin{equation}
  \label{eq:ex:1}
  H_i(x, p; m_i, m_j)= \frac {|p|^2}{1+ m_i + 5 m_j  },
\end{equation}
and 
\begin{equation}
  \label{eq:ex:2}
  \Phi_i(x,m_i,m_j)= 0.5+ 0.5  \left( \frac {m_i}{ m_i+m_j+\epsilon} -0.5  \right)_-  + (m_i+m_j-4)_+,
\end{equation}
where $\epsilon$ is a small parameter and  $j=1-i$. Note that we may replace (\ref{eq:ex:2})  by a   smooth function obtained by a regularization involving another small parameter $\epsilon_2$.
\begin{itemize}
\item The choice of $H_i$ aims at modelling the fact that motion  gets costly in the highly populated regions of the state space. The different factors in front of $m_i$ and $m_j$ in (\ref{eq:ex:1})
 aim at modelling the fact that the cost of motion of an agent of a given type, say $i$,
 is more sensitive to the density of agents of the different type, say $j$; indeed, since  the  agents of different types have different objectives,  their optimal 
controls are unlikely to be aligned, which makes motion even more difficult.
\item
The coupling cost in (\ref{eq:ex:2}) is made of three terms: the first term, namely $0.5$, is the instantaneous  cost for staying in the domain; the term  
$0.5  \left( \frac {m_i}{ m_i+m_j+\epsilon} -0.5  \right)_- $ translates the fact that an agent in population $i$ feels uncomfortable if its population  
is locally in minority.  The last term, namely $(m_i+m_j-4)_+$, models aversion to highly crowded regions.
\end{itemize}

\subsubsection{The domain and the boundary conditions}
\label{sec:doma-bound-cond}
The domain $\Omega$ is displayed in Figure \ref{fig:dom1}. The solid lines stand for walls, i.e. parts of the boundaries that cannot be crossed by the agents.
 The colored arrows indicate entries or exits  depending if they are pointing inward or outward. The two different colors correspond to the two different populations, green for population 0 and orange for population 1.  The length of the northern wall is $5$, and the north-south diameter is $2.5$. The width of the southern exit is $1$.
The widths of the other exits and entrances are $0.33$.  The width of the outward arrows stands for the reward for exiting.
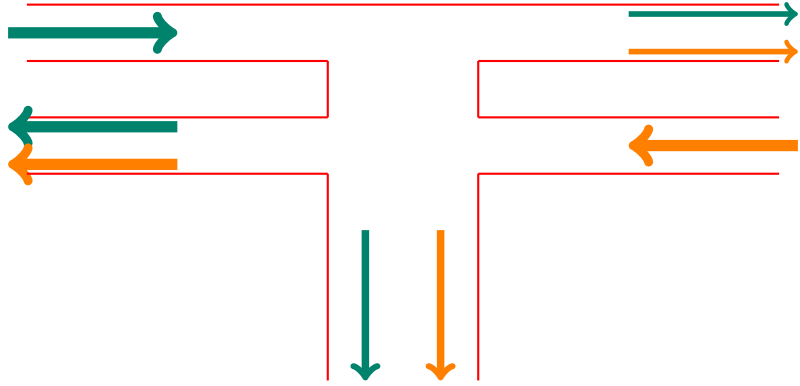
\begin{figure}[htbp]
  \centering
    \begin{center}
    \begin{tikzpicture}[scale=0.5
, trans/.style={thick,<->,shorten >=2pt,shorten <=2pt,>=stealth} ]
      \draw[red,thick] (0,10) -- (20,10);
      \draw [PineGreen, line width =1.5mm,->]  (-0.5,9.25)--(4,9.25); 
      \draw [PineGreen, line width =0.75mm,->] (16,9.75)--(20.5,9.75);
      \draw [orange, line width =0.75mm,->] (16,8.75)--(20.5,8.75);
      \draw[red,thick] (0,8.5) -- (8,8.5);
      \draw[red,thick] (12,8.5) -- (20,8.5);
      \draw[red,thick] (8,8.5) -- (8,7);
      \draw[red,thick] (12,8.5) -- (12,7);
      \draw[red,thick] (0,7) -- (8,7);
      \draw[red,thick] (12,7) -- (20,7);
      \draw [PineGreen, line width =1.5mm,->] (4,6.75) -- (-0.5,6.75) ;
        \draw [orange, line width =1.5mm,->] (4,5.75)-- (-0.5,5.75);
      \draw [orange, line width =1.5mm,->] (20.5,6.25) -- (16,6.25);
      \draw[red,thick] (0,5.5) -- (8,5.5);
      \draw[red,thick] (12,5.5) -- (20,5.5);
      \draw[red,thick] (8,5.5) -- (8,0);
      \draw[red,thick] (12,5.5) -- (12,0);
      \draw [PineGreen, line width =1mm,->]  (9,4)-- (9,0);
      \draw [orange, line width =1mm,->]  (11,4) -- (11,0);
    \end{tikzpicture}
  \end{center}
  \caption{\label{fig:dom1}The domain $\Omega$. The colored arrows indicate entries or exits  depending if they are pointing inward or outward.  The two different colors correspond to the two different populations. The width of the outward arrows stands for the reward for exiting.}
\end{figure}
\begin{itemize}
\item The boundary conditions at walls are as follows:
\begin{equation}\label{eq:ex:7}
     \frac{\partial u_i}{\partial n}(x)= 0,
\quad \hbox{and }\quad     \frac{\partial m_i}{\partial n}(x)= 0.
\end{equation}
The first condition in (\ref{eq:ex:7}) comes from the fact that the stochastic process describing the state of a given agent in population $i$ is reflected on walls.
The second condition  in (\ref{eq:ex:7})  is in fact  $-\nu \frac{\partial m_i}{\partial n} -   m_i\; n\cdot \frac {\partial H_i} {\partial p} (\nabla  u_i; m_i,m_j) = 0$, where we have taken into account  the Neumann condition on $u_i$.
\item At an exit door  for population $i$, the  boundary conditions are as follow
\begin{equation}\label{eq:ex:8}
       u_i =\hbox{ exit cost},\quad \hbox{and }\quad m_i=0.
\end{equation}
A negative exit cost means that the agents are rewarded for exiting the domain through this door. The homogeneous Dirichlet condition on $m_i$ in (\ref{eq:ex:8}) can be explained by saying that the agents stop taking part to the game as soon as they reach the exit.
\item At an entrance for population $i$, the  boundary conditions are as follows:
\begin{equation}\label{eq:ex:9}
       u_i  =\hbox{ exit cost},\quad \hbox{ and }\quad 
\nu \frac{\partial m_i}{\partial n} +  m_i\;n\cdot \frac {\partial H_i} {\partial p} (\nabla  u_i; m_i,m_j) =\hbox{ entry flux}.
\end{equation}
Setting a high exit cost prevents the agents from exiting though the entrance doors.
\end{itemize}
In our simulations, the exit costs of population 0 are as follows:
 \begin{enumerate}
  \item North-West entrance : $0$
  \item South-West exit : $-8.5$
  \item North-East exit : $-4$
  \item South-East exit : $0$
  \item South exit : $-7$
  \end{enumerate}
and the exit costs of population 1 are
  \begin{enumerate}
  \item North-West exit : $0$
  \item South-West exit : $-7$
  \item North-East exit : $-4$
  \item South-East entrance : $0$
  \item South exit: $-4$.
  \end{enumerate}
The entry fluxes are as follows:
 \begin{enumerate}
  \item Population 0:  at the North-West entrance, the entry flux is $1$
  \item Population 1:  at the South-East entrance, the entry flux is $1$.
  \end{enumerate}

For equilibria in finite horizon $T$, the system should be supplemented with an initial  Dirichlet condition for $m_0, m_1$  since the laws of the initial 
distributions are known,  and a terminal Dirichlet-like condition for $u_0,u_1$ accounting for the terminal costs.
  \subsection{Stationary equilibria}
\label{sec:stat-equil}
  
We look for a stationary equilibrium. For that, we solve numerically (\ref{eq:ex:3})-(\ref{eq:ex:6}) with a finite horizon $T$, 
and with the boundary conditions described in paragraph~\ref{sec:doma-bound-cond}, see paragraph~\ref{sec:algor-solv-syst} below   and use
 an iterative method in order to progressively diminish the effects of the initial and terminal conditions:  starting from $(u_i^0, m_i^0)_{i=0,1}$, the numerical solution of the finite horizon problem described above, we construct a sequence of approximate solutions $(u_i^\ell, m_i^\ell)_{\ell \ge 1}$ by the following induction:  $(u_i^{\ell+1}, m_i^{\ell+1})$ is the solution of the finite horizon problem with the same system of PDEs in $(0,T)\times \Omega$, the same boundary conditions on $(0,T)\times \partial \Omega$, 
and the  new initial and terminal conditions as follows:
\begin{eqnarray}
\label{eq:ex:10}
u_i^{\ell+1}(T, x)&=& u_i^{\ell}\left(\frac T 2, x\right),\qquad x\in \Omega,\; i=0,1,\\
m_i^{\ell+1}(0, x)&=& m_i^{\ell}\left(\frac T 2, x\right),\qquad x\in \Omega,\; i=0,1 .
\end{eqnarray}
As $\ell$ tends to $+\infty$, we observe that  $(u_i^\ell, m_i^\ell)$ converge to time-independent functions. At the limit, we obtain a steady solution of 
 \begin{align}\label{eq:ex:11}
\nu \Delta  u_0 - H_0(\nabla  u_0; m_0,m_1)  &= - \Phi_0 (x,m_0,m_1),  \\
\label{eq:ex:12}
-\nu \Delta  m_0 - \diver\left(  m_0\frac {\partial H_0} {\partial p} (\nabla  u_0; m_0,m_1)\right)  &= 0, \\
\label{eq:ex:13}
\nu \Delta  u_1 - H_1( \nabla  u_1; m_1,m_0)  &=  -\Phi_1(x,m_1,m_0), \\
\label{eq:ex:14}
- \nu \Delta  m_1 - \diver\left(  m_1\frac {\partial H_1} {\partial p} (\nabla  u_1; m_1,m_0)\right)  &= 0,
\end{align}
with the boundary conditions on $\partial \Omega$ described in paragraph~\ref{sec:doma-bound-cond}.

\subsection{ A stationary equilibrium with $\nu=0.30$}\label{sec:stat-equil-with}
We first take a relatively large viscosity coefficient namely $\nu=0.30$.
In Figure \ref{fig:stat_sol_0.31}, we display the distributions of states for the two populations, see Subfigure~\ref{fig:1.1}, the value functions for both populations, see Subfigure~\ref{fig:1.2}, and the optimal feedback controls of population 0 (respectively 1) in Subfigure~\ref{fig:1.3} (respectively Subfigure~\ref{fig:1.4}).
 We see that population 0 enters the domain via the north-west entrance, and most probably exits by the south-west exit door. Population 1 
enters the domain via the south-east entrance, and exits by two doors, the south-west and the southern ones. The effect of viscosity is large enough to prevent complete segregation of the two populations.
\begin{figure}	
	\centering
	\begin{subfigure}[t]{0.45\linewidth}
		\centering
		\includegraphics[width=\linewidth]{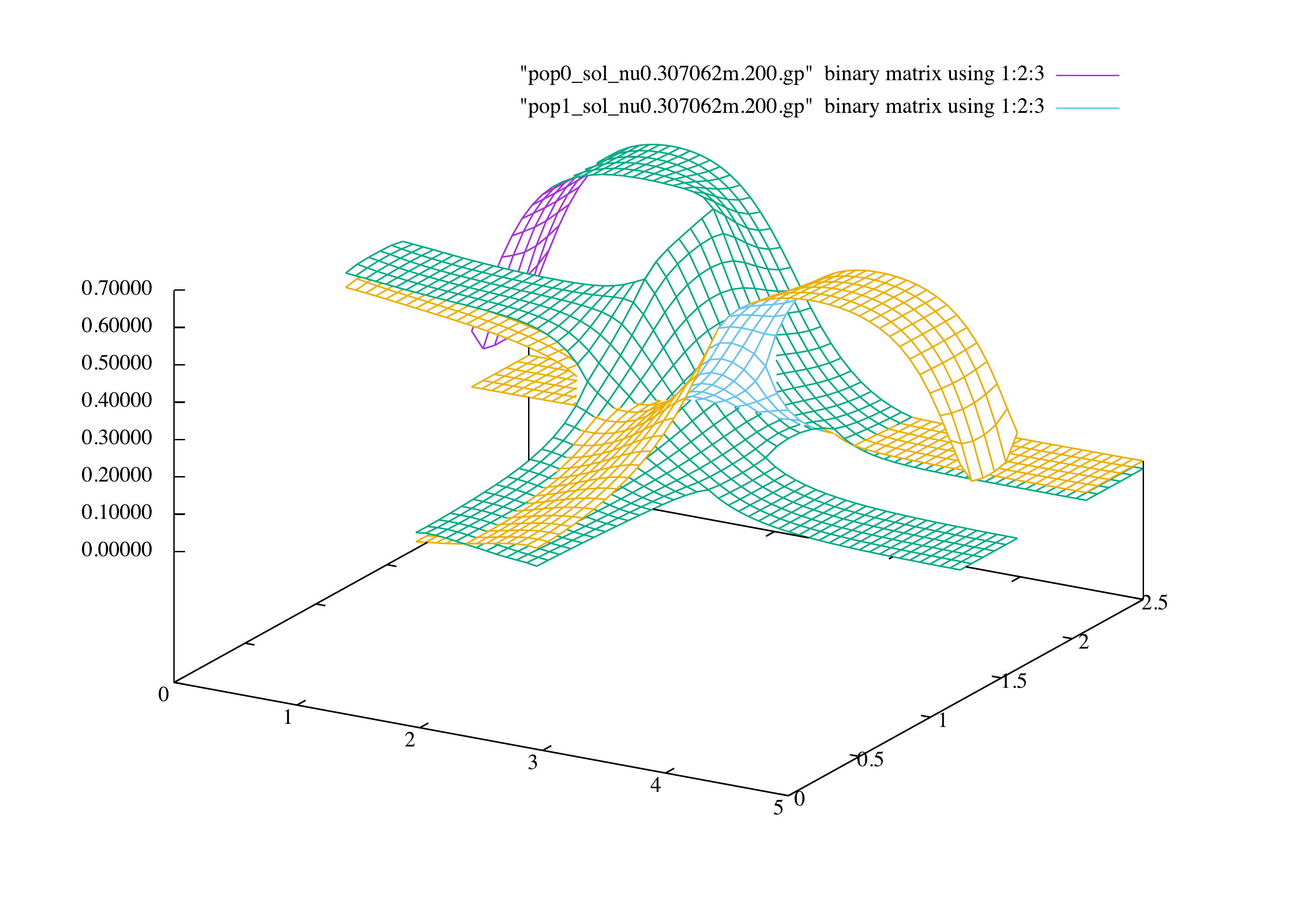}
		\caption{Distributions of the two populations}\label{fig:1.1}
	\end{subfigure}
	\quad
	\begin{subfigure}[t]{0.45\linewidth}
		\centering
		\includegraphics[width=\linewidth]{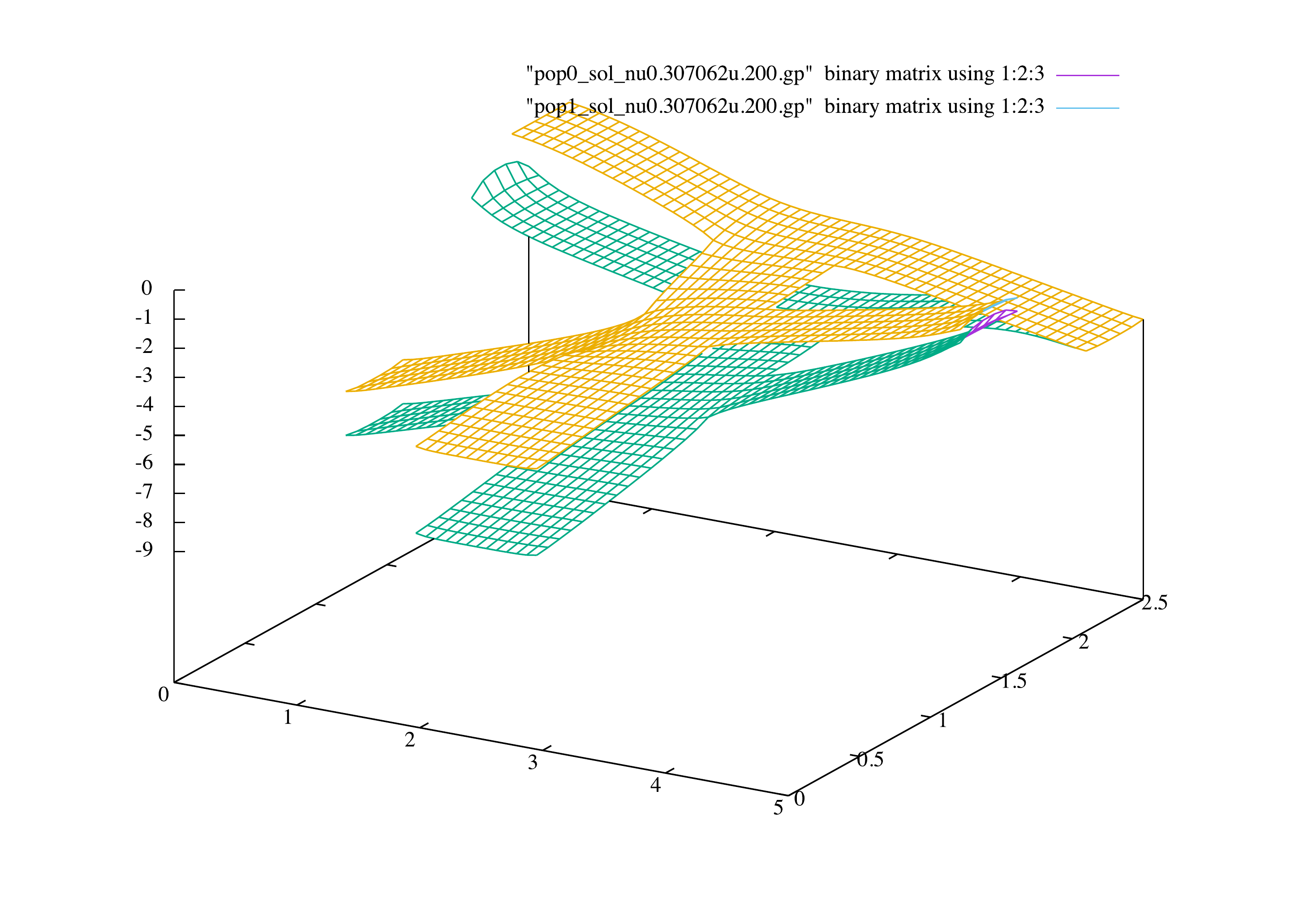}
		\caption{Value functions of the two populations}\label{fig:1.2}
	\end{subfigure}
	\vspace{1em}
	
	\begin{subfigure}[t]{0.45\linewidth}
		\centering
		\includegraphics[width=\linewidth]{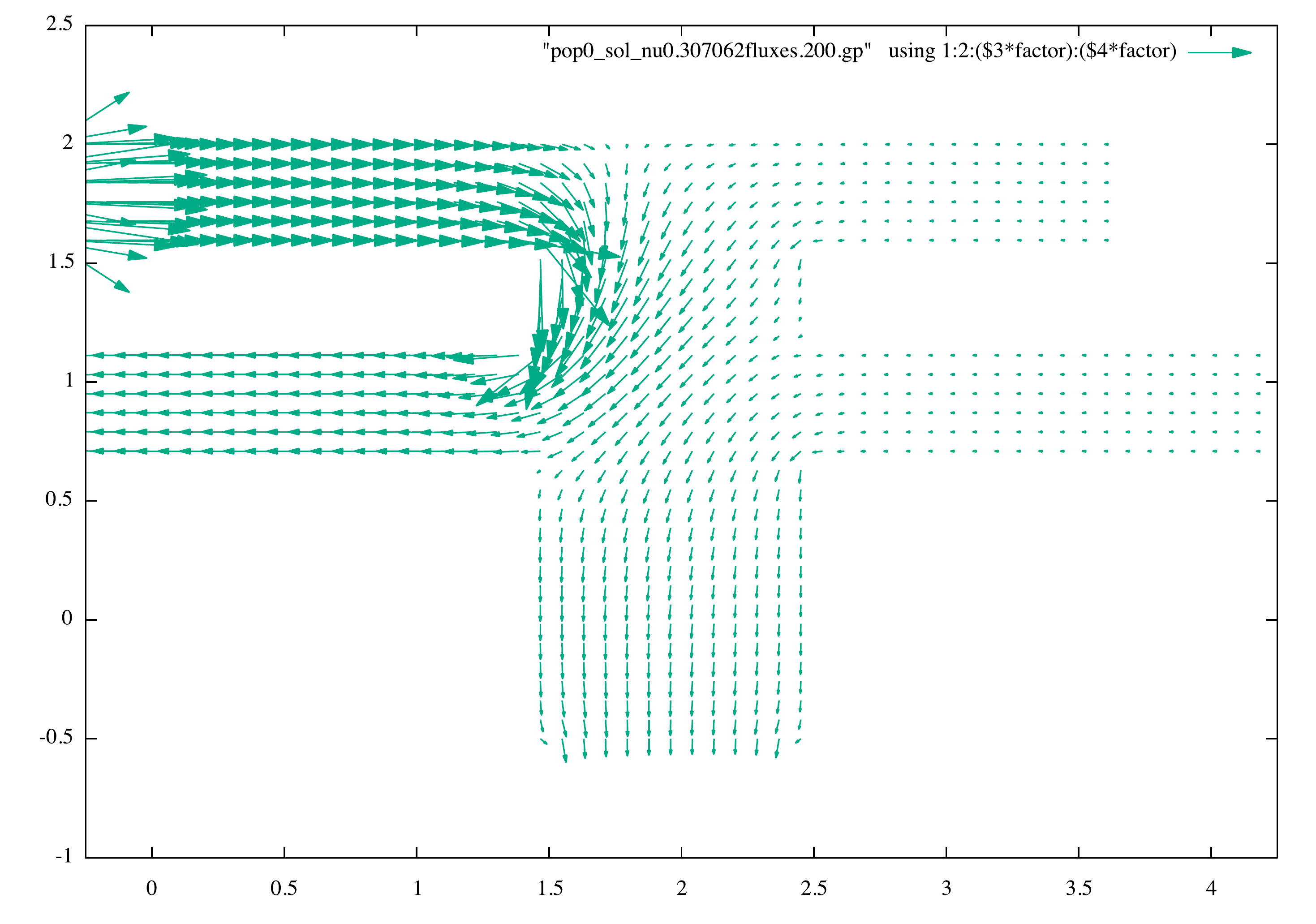}
		\caption{Optimal feedback for population $0$}\label{fig:1.3}
	\end{subfigure}
	\quad
	\begin{subfigure}[t]{0.45\linewidth}
		\centering
		\includegraphics[width=\linewidth]{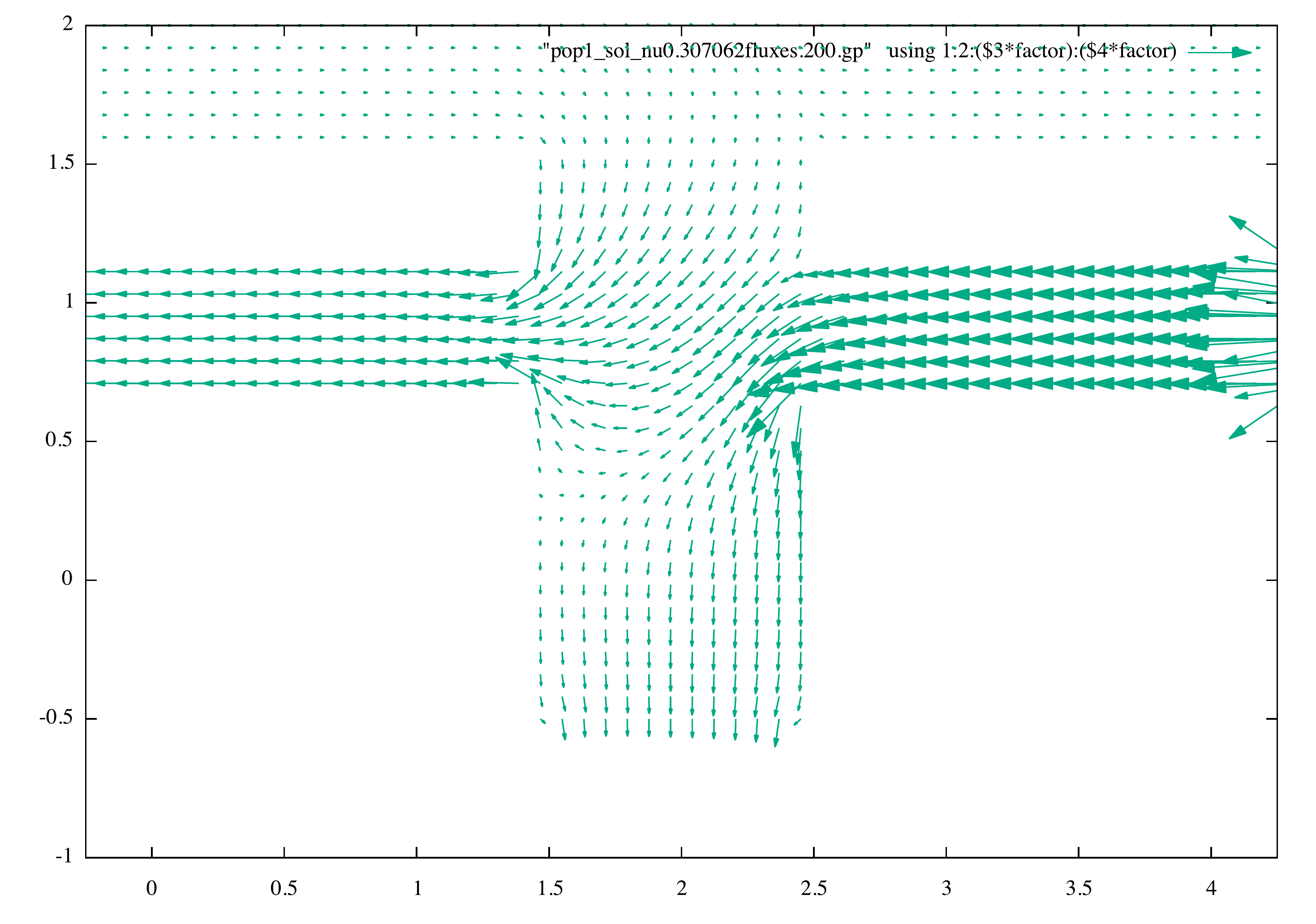}
		\caption{Optimal feedback for population $1$}\label{fig:1.4}
	\end{subfigure}
	\caption{\label{fig:stat_sol_0.31} Numerical Solution to Stationary Equilibrium with $\nu\sim0.3$}
\end{figure}

\subsection{ A stationary equilibrium with $\nu=0.16$}\label{sec:stat-equil-with-3}
We decrease the  viscosity coefficient to the value $\nu=0.16$. We are going to see that the solution is quite different from the one obtained for $\nu=0.3$, because the populations now occupy separate regions. 
In Figure \ref{fig:stat_sol_0.16}, we display the distributions of states for the two populations, see Subfigure~\ref{fig:2.1}, the value functions for both populations, see Subfigure~\ref{fig:2.2}, and the optimal feedback controls of population 0 (respectively 1) in Subfigure~\ref{fig:2.3} (respectively Subfigure~\ref{fig:2.4}).
 We see that population 0 enters the domain via the north-west entrance, and most probably exits by the south-west exit door. Population 1 
enters the domain via the south-east entrance, and  most probably exits the domain by the southern door. The populations occupy almost separate regions.  
We have made simulations with smaller viscosities, up to  $\nu=10^{-3}$, and we have observed that the results are qualitatively similar to the ones displayed on Figure \ref{fig:stat_sol_0.16}, i.e. the distribution of the populations overlap less and less and the optimal strategies are similar. As $\nu$ is decreased, the gradients of the distributions of states increase 
in the transition regions between the two populations.

\begin{figure}	
	\centering
	\begin{subfigure}[t]{0.45\linewidth}
		\centering
		\includegraphics[width=\linewidth]{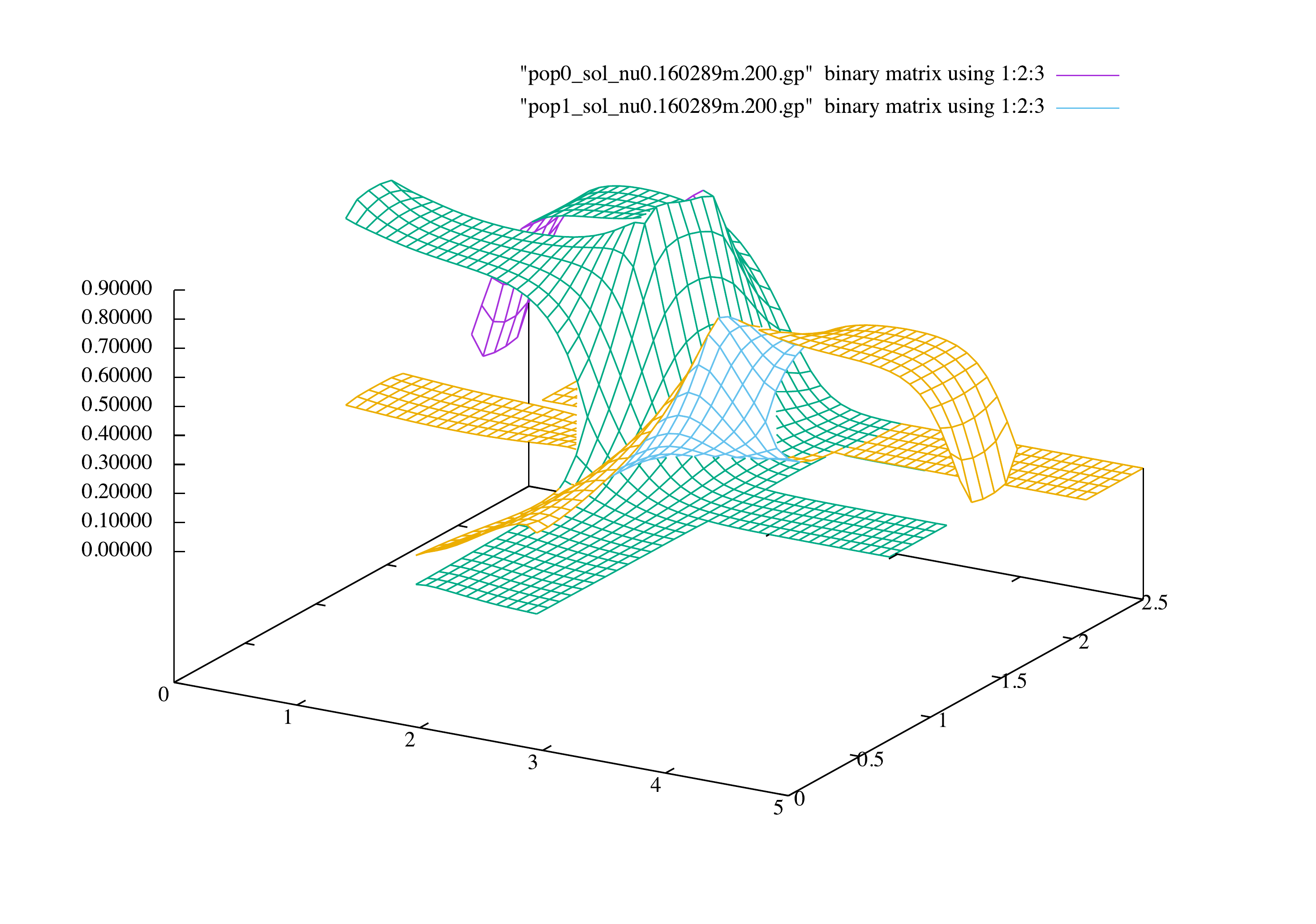}
		\caption{Distributions of the two populations}\label{fig:2.1}
	\end{subfigure}
	\quad
	\begin{subfigure}[t]{0.45\linewidth}
		\centering
		\includegraphics[width=\linewidth]{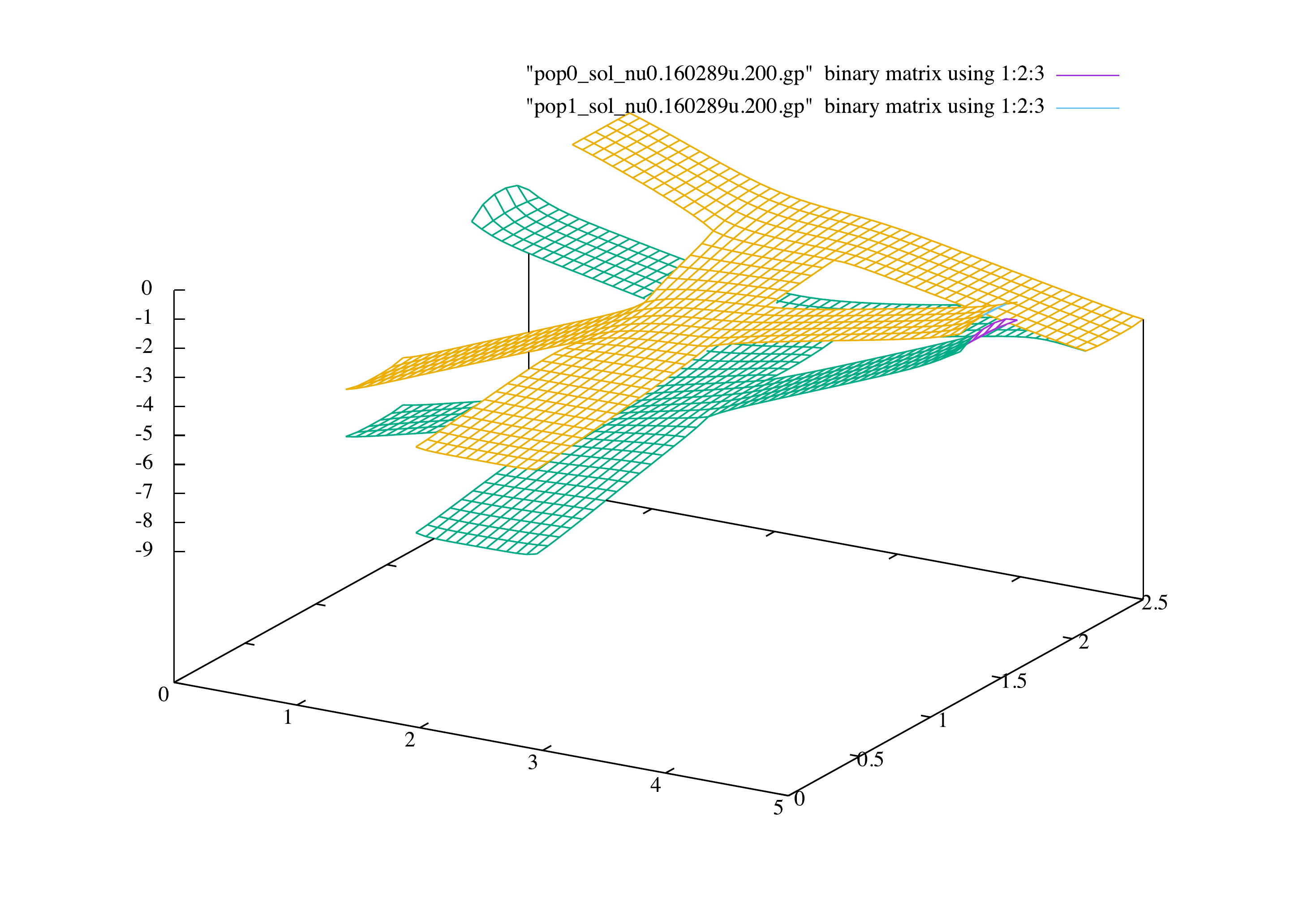}
		\caption{Value functions of the two populations}\label{fig:2.2}
	\end{subfigure}
	\vspace{1em}
	
	\begin{subfigure}[t]{0.45\linewidth}
		\centering
		\includegraphics[width=\linewidth]{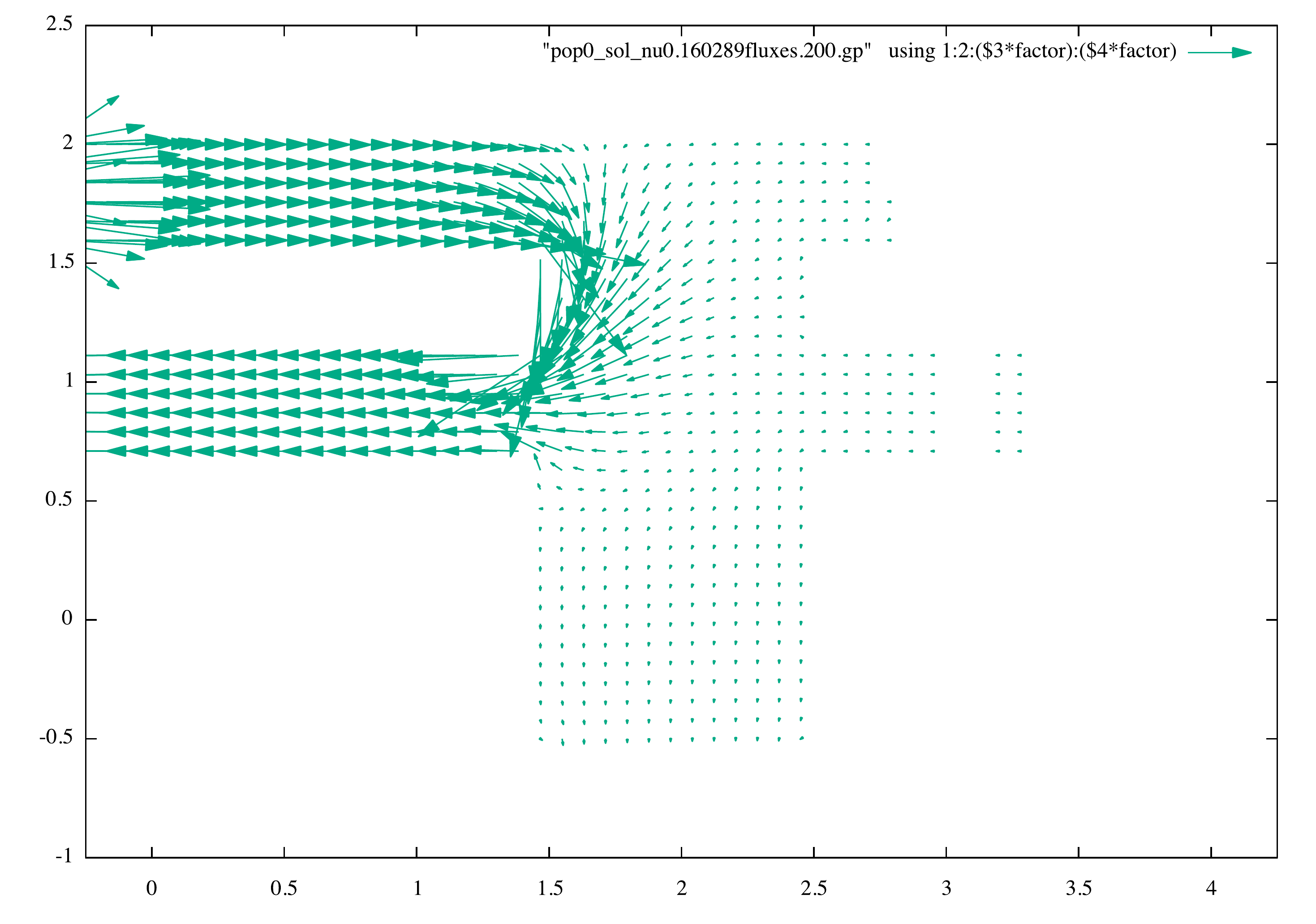}
		\caption{Optimal feedback for population $0$}\label{fig:2.3}
	\end{subfigure}
	\quad
	\begin{subfigure}[t]{0.45\linewidth}
		\centering
		\includegraphics[width=\linewidth]{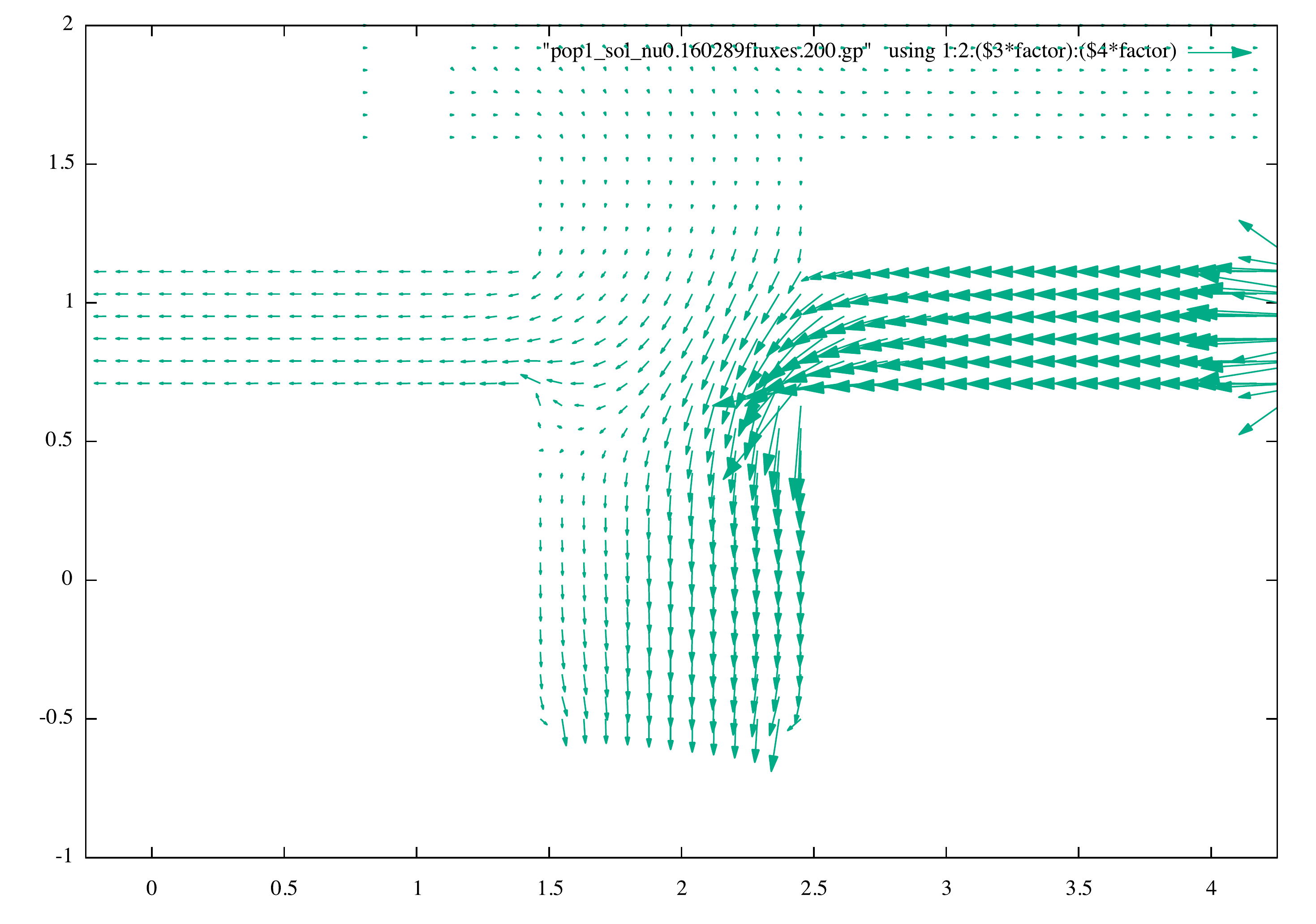}
		\caption{Optimal feedback for population $1$}\label{fig:2.4}
	\end{subfigure}
	\caption{\label{fig:stat_sol_0.16}Numerical Solution to Stationary Equilibrium with $\nu\sim0.16$}
\end{figure}

\subsection{ Algorithm for solving the system of nonlinear equations}\label{sec:algor-solv-syst}
We briefly describe the iterative method used in order to solve the system of nonlinear equations arising from the discretization of the finite horizon problem.
Since the latter system couples  forward and backward (nonlinear) equations, 
it cannot be solved by merely marching in time. 
Assuming that the discrete Hamiltonians are $C^2$ and the coupling functions are $C^1$ (this is true after a regularization procedure involving a small regularization parameter)
 allows us to use a Newton-Raphson method for the whole 
system of nonlinear equations  (which can be huge if  $d\ge 2$). \\

Disregarding the boundary conditions for simplicity, the discrete version of the MFG system reads 
\begin{eqnarray}
  \label{eq:ex:15}\ds \frac{U_{i,j}^{k,n+1}-U_{i,j}^{k,n}}{\Delta t} - \nu (\Delta_h U^{k,n+1})_{i,j} + g\left([\nabla_h U^{k,n+1}]_{i,j},  Z^{k,n+1}_{i,j}\right)
 =  Y^{k,n+1}_{i,j} \\
  \label{eq:ex:16}\ds \frac{M_{i,j}^{k,n+1}-M_{i,j}^{k,n}}{\Delta t} - \nu (\Delta_h M^{k,n+1})_{i,j} - \cT_{i,j} \left(U^{k,n+1}, M^{k,n+1},  Z^{k,n+1}\right) = 0,  \\
  \label{eq:ex:17}\ds Y^{k,n+1}_{i,j} =-\Phi_k\left(x_{i,j},M^{k,n+1}_{i,j},  M^{1-k,n+1}_{i,j}\right) ,\\
  \label{eq:ex:18} \ds Z^{k,n+1}_{i,j}= \left( 1+ M^{k,n+1}_{i,j} +5M^{1-k,n+1}_{i,j}\right)^{-1},
\end{eqnarray}
for $k=0,1$. The system above is satisfied for internal points of the grid, i.e. $2 \le i,j \le N_h - 1$, and is supplemented with suitable boundary conditions.
 The ingredients in (\ref{eq:ex:15})-(\ref{eq:ex:18}) are as follows:
\[
[\nabla_h U]_{i,j} = ((D^+_1 U)_{i,j}, (D^+_1 U)_{i-1,j}, (D^+_2 U)_{i,j}, (D^+_2 U)_{i,j-1})\in \R^4,
\]
and the Godunov discrete Hamiltonian is
\[
\tilde H(q_1,q_2,q_3,q_4,z)= z\left( [(q_1)^-]^2 + [(q_3)^-]^2 + [(q_2)^+]^2 + [(q_4)^+]^2\right).
\]
The transport operator in the discrete Fokker-Planck equation is given by 
\[
\cT_{i,j}(U,M,Z) = \frac{1}{h} \left( \begin{array}{l}
M_{i,j} \partial_{q_1} \tilde H[\nabla_h U]_{i,j}, Z_{i,j}) - M_{i-1,j} \partial_{q_1} \tilde H[\nabla_h U]_{i-1,j},Z_{i-1,j}) \\
\quad + M_{i+1,j} \partial_{q_2} \tilde H[\nabla_h U]_{i+1,j},Z_{i+1,j}) - M_{i,j} \partial_{q_2} \tilde H[\nabla_h U]_{i,j},Z_{i,j}) \\
\quad\quad + M_{i,j} \partial_{q_3} \tilde H[\nabla_h U]_{i,j},Z_{i,j}) - M_{i,j-1} \partial_{q_3} \tilde Hx,[\nabla_h U]_{i,j-1},Z_{i,j-1}) \\
\quad\quad\quad + M_{i,j+1} \partial_{q_4} \tilde H[\nabla_h U]_{i,j+1},Z_{i,j+1}) - M_{i,j} \partial_{q_4} \tilde H[\nabla_h U]_{i,j},Z_{i,j})
\end{array}\right),
\]
 We define the map $\Theta$ : \[ \Theta:\quad \left(Y^{0,n}_{i,j}, Y^{1,n}_{i,j},Z^{0,n}_{i,j}, Z^{1,n}_{i,j} \right)_{i,j,n} \mapsto
  \left(M^{0,n}_{i,j}, M^{1,n}_{i,j}   \right)_{i,j,n},\] 
 by solving first the discrete HJB equation (\ref{eq:ex:15}) (supplemented with boundary conditions)
 then the discrete Fokker-Planck equation (\ref{eq:ex:16}) (supplemented with boundary conditions).  We then summarize (\ref{eq:ex:17}) and (\ref{eq:ex:18}) by writing 
\begin{equation}
  \label{eq:ex:19}
\left(Y^{0,n}_{i,j}, Y^{1,n}_{i,j},Z^{0,n}_{i,j}, Z^{1,n}_{i,j}   \right)_{i,j,n}= \Psi \left(  \left(M^{0,n}_{i,j}, M^{1,n}_{i,j}   \right)_{i,j,n}\right).
\end{equation}
Here $n$ takes its values in  $\{1\dots, N\}$ and   $i,j$ take their values  in  $\{1\dots, N_h\}$.
We then see the full system (\ref{eq:ex:15})-(\ref{eq:ex:18}) (supplemented with boundary conditions) as a fixed point problem for the map $\Xi=\Psi\circ\Theta$.
\\
Note that in (\ref{eq:ex:15}) the two discrete Bellman equations are decoupled and do not involve $M^{k,n+1}$. 
Therefore, one can first obtain $U^{k,n}$ $0\le n\le N$, $k=0,1$ 
by marching backward in time (i.e. performing a backward loop with respect to the index $n$).
For every time index $n$, the two systems of nonlinear equations for $U^{k,n}$, $k=0,1$ are themselves solved by means 
of a nested Newton-Raphson method. 
 Once an approximate solution of (\ref{eq:ex:15}) has been found, 
one can solve the (linear) Fokker-Planck equations (\ref{eq:ex:16}) for $M^{k,n}$ $0\le n\le N$, $k=0,1$,  
by marching forward in time (i.e. performing a forward loop with respect to the index $n$).  
The solutions of (\ref{eq:ex:15})-(\ref{eq:ex:16}) are such that $M^{k,n}$ are nonnegative.
\\
 The fixed point equation for $\Xi$ is solved numerically 
by using a Newton-Raphson method. This  requires the differentiation of both the Bellman and Kolmogorov equations in (\ref{eq:ex:15})-(\ref{eq:ex:16}), which may be done either analytically (as in the present simulations) or via automatic differentiation of computer codes (to the best of our knowledge,  no  computer code for MFGs based on automatic differentiation is available yet, but developing such codes seems doable and interesting).
\\
A good choice of an initial guess  is important, as always for  Newton methods. 
To address this matter, we first observe that the 
above mentioned iterative  method generally quickly converges to a solution when the value of $\nu$ is large.
This leads us to use a continuation method in the variable $\nu$:
 we start solving  (\ref{eq:ex:15})-(\ref{eq:ex:18}) 
with a rather high value of the parameter $\nu$ (of the order of $1$),
 then gradually decrease $\nu$ down to the desired value,
 the solution found for a  value of $\nu$ being used as an initial guess
 for the iterative solution with the next and smaller value of $\nu$.

\section{Mean field type control}
\label{MFTC}

As mentioned in the introduction, the theory of mean field games allows one to study Nash equilibria in games with a number of players tending to infinity. In such models, the players are selfish and try to minimize their own individual cost. 
Another kind of asymptotic regime is obtained by  assuming that all the agents use the same distributed feedback strategy
 and by passing  to the limit as $N\to \infty$ before optimizing the common feedback. For a fixed common feedback strategy, the asymptotic behavior is 
given by the McKean-Vlasov theory, \cite{MR0221595,MR1108185}: the dynamics of a representative agent is found by solving  a stochastic differential equation with coefficients depending on a mean field, namely the statistical distribution of the states, which may also affect the objective function. Since the feedback strategy is common to all agents, perturbations of the latter affect the mean field (whereas in a mean field game, the other players' strategies are fixed when a given player optimizes).  Then, having each agent optimize her objective function amounts to solving a control problem
 driven by the McKean-Vlasov dynamics. The latter is named control of McKean-Vlasov dynamics by R.~Carmona and F.~Delarue~\cite{MR3045029,MR3091726,CarmonaDelarue_book_I} and mean field type control by A.~Bensoussan et al~\cite{MR3037035,MR3134900}.
 Besides the aforementioned interpretation as a social optima in collaborative games with a number of agents growing to infinity, mean field type control problems have also found applications in finance and risk management for problems in which the distribution of states is naturally involved in the dynamics or the cost function.  Mean field type control problems lead to a system of forward-backward PDEs which has some features similar to the MFG system, but which can always be seen as the optimality conditions of a minimization problem. These problems can be tackled using the methods presented in this survey, see e.g.~\cite{MR3498932,MR3575615}. For the sake of comparison with mean field games, we provide in this section an example of crowd motion (with a single population) taking into account congestion effects. The material of this section is taken from~\cite{MR3392611}.

\subsection{Definition of the problem}
Before focusing a specific example, let us present the generic form of a mean field type control problem. To be consistent with the notation used in the MFG setting, we consider the same form of dynamics and the same running and terminal costs functions $f$ and $\phi$. However, we focus on a different notion of solution: Instead of looking for a fixed point, we look for a minimizer when the control directly influences the evolution of the population distribution. More precisely, the goal is to find a feedback control $v^*: Q_T\to \RR^d$ minimizing the following functional:
\begin{align*}
	J: v \mapsto J(v) = \EE \left[\int_0^T f(X_t^v, m^v(t,X_t^v), v(t,X_t^v) ) dt + \phi(X_T^v, m^v(T,X_T^v)) \right]
\end{align*}
under the constraint that the process $X^v = (X_t^v)_{t \ge 0}$ solves the stochastic differential equation (SDE)
\begin{equation}
\label{eq:dyn-X-general-MFTC}
	d X_t^v = b(X_t^v, m^v(t,X_t^v), v(t, X_t^v)) dt + \sqrt{2 \nu} d W_t, \qquad t \ge 0,
\end{equation}
and $X_0^v$ has distribution with density $m_0$. Here $m^v_t$ is the probability density of the law of $X_t^v$, so the dynamics of the stochastic process is of McKean-Vlasov type. 

For a given feedback control $v$, $m^v_t$ solves the same Kolmogorov-Fokker-Planck (KFP) equation~\eqref{eq:intro-MFG-KFP} as in the MFG, but the key difference between the two problems lies in the optimality condition. For the mean field type control problem, one can not rely on standard optimal control theory because the distribution of the controlled process is involved in a potentially non-linear way.

In \cite{MR3134900}, A. Bensoussan, J. Frehse and P. Yam have proved that a necessary condition
for the existence of a smooth feedback function $v^* $ achieving
$J( v^*)= \min J(v)$ is that
$$
	v^*(t,x) = \argmax_{a \in \RR^d} \big\{ -f(x, m(t,x), a) - \langle  b(x, m(t,x), a) , \nabla u(t,x) \rangle  \big\},
$$
where $(m,u)$ solve the
following system of partial differential equations
\begin{subequations}
     \begin{empheq}[left=\empheqlbrace]{align*}
     	\displaystyle
	&-\frac{\partial u}{\partial t}(t,x)  - \nu \Delta u(t,x) + H(x, m(t,x), \grad u(t,x))  + \int_{\TT^d} \frac{\partial H}{\partial m}(\xi, m(T,\xi), \grad u(t, \xi)) m(T,\xi) d\xi =0, 
	&&\hbox{ in } [0,T) \times \TT^d,
	\\
	\displaystyle
	&\frac{\partial m}{\partial t}(t,x)  - \nu \Delta m(t,x) - \div\left( m(t,\cdot) H_p (\cdot, m(t,\cdot), \grad u(t,\cdot)) \right)(x) = 0, 
	&&\hbox{ in } (0,T] \times \TT^d,
	\\
	&u(T,x) = \phi(x, m(T,x)) + \int_{\TT^d} \frac{\partial \phi}{\partial m}(\xi, m(T,\xi)) m(T,\xi) d\xi , \qquad m(0,x) = m_0(x), 
	&&\hbox{ in } \TT^d.
     \end{empheq}
\end{subequations}
Here, $\frac{\partial}{\partial m}$ denotes a derivative with respect to the argument $m$, which is a real number because the dependence on the distribution is purely local in the setting considered here. When the Hamiltonian depends on the distribution $m$ in a non-local way, one needs to use a suitable notion of derivative with respect to a probability density or a probability measure, see e.g.~\cite{cardaliaguet2010,MR3134900,CarmonaDelarue_book_I} for detailed explanations.

\subsection{Numerical simulations}
\label{sec:some-simulations}
Here we model a situation in which a crowd of pedestrians   is driven to leave a given square hall (whose side is 50 meters long) containing rectangular  obstacles:
 one can imagine for example a situation of panic in
a closed building, in which the population tries to reach the exit doors. The chosen geometry is represented on Figure~\ref{fig:MFTC-geom-m0}.
\begin{figure}
	\centering
	\begin{subfigure}[t]{0.4\linewidth}
		\centering
		\includegraphics[height=4cm]{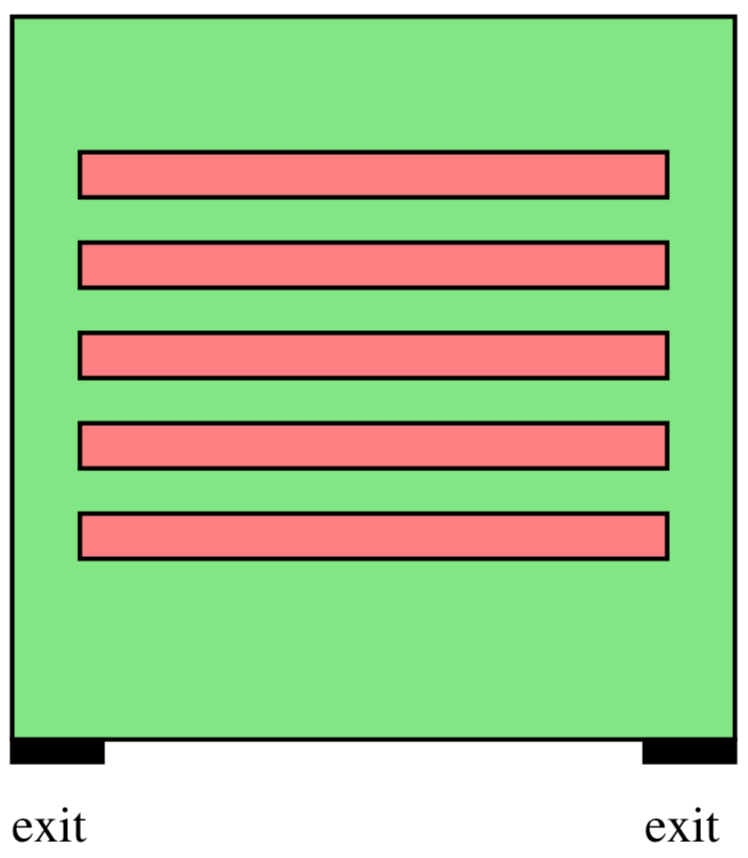}
	\end{subfigure}
	\quad
	\begin{subfigure}[t]{0.4\linewidth}
		\centering
	  	\includegraphics[height=5cm]{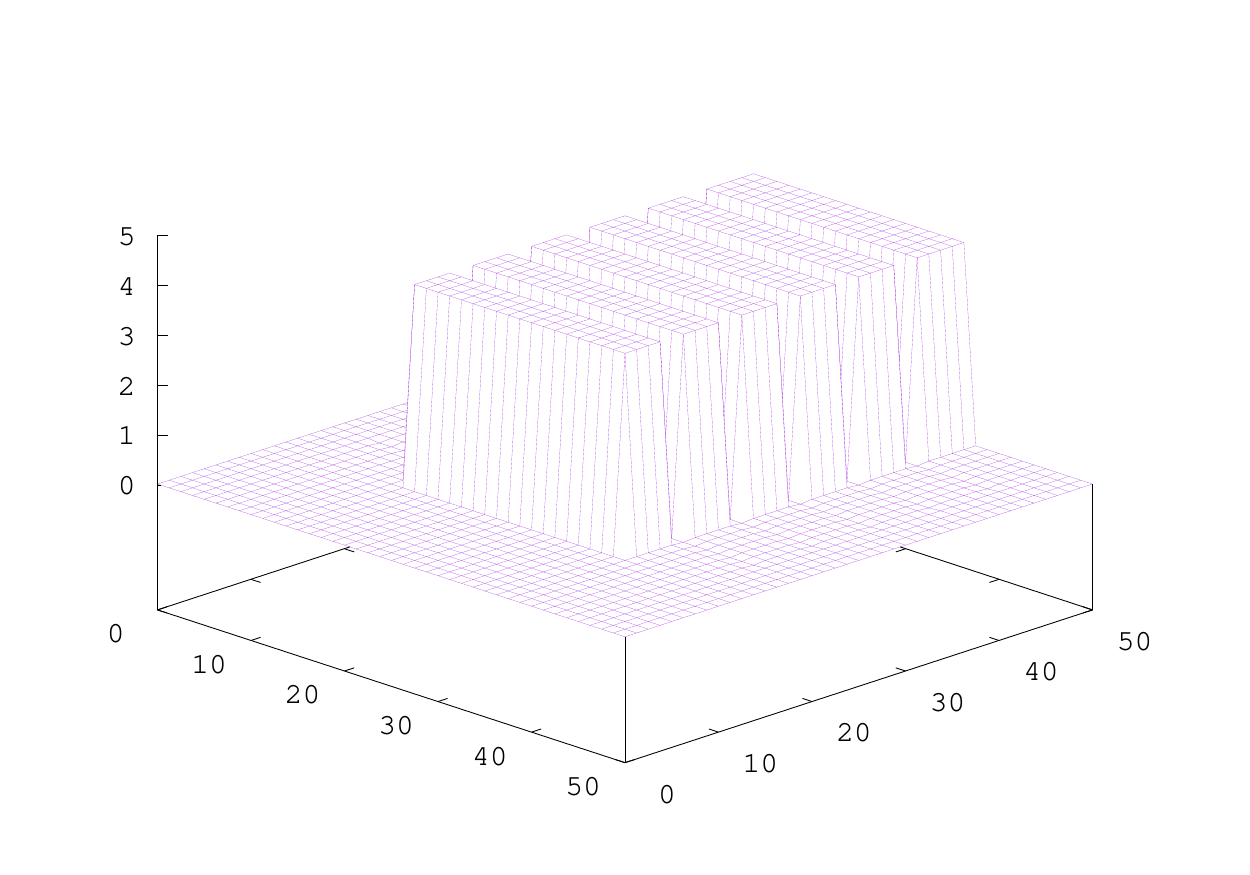}
	\end{subfigure}
	\caption{\label{fig:MFTC-geom-m0} Left: the geometry (obstacles are in red). Right: the density at $t=0$.}
\end{figure}

The aim is to compare the evolution of the density in two models:
\begin{enumerate}
\item Mean field games: we choose $\nu=0.05$ and  the Hamiltonian takes congestion effects
 into account  and depends locally  on $m$; more precisely:
  \begin{displaymath}
    H(x, m, p)= -\frac {8 |p|^2} {(1+m)^{\frac 3 4}} + \frac 1 {3200} \,.
  \end{displaymath}
The MFG  system of PDEs ~\eqref{eq:PDE-system-MFG} becomes
\begin{subequations}
     \begin{empheq}[left=\empheqlbrace]{align}
     	\displaystyle
	\frac {\partial  u}{\partial t}  +   0.05 \;  \Delta  u
     -\frac {8} {(1+m)^{\frac 3 4}} \;   |\nabla u|^2 &=  -    \frac  1  {3200} \,, 
	\\
	\displaystyle
	\frac {\partial  m}{\partial t}  -   0.05\;  \Delta m   - 16 \, \div\left(   \frac { m \nabla u } {(1+m)^{\frac 3 4}}  \right) &= 0 \,.
	\label{eq:MFG-KFP-congestion}
     \end{empheq}
\end{subequations}
The horizon $T$  is $T=50$ minutes. There is no terminal cost, i.e. $\phi \equiv 0$. 

There are two exit doors, see Figure~\ref{fig:MFTC-geom-m0}. The part of the boundary corresponding to the doors is called  $\Gamma_D$. The boundary conditions at the exit doors are chosen as follows:  there is a Dirichlet condition for $u$ on $\Gamma_D$, corresponding to an exit cost; in our simulations, we have chosen $u=0$ on $\Gamma_D$.
For $m$, we may assume that $m =0$ outside the domain, so we also get the Dirichlet condition  $m=0$ on $\Gamma_D$.

The boundary $\Gamma_N$  corresponds to the solid walls of the hall and of the obstacles.
A natural boundary condition for $u$ on $\Gamma_N$ is a homogeneous Neumann boundary condition, i.e.
$ \frac{\partial u}{\partial n}=0$ which says that the velocity of the pedestrians is tangential to the walls.
 The natural condition for the density $m$ is that
$\nu \frac {\partial  m}{\partial n}+m     \frac {\partial \tilde H} {\partial p} (\cdot, m, \nabla  u )\cdot n=0$, therefore $ \frac {\partial  m}{\partial n}=0$
on $\Gamma_N$ .
\item Mean field type control:  this is the situation where pedestrians or robots use the same feedback law (we may imagine that they follow the strategy  decided by a leader); we keep the same Hamiltonian, and the HJB equation becomes
  \begin{equation}
    \label{eq:71}
   \frac {\partial  u}{\partial t}  +   0.05 \;  \Delta  u   -
     \left( \frac {2} {(1+m)^{\frac 3 4}}    +   \frac {6} {(1+m)^{\frac 7 4}}\right) \;   |\nabla u|^2= -     \frac  1  {3200}.
  \end{equation}
while \eqref{eq:MFG-KFP-congestion} and the boundary conditions are unchanged.
\end{enumerate}
The initial density $m_0$ is piecewise constant and takes two values $0$ and $4$ people/m$^2$, see Figure \ref{fig:MFTC-geom-m0}.
At $t=0$, there are 3300 people in the hall.

We use the method described in section~\ref{sec:newton-algorithms}, i.e., Newton iterations with the finite difference scheme originally proposed in \cite{MR2679575},
see \cite{MR3135339} for some details on the implementation.

On Figure~\ref{fig:MFTC-MFG-congestion-evol-m}, we plot the density $m$ obtained by the simulations for the two models, at $t=1$, $2$, $5$ and $15$ minutes.
With both models, we see that the pedestrians rush towards the narrow corridors leading to the exits, at the left and right sides of the hall, and that the density reaches high values
at the intersections of corridors; then congestion effects explain why the velocity is low (the gradient of $u$)
in the regions where the density is high, see Figure~\ref{fig:MFTC-MFG-congestion-evol-gradu-cmp}. We see on Figure~\ref{fig:MFTC-MFG-congestion-evol-m} that the mean field type control leads to
 lower peaks of density, and on Figure~\ref{fig:MFTC-MFG-congestion-evol-totalmass} that it leads to a faster exit of the room. We can hence infer that the mean field type control performs better than the mean field game, leading to a positive {\sl price of anarchy}.

 \begin{center}
   \includegraphics[width=7cm]{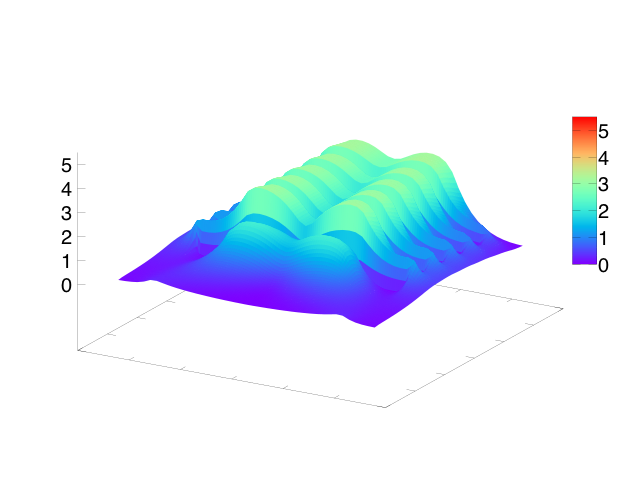} \quad \includegraphics[width=7cm]{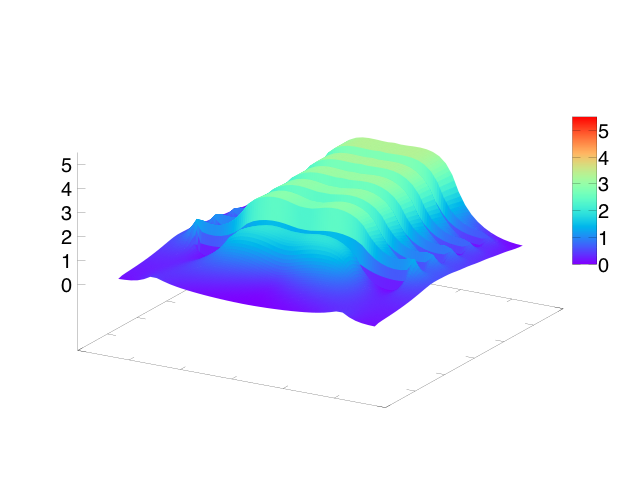}
   \\    \vskip -1cm
    \includegraphics[width=7cm]{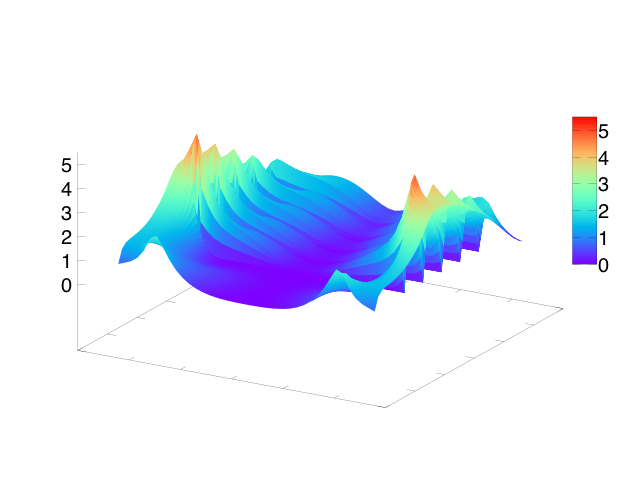} 
  \quad  
   \includegraphics[width=7cm]{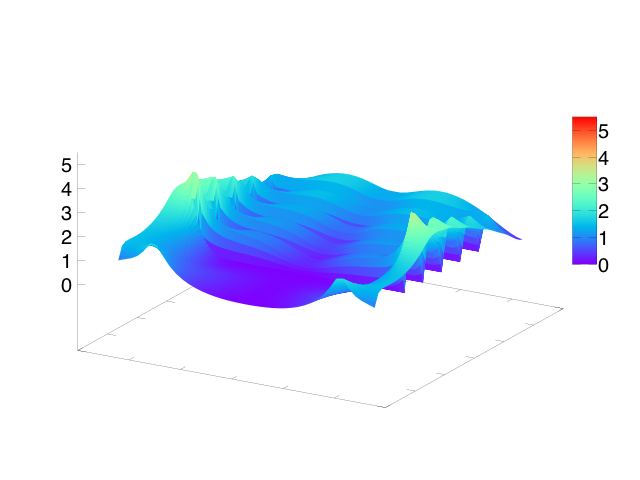}
   \\    \vskip -1cm
   \includegraphics[width=7cm]{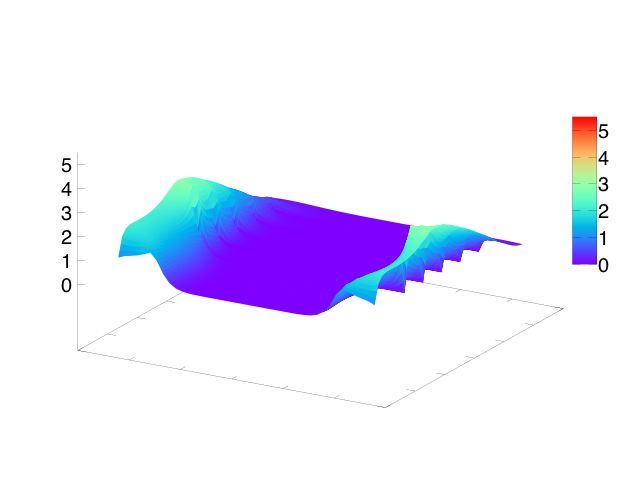}
  \quad  
  \includegraphics[width=7cm]{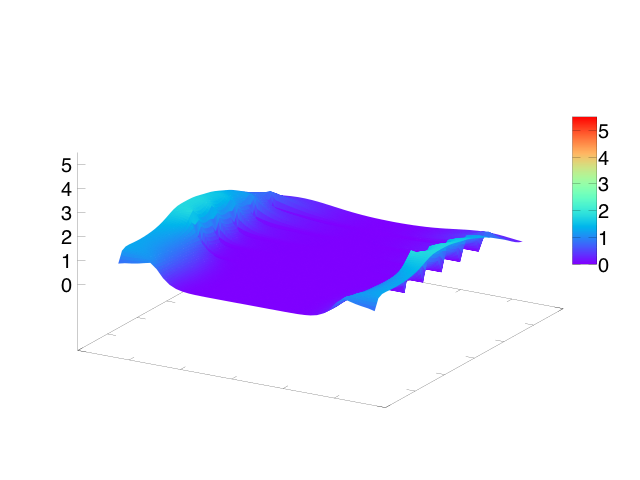}
     \vskip -1cm
  \captionof{figure}{ \label{fig:MFTC-MFG-congestion-evol-m} The density computed with the two models at different dates, $t=1,5$ and $15$ minutes (from top to bottom). Left: Mean field game. Right: Mean field type control.}
 \end{center}

 \begin{center}
   \includegraphics[width=7cm]{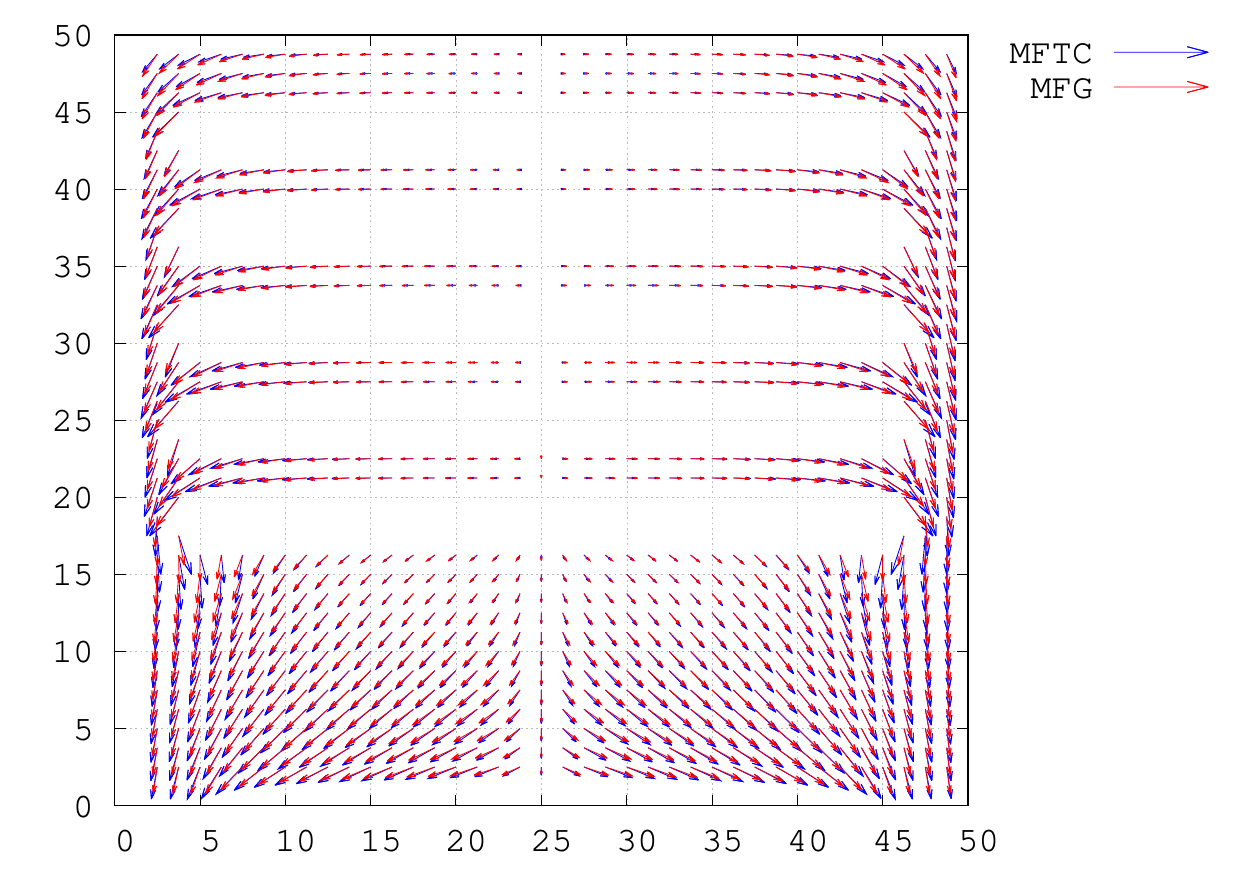}\quad
   \includegraphics[width=7cm]{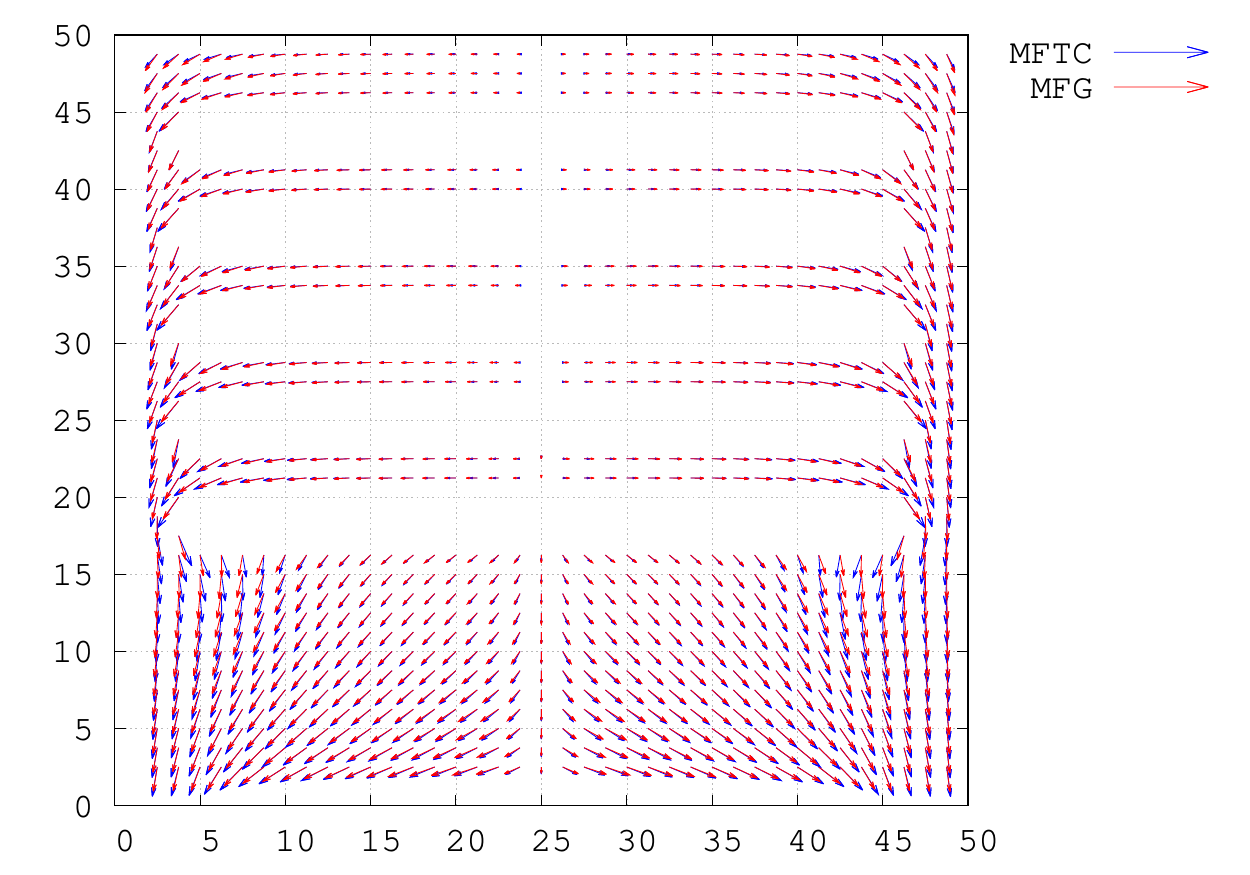}
   \\
   \includegraphics[width=7cm]{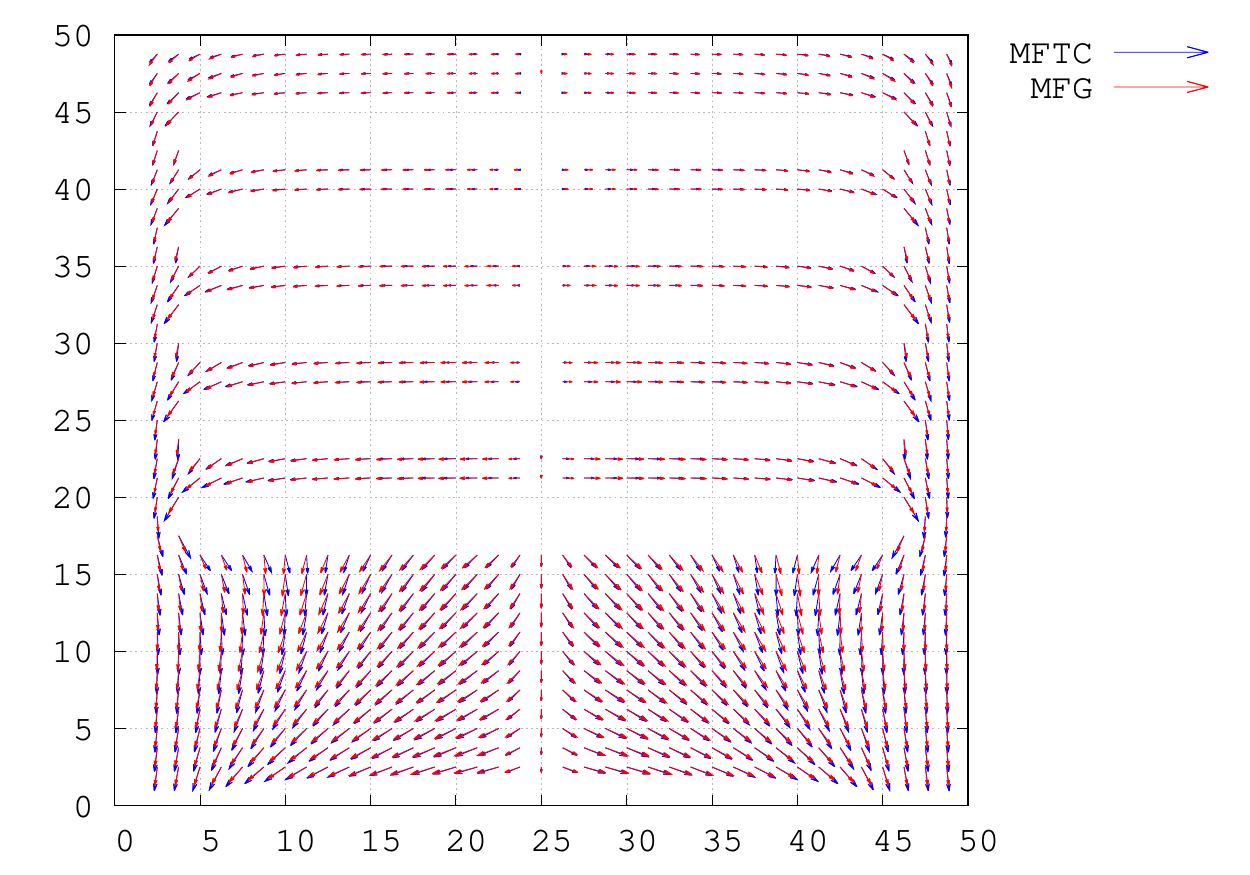}\quad  \includegraphics[width=7cm]{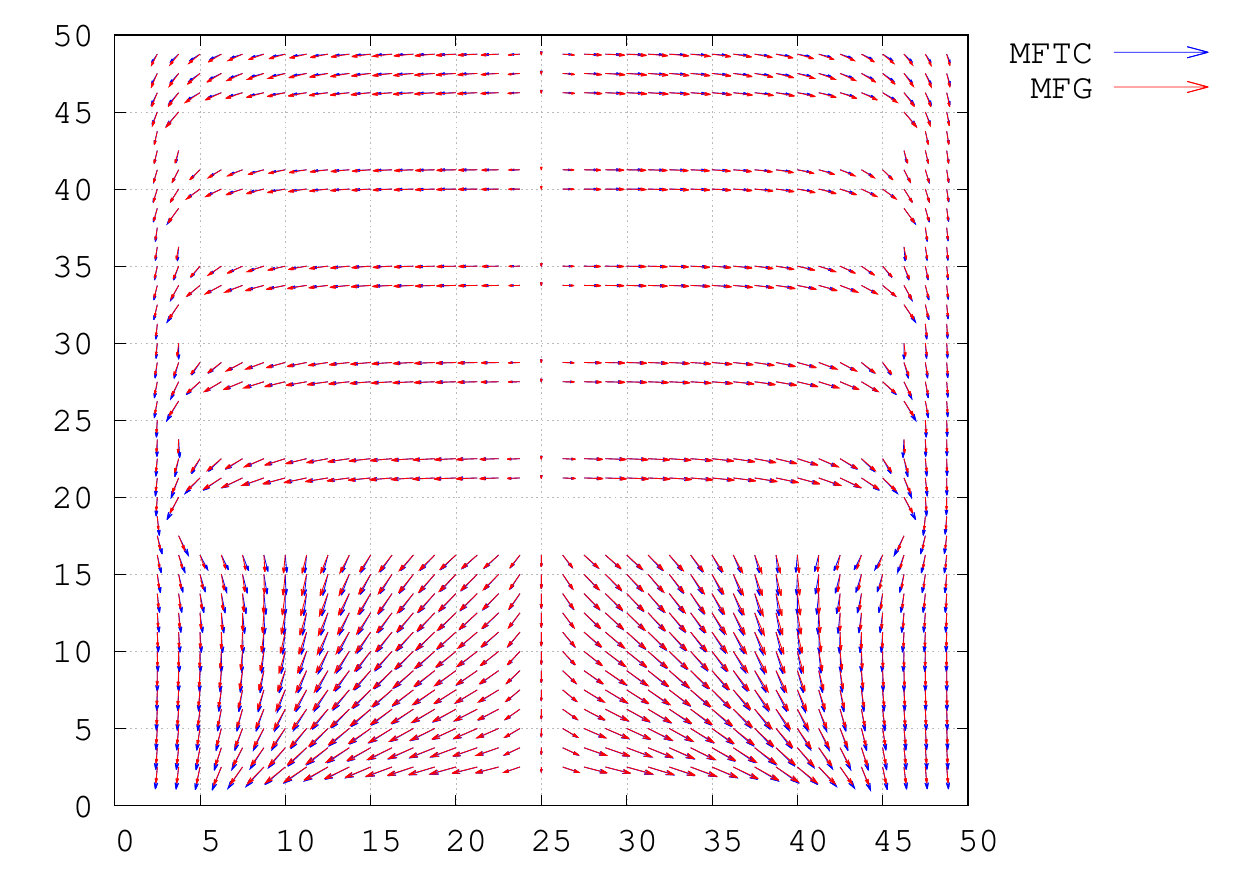}
    \captionof{figure}{ \label{fig:MFTC-MFG-congestion-evol-gradu-cmp} The velocity ($v(t,x) = -16 \nabla u(t,x) (1+m(t,x))^{-3/4}$) computed with the two models at different dates, $t=1,2,5$ and $t=15$ minutes (from top-left to bottom-right). Red: Mean field game. Blue: Mean field type control. }
 \end{center}

\begin{figure}[H] 
	\centering
	\begin{subfigure}[t]{0.6\linewidth}
		\centering
		\includegraphics[height=5cm]{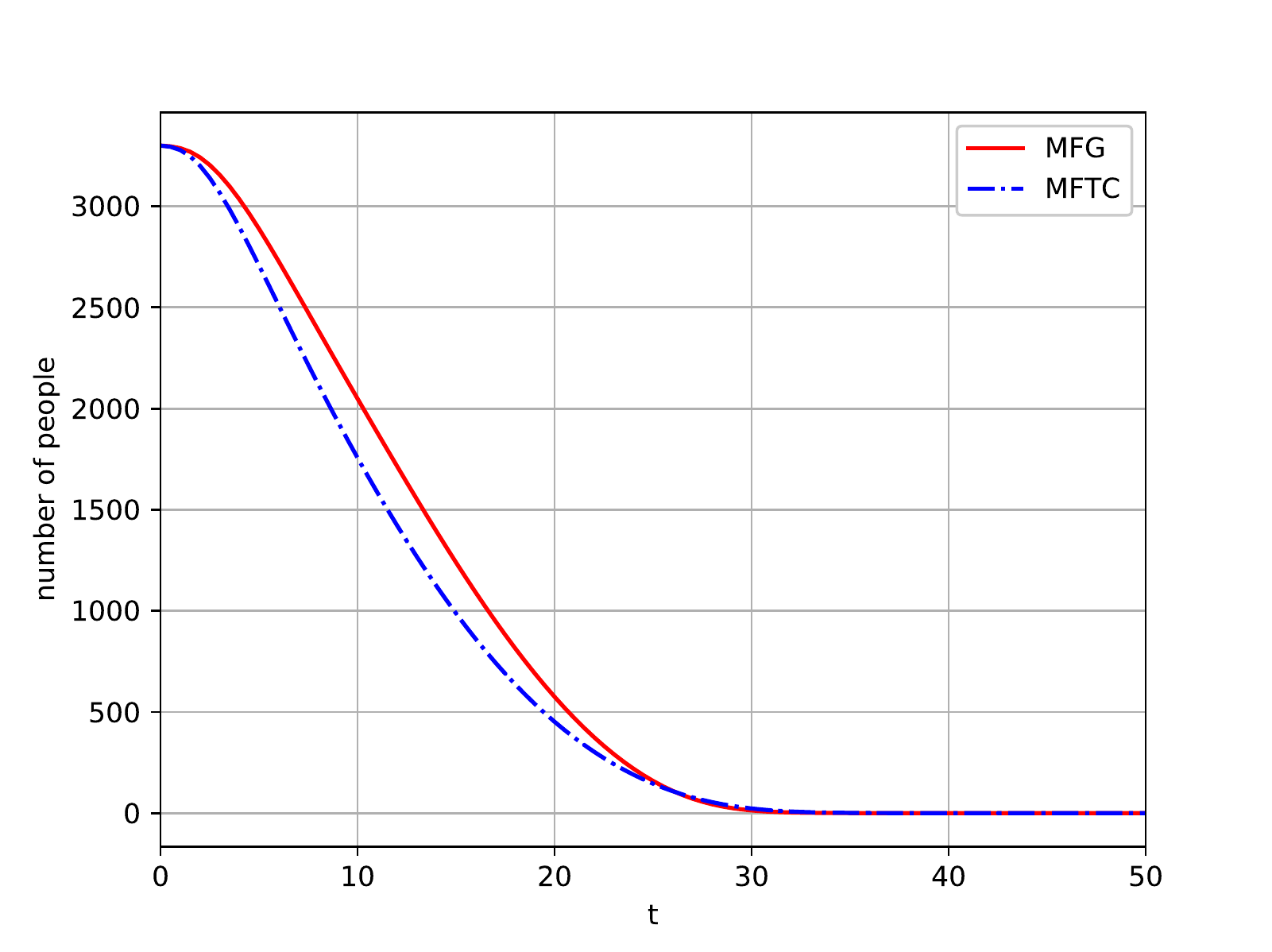}
	\end{subfigure}
	\caption{\label{fig:MFTC-MFG-congestion-evol-totalmass} Evolution of the total remaining number of people in the room for the mean field game (red line) and the mean field type control (dashed blue line).}
\end{figure}

\section{MFGs in macroeconomics}\label{sec:mfgs-macroeconomics}
The material of this section  is taken from a  work of the first author with  Jiequn Han,  Jean-Michel Lasry,  Pierre-Louis Lions,  and Benjamin Moll, 
 see~\cite{achdou2017income}, see also~\cite{MR3268061}. 
In economics, the ideas underlying MFGs have been investigated in the so-called
   {\sl heterogeneous agents models}, see  for example~\cite{bewley, huggett, aiyagari,krusell-smith}, for a long time before the mathematical formalization of Lasry and Lions.
The latter models mostly involve discrete in time  optimal control problems, because the economists who proposed them felt  more comfortable with the related mathematics. 
The connection  between heterogeneous agents models and MFGs is discussed in~\cite{MR3268061}. 
\\
We will first describe the simple yet prototypical Huggett model, see~\cite{huggett, bewley}, in which  the  constraints on the state variable play a key role.   We will then discuss the finite difference scheme for the related system of differential equations, putting the stress on the boundary conditions. Finally, we will present numerical simulations of the richer  Aiyagari
 model, see~\cite{aiyagari}.\\
Concerning HJB  equations, state constraints and related numerical schemes have been much studied in the literature since the pioneering works of Soner and Capuzzo Dolcetta-Lions, see~\cite{MR838056,MR861089,MR951880}. On the contrary, the boundary conditions  for the related invariant measure  had not been addressed before~\cite{MR3268061,achdou2017income}. The recent references~\cite{cannarsa2017existence,cannarsa2018mean} contain a rigorous study of MFGs with state constraints in a different setting from the economic models considered below.

\subsection{A prototypical model with heterogeneous agents: the Huggett model}

\paragraph{\bf The optimal control problem solved by an agent}
We consider households which are heterogeneous in  wealth $x$ and income $y$. The dynamics of the wealth of a  household is given by 
\begin{displaymath}     
 dx_t = (r x_t -c_t + y_t) dt,
    \end{displaymath}
where $r$ is an interest rate and $c_t$ is the consumption (the control variable). The income $y_t$ is a two-state Poisson process with intensities $\lambda_1$ and $\lambda_2$, i.e.
$y_t \in \{ y_1, y_2\}$ with $y_1<y_2$, and 
      \begin{displaymath} 
        \begin{split}
          &\PP\left(y_{t+\dt} =y_1\right| \left.  y_{t} =y_1 \right)= 1- \lambda_1 \dt +o(\dt), \quad     \PP\left(y_{t+\dt} =y_2\right| \left.  y_{t} =y_1 \right)= \lambda_1 \dt +o(\dt) ,
 \\
          &   \PP\left(y_{t+\dt} =y_2\right| \left.  y_{t} =y_2 \right)= 1-\lambda_2 \dt +o(\dt), \quad     
          \PP\left(y_{t+\dt} =y_1\right| \left.  y_{t} =y_2 \right)= \lambda_2 \dt +o(\dt).
        \end{split}
      \end{displaymath}
Recall that a negative wealth means debts; there is a {\sl borrowing constraint}: the wealth of a given household cannot be less than a given {\sl borrowing limit}  $\underline{x}$. 
In the terminology of control theory,  $  x_t\ge  \underline{x}$ is a constraint on the state variable.\\
A  household solves the optimal control problem
    \begin{equation*}
      \max_{\{c_t\}}  \mathbb{E}\int_0^\infty e^{-\rho t} u(c_t)dt \quad \hbox{subject to }
      \left\{
        \begin{array}[c]{rcl}
          dx_t &=& (y_t + r x_t - c_t)dt,\\
          x_t  &\geq& \underline{x},    
        \end{array}
      \right.   
    \end{equation*}
    where
   \begin{itemize}   
     \item $\rho$ is a positive discount factor
     \item $u$ is a  utility function,  strictly increasing and strictly concave, e.g. the CRRA (constant relative risk aversion) utility:
       \[ u(c)=c^{1-\gamma}/(1-\gamma),\qquad \gamma>0.\]
     \end{itemize}
For $i=1,2$, let $ v_i(x)$ be the value of the optimal control problem solved by an agent when her wealth is $x$ and her income $y_i$.  
The value functions $(v_1(x),v_2(x))$ satisfy the system of differential equations:
 \begin{eqnarray}\label{eq:macroeco:1}
   -\rho v_1 + H (x,y_1, \partial_x v_1) + \lambda_1 v_2(x)-\lambda_1 v_1(x)=0,\quad
 x>\underline{x},\\\label{eq:macroeco:2}
   -\rho v_2 + H (x,y_2, \partial_x v_2) + \lambda_2 v_1(x)-\lambda_2 v_2(x)=0,\quad x>\underline{x},
 \end{eqnarray}
where the Hamiltonian $H$ is given by
\begin{equation*}
H(x,y,p)= \max_{c \ge 0}    \Bigl(   (y+rx-c)p +u(c) \Bigr).
\end{equation*}
The system of differential equations (\ref{eq:macroeco:1})-(\ref{eq:macroeco:2}) must be supplemented with boundary conditions connected to the state constraints  $x\ge\underline{x}$.
Viscosity solutions of Hamilton-Jacobi equations with such boundary conditions have been studied in~\cite{MR838056,MR861089,MR951880}. 
It is convenient to introduce the  non-decreasing and non-increasing envelopes $H^{\uparrow}$ and $H^{\downarrow}$ of $H$:
\begin{equation*}
  \begin{split}
H^\uparrow (x,y,p)&= \max_{0\le c \le y+rx}    \Bigl(   (y+rx-c)p +u(c) \Bigr),\\
H^\downarrow (x,y,p)&= \max_{  \max(0,y+rx)\le c}    \Bigl(   (y+rx-c)p +u(c) \Bigr).    
  \end{split}
\end{equation*}
It can be seen that 
\begin{displaymath}
  H(x,y,p)=H^\uparrow (x,y,p)+H^\downarrow (x,y,p)- \min_{p\in \R} H(x,y,p).
\end{displaymath}
 The boundary conditions associated to the state constraint can be written
 \begin{eqnarray}\label{eq:macroeco:3}
   -\rho v_1  (\underline{x}) + H^\uparrow ( \underline{x},y_1, \partial_x v_1) + \lambda_1 v_2 (\underline{x})-\lambda_1 v_1 (\underline{x})=0,\\
\label{eq:macroeco:4}   -\rho v_2 (\underline{x}) + H^\uparrow  (\underline{x},y_2, \partial_x v_2) + \lambda_2 v_1 (\underline{x})-\lambda_2 v_2 (\underline{x})=0.
 \end{eqnarray}

Consider for example the prototypical case when $u(c)=\frac 1  {1-\gamma} c^{1-\gamma}$ with $\gamma>0$. Then 
\begin{equation*}
H (x,y,p)= \max_{0\le c }    \Bigl(   (y+rx-c)p + \frac 1 {1-\gamma} c^{1-\gamma}\Bigr)= (y+rx)p +\frac {\gamma}{1-\gamma } p^{1-\frac 1 \gamma} .
\end{equation*}
In Figure~\ref{fig:macroeco:0}, we plot the graphs of the functions $p\mapsto H(x,y,p)$, $p\mapsto H^{\uparrow}(x,y,p)$ and $p\mapsto H^{\downarrow}(x,y,p)$ for $\gamma=2$.
\begin{center}
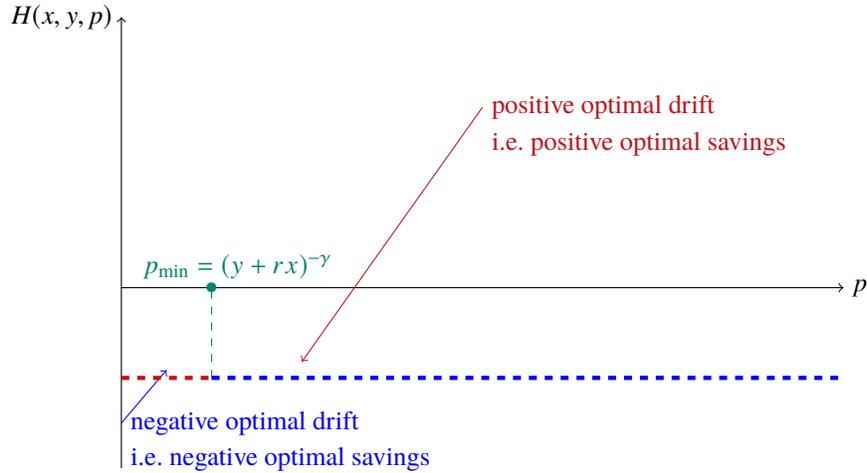

  \begin{tikzpicture}[scale=1.2]  
    \draw[->] (0,0) -- (8,0) node[right] {$p$};
    \draw[->] (0,-2) -- (0,3)    ;
    \draw[color=PineGreen, dashed] (1,-1) -- (1,0)    ;
     \draw[color=blue,<-] (0.5,-0.914)-- (0,-1.5)  node[right,color=blue] {negative optimal drift};
    \draw (0,-1.9)node[right,color=blue] {i.e. negative optimal savings};
    \draw[color=blue, ultra thick, domain=0:1] plot[id=0112+] function{    x-2*sqrt(x)} ;
    \draw[color=blue,ultra thick,dashed] (1,-1)-- (8,-1)  ;
    \draw[color=Red,<-] (2,-0.82)-- (4,2)  node[right,color=Red] {positive optimal drift};
    \draw[color=Red,ultra thick,dashed] (0,-1)-- (1,-1)  ;
    \draw (4,1.6)node[right,color=Red] {i.e. positive optimal savings};
    \draw[color=Red, ultra thick, domain=1:6] plot[id=0113+] function{    x-2*sqrt(x)} ;
    \draw(0,3)      node[left] {$ H(x,y,p)$};
    \draw[color=PineGreen] (1,0) node{$\bullet$};
    \draw(1,0)   node[above, color=PineGreen] {$\quad \quad  p_{\min}= (y+rx)^{-\gamma}$};
  \end{tikzpicture}
\captionof{figure}{The bold line is the graph of the function $p\mapsto H(x,y,p)$.  The dashed and bold  blue  lines form the graph of $p\mapsto H^{\downarrow}(x,y,p)$. The  dashed and bold red lines form the graph of $p\mapsto H^{\uparrow}(x,y,p)$}
\label{fig:macroeco:0}
\end{center}
\paragraph{\bf The ergodic measure and the coupling condition}
In Huggett model, the  interest rate $r$ is an unknown which is found from the fact that the considered economy neither creates nor destroys  wealth: the aggregate wealth 
! can be fixed to $0$. 
\\ 
 On the other hand, the latter quantity can be deduced from the ergodic measure of the process, i.e. $m_1$ and $m_2$,
the  wealth distributions of households with respective income $y_1$ and $y_2$, defined by
\begin{displaymath}
  \lim_{t\to +\infty} \EE(\phi(x_t)|y_t=y_i) =\langle m_i, \phi \rangle, \quad \forall \phi\in C_b([\underline x, +\infty)).
\end{displaymath}
 The measures  $m_1$ and $m_2$  are defined on $[\underline{x},+\infty)$ and  obviously satisfy the following property:
 \begin{equation}
   \label{eq:macroeco:5}
\int_{x\ge \underline{x}} dm_1(x)+ \int_{x\ge \underline{x}} dm_2(x)=1. 
 \end{equation}
Prescribing the aggregate  wealth amounts to writing that
\begin{equation}
  \label{eq:macroeco:6}
\int_{x\ge \underline{x}} xdm_1(x)+ \int_{x\ge \underline{x}} xdm_2(x)=0.
\end{equation}
In fact (\ref{eq:macroeco:6}) plays the role of the coupling condition in MFGs. We will discuss this point after having characterized the  measures $m_1$ and $m_2$: 
it is well known that they satisfy the Fokker-Planck equations 
 \begin{eqnarray*}
-\partial _x\left(  m_1 H_p (\cdot,y_1, \partial_x v_1) \right)   + \lambda_1     m_1-\lambda_2 m_2=0,\\
-\partial _x\left(  m_2 H_p (\cdot,y_2, \partial_x v_2) \right)   + \lambda_2     m_2-\lambda_1 m_1=0,
 \end{eqnarray*}
in the sense of distributions in  $(\underline{x},+\infty)$. 
\\
Note that even if $m_1$ and $m_2$ are regular in the open interval  $(\underline{x},+\infty)$, it may happen that the optimal strategy of a representative agent (with income $y_1$ or $y_2$) consists  of reaching the borrowing limit $x=\underline{x}$ in finite time and staying there; in such a case, the ergodic measure has a singular part supported on $\{x=\underline{x}\}$, and its absolutely continuous part (with respect to Lebesgue measure) may blow up near the boundary.
\\
For this reason, we  decompose $m_i$ as the sum of a measure  absolutely continuous with respect to Lebesgue measure on $(\underline x,+\infty)$  with density $g_i$
 and possibly of a Dirac mass at $x=\underline x$: for each Borel set $A\subset [\underline{x},+\infty)$,
  \begin{equation}
    \label{eq:macroeco:16}
 m_i(A)= \int_A g_i(x) dx + \mu_i   \delta_{\underline{x}}(A),\quad i=1,2.
  \end{equation}
Here, $g_i$ is a  measurable and nonnegative function.

 Assuming that $\partial_x v_i(x)$ have limits when $x\to \underline{x}_+$  (this can be justified rigorously for suitable choices of $u$),
 we see by definition of the ergodic measure, that for all test-functions $(\phi_1,\phi_2)\in \left(C^1_c([\underline x, +\infty))\right)^2$, 

   \begin{equation*}
    \begin{split}
  0=&  \ds
 \int_{x> \underline x}  g_1(x)  \Bigl(  \lambda_1 \phi_2(x)   -\lambda_1 \phi_1(x)          
  +   H_p(x ,y_1,\partial_x v_1(x))  \partial_x \phi_1(x)  \Bigr)dx 
  \\
 &\ds + \int_{x> \underline x}  g_2(x) \Bigl(  \lambda_2 \phi_1 (x)     -\lambda_2 \phi_2 (x)    
  +    H_p(x,y_2,\partial_x v_2(x))  \partial_x \phi_2(x)  \Bigr) dx
 \\
& \ds+
\mu_1 \Bigl(  \lambda_1 \phi_2(\underline{x}) -\lambda_1 \phi_1(\underline{x}) 
    +  H_p^\uparrow (\underline{x}, y_1 , \partial_x v_1({\underline{x}}_+))  \partial_x \phi_1 (\underline{x}) \Bigr)
\\ &\ds +
\mu_2 \Bigl( \lambda_2 \phi_1(\underline{x}) -\lambda_2 \phi_2(\underline{x}) 
    +  H_p^\uparrow (\underline{x}, y_2 , \partial_x v_2({\underline{x}}_+))  \partial_x \phi_2 (\underline{x}) \Bigr).
\end{split}
 \end{equation*}
It is also possible to use $\phi_1=1$ and $\phi_2=0$ as test-functions in the equation above. This  yields 
\begin{displaymath}
\lambda_1\int_{x\ge \underline{x}} dm_1(x)= \lambda_2\int_{x\ge \underline{x}} dm_2(x)  .
\end{displaymath}
Hence 
\begin{equation}
  \label{eq:macroeco:7}
\int_{x\ge \underline{x}} dm_1(x) =\frac {\lambda_2}{\lambda_1 +\lambda_2},\quad\quad   \int_{x\ge \underline{x}} dm_2(x)=\frac {\lambda_1}{\lambda_1 +\lambda_2}.
\end{equation}
\paragraph{\bf Summary}
To summarize, finding an equilibrium in the  Huggett model  given that the aggregate wealth is $0$ consists in looking for $(r,v_1,v_2, m_1, m_2)$ such that 
\begin{itemize}
\item the functions $v_1$ and $v_2$ are viscosity solutions to (\ref{eq:macroeco:1}),(\ref{eq:macroeco:2}),(\ref{eq:macroeco:3}),(\ref{eq:macroeco:4})
\item the measures $m_1$ and $m_2$ satisfy   (\ref{eq:macroeco:7}), 
 $m_i=g_i L+\mu_i \delta_{x=\underline x}$,
 where $L$ is the Lebesgue measure on $(\underline{x}, +\infty)$, $g_i$ is a measurable nonnegative function  $(\underline{x}, +\infty)$ and
  for all $i=1,2$, for all test function $\phi$:
  \begin{equation}
    \label{eq:macroeco:15}
 \left. \begin{array}[c]{r}
    \ds 
   \ds 
 \int_{x> \underline x}  \left( \lambda_j  g_j(x)-\lambda_i  g_i(x) \right) \phi(x) dx 
  +  \int_{x> \underline x}  H_p(x ,y_i,\partial_x v_i(x))  \partial_x \phi(x)  dx 
\\
 + \ds
\mu_i\Bigl(  -\lambda_i \phi(\underline{x})     +  H_p^\uparrow (\underline{x}, y_i , \partial_x v_i(\underline{x}_+))  \partial_x \phi (\underline{x}) \Bigr)
+\mu_j  \lambda_j \phi(\underline{x})
  \end{array}
\right\}=0,
  \end{equation}
where $j=2$ if $i=1$ and $j=1$ if $i=2$.
\item the aggregate wealth is fixed, i.e. (\ref{eq:macroeco:6}).  By contrast with the previously considered MFG models, the coupling in the Bellman equation does not come from a coupling cost, but from the implicit relation (\ref{eq:macroeco:6}) which can be seen as an equation for $r$.  Since the latter coupling is only through the constant value $r$, it can be said to be weak.
\end{itemize}
\paragraph{\bf Theoretical results}
The following theorem, proved in~\cite{achdou2017income}, gives quite accurate information on  the equilibria for which $r<\rho$:
\begin{theorem}
  \label{sec:bf-summary}
Let $c_i^*(x)$ be the optimal consumption of an agent with  income $y_i$ and wealth $x$. 
\begin{enumerate}
\item If $r<\rho$,  then the optimal drift (whose economical interpretation is the optimal saving policy) corresponding to income $y_1$, namely $rx+y_1-c_1^*(x)$, is negative for all $x>\underline{x}$.
\item If $r<\rho$ and  $-\frac {u''(y_1+r\underline x)}{u'(y_1+r\underline x)}<+\infty$, then there exists a positive number $\nu_1$ (that can be written explicitly) such that  for $x$ close to $\underline{x}$, 
\begin{eqnarray*}
  rx+y_1-c_1^*(x) \sim -\sqrt{2\nu_1} \sqrt{x-\underline x}<0,
\end{eqnarray*}
so the  agents for which the income stays at the lower value $y_1$ hit the borrowing limit in finite time, and
\begin{displaymath}
 \mu_1>0, \qquad
 g_1(x)\sim \frac {\gamma_1 }{\sqrt{x-\underline{x}}} ,
\end{displaymath}
for some positive value $\gamma_1$ (that can be written explicitly).
\item If  $r<\rho$ and $\sup_{c} \left(-\frac {cu''(c)}{u'(c)}\right)<+\infty$, then  there exists $\underline x \le \overline{x}<+\infty$ such that
\begin{eqnarray*}
  rx+y_2-c_2^*(x)<0, &\quad & \hbox{ for all } x>\overline{x},\\
rx+y_2-c_2^*(x)>0, &\quad & \hbox{ for all } \underline x < x<\overline{x}.
\end{eqnarray*}
Moreover $\mu_2=0$ and if  $\overline {x}>\underline{x}$, then for some positive constant $\zeta_2$,
 $rx+y_2-c_2^*(x) \sim \zeta_2 (\overline{x}-x)$ for $x$ close to $\overline{x}$.
\end{enumerate}
\end{theorem}
From Theorem~\ref{sec:bf-summary}, we see in particular that agents whose  income remains at the lower value 
 are drifted toward the borrowing limit and get stuck there, while agents whose income remains at the higher value get richer if their wealth is small enough.
 In other words, the only way for an agent in the low income state to avoid being drifted to  the borrowing limit is to draw the high income  state.
Reference~\cite{achdou2017income}  also contains existence and uniqueness results.
 \subsection{A finite difference method for the Huggett model}\label{sec:finite-diff-meth}
 We wish to simulate the Huggett model in  the interval $[\underline x,  \tilde x]$.
Consider a uniform grid on $[\underline x,  \tilde x]$ with step $h=  \frac {\tilde x-\underline x}{N_h}  $:  we set $x_i=\underline x+ ih$, $i=0,\dots, N_h$.
The discrete approximation of $v_j(x_i)$ is named $v_{i,j}$.
The discrete version of the Hamiltonian $H$ involves the  function $\cH:  \R^4\to \R$,
\begin{displaymath}
  \cH(x,y,\xi_r,\xi_{\ell})= H^{\uparrow}(x,y,\xi_{r})+ H^{\downarrow}(x,y,\xi_\ell)- \min_{p\in \R} H(x,y,p).
\end{displaymath}
We do not discuss the boundary conditions at $x=\tilde{x}$ in order to focus on the difficulties coming from the state constraint $x_t\ge \underline x$. In fact, if $\tilde x$ is large enough, state constraint boundary conditions can also be chosen at $x=\tilde x$. 
The numerical monotone scheme for (\ref{eq:macroeco:1}),(\ref{eq:macroeco:2}),(\ref{eq:macroeco:3}),(\ref{eq:macroeco:4}) is 
\begin{eqnarray}
  \label{eq:macroeco:8}
  -\rho v_{i,j}+ \cH\left(x_i,y_j,\frac { v_{i+1,j}-v_{i,j} } h,\frac { v_{i,j}-v_{i-1,j} } h\right) +\lambda_i \left( v_{i,k}- v_{i,j}  \right)&=&0,\quad   \quad \hbox{for } 0<i<N_h,
\\\label{eq:macroeco:9}
  -\rho v_{0,j}+ H^\uparrow \left( \underline x,y_j,\frac { v_{1,j}-v_{0,j} } h \right) +\lambda_i \left( v_{0,k}- v_{0,j}  \right)&=&0,\quad   \quad \hbox{for } i=0,
\end{eqnarray}
where $k=2$ if $j=1$ and $k=1$ if $j=2$. In order to find the discrete version of the Fokker-Planck equation, we start by writing (\ref{eq:macroeco:8})-(\ref{eq:macroeco:9}) in the following compact form:
\begin{equation}\label{eq:macroeco:10}
-\rho V +\cF(V)=0  ,
\end{equation}
with self-explanatory notations.  The differential of  $\cF$ at $V$  maps $W$ to the grid functions whose $(i,j)$-th components is:
\begin{itemize}
\item if $ 0<i<N_h$:
\begin{equation}
\label{eq:macroeco:11}
   H_p^\uparrow \left( x_i,y_j,\frac { v_{i+1,j}-v_{i,j} } h \right)\frac { w_{i+1,j}-w_{i,j} } h
+  H_p^\downarrow \left( x_i,y_j,\frac { v_{i,j}-v_{i-1,j} } h \right)\frac { w_{i,j}-w_{i-1,j} } h
 +\lambda_i \left( w_{i,k}- w_{i,j}  \right),  
\end{equation}
\item If $i=0$:
\begin{equation}
\label{eq:macroeco:12}
   H_p^\uparrow \left( \underline x,y_j,\frac { v_{1,j}-v_{0,j} } h \right)\frac { w_{1,j}-w_{0,j} } h
 +\lambda_i \left( w_{0,k}- w_{0,j}  \right).
\end{equation}
\end{itemize}
We get the discrete Fokker-Planck equation by requiring that $  \langle D\cF(V)W, M\rangle =0$ for all $W$,
 where  $M=(m_{i,j})_{0\le i\le N_h, j=1,2}$; more explicitly, we obtain:
\begin{equation}\label{eq:macroeco:13}
0=\left\{
    \begin{array}[c]{l}
\lambda_k  m_{i,k}-\lambda_j  m_{i,j} 
 \\ 
\ds +  \frac 1 {h} \left(  m_{i,j} H^{\downarrow}_p\left(x_{i},y_j, \frac { v_{i,j}- v_{i-1,j}}{h}\right)
-m_{i+1,j} H_p^{\downarrow}\left(x_{i+1}, y_j, \frac { v_{i+1,j}- v_{i,j}}{h}\right)  \right)   \\
\\ \ds
- \frac 1 {h} \left(  m_{i,j} H_p^{\uparrow}\left(x_{i}, y_j, \frac { v_{i+1,j}- v_{i,j}}{h}\right)
- m_{i-1,j} H_p^{\uparrow}\left(x_{i-1},y_j, \frac { v_{i,j}- v_{i-1,j}}{h}\right)\right)   ,
    \end{array}
\right.
\end{equation}
for $0<i<N_h$, and 
 \begin{equation}\label{eq:macroeco:14}
 0=
 \ds  \lambda_k  m_{0,k}-\lambda_j  m_{0,j}
-\frac  1 h \left(   m_{0,j} H_p^{\uparrow}\left(\underline x, y_j, \frac { v_{1,j}- v_{0,j}}{h}\right)  +
m_{1,j} H_p^{\downarrow}\left(x_{1}, y_j, \frac { v_{1,j}- v_{0,j}}{h}\right) \right).
 \end{equation}
Assuming for simplicity that in (\ref{eq:macroeco:16}), $(g_j)_{j=1,2}$ are absolutely continuous with respect to the Lebesgue measure, 
equations  (\ref{eq:macroeco:13}) and (\ref{eq:macroeco:14}) provide a consistent discrete scheme for (\ref{eq:macroeco:15}) if $h m_{0,j}$ is seen as the discrete version of $\mu_j$
and  if $ m_{i,j}$ is the discrete version of $g_{j}(x_i)$. If we do not suppose that $g_j$ is absolutely continuous with respect to the Lebesgue measure,
then $m_{i,j}$ may be seen as a discrete version of $\frac 1 h \int_{x_i -h /2 } ^{x_i +h/2 } dg_j(x)$.  \\
The consistency of (\ref{eq:macroeco:13}) for $i>0$ is obtained as usual. Let us focus on (\ref{eq:macroeco:14}):
 assume that  the following convergence holds:
   \begin{displaymath}
  \lim_{h\to 0} \max_{i,j} | v_{i,j}- v(x_{i},y_j)|=0 \quad \hbox{and}\quad    \lim_{h\to 0} \max_{j}  \left| \frac { v_{1,j}- v_{0,j}}{h} - \partial_x v_j (\underline x)\right|=0.
   \end{displaymath}
Assume  that $H_p(\underline x, \partial_x v_j(\underline x))>\varepsilon$  for a fixed positive number $\varepsilon>0$. Then for $h$ small enough, 
 \begin{displaymath}
  H_p^{\uparrow}\left( x_{i},y_j, \frac { v_{1,j}- v_{0,j}}{h}\right) >0\quad  \hbox{ and}\quad 
   H_p^{\downarrow}\left( x_{i},y_j,\frac { v_{1,j}- v_{0,j}}{h}\right) =0,   
 \end{displaymath}
for $i=0,1$. Plugging this information in (\ref{eq:macroeco:14}) yields that $m_{0,j}$ is of the same order as $hm_{0,k}$, in agreement with the fact that since the optimal drift is positive, there is no Dirac mass at $x=\underline x$.  In the opposite case, we see that (\ref{eq:macroeco:14}) is consistent with the asymptotic expansions found in Theorem~\ref{sec:bf-summary}.

\paragraph{\bf Numerical simulations}
On Figure \ref{fig:macroeco:1}, we plot the densities $g_1$ and $g_2$ and the cumulative distributions $\int_{\underline{x}}^{x} dm_1$ and $\int_{\underline{x}}^{x} dm_2$, computed by  two methods:
the first one is the finite difference descrived above. The second one consists of coupling the finite difference scheme described above for the HJB equation and a closed formula for the Fokker-Planck equation. Two different grids are used with $N_h=500$ and $N_h=30$.   We see that for $N_h=500$, it is not possible to distinguish the graphs
obtained by the two methods. The density $g_1$ blows up at $x=\underline{x}$ while $g_2$ remains bounded. From the graphs of the cumulative distributions, we also see that $\mu_1>0$ while $\mu_2=0$.

\begin{center}
  \includegraphics[width=16cm]{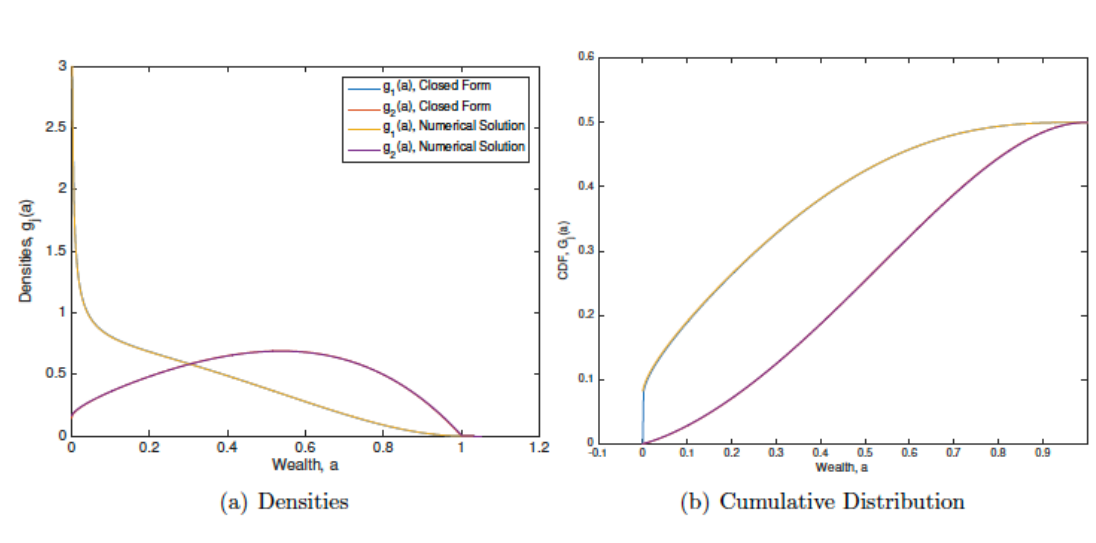}
\includegraphics[width=16cm]{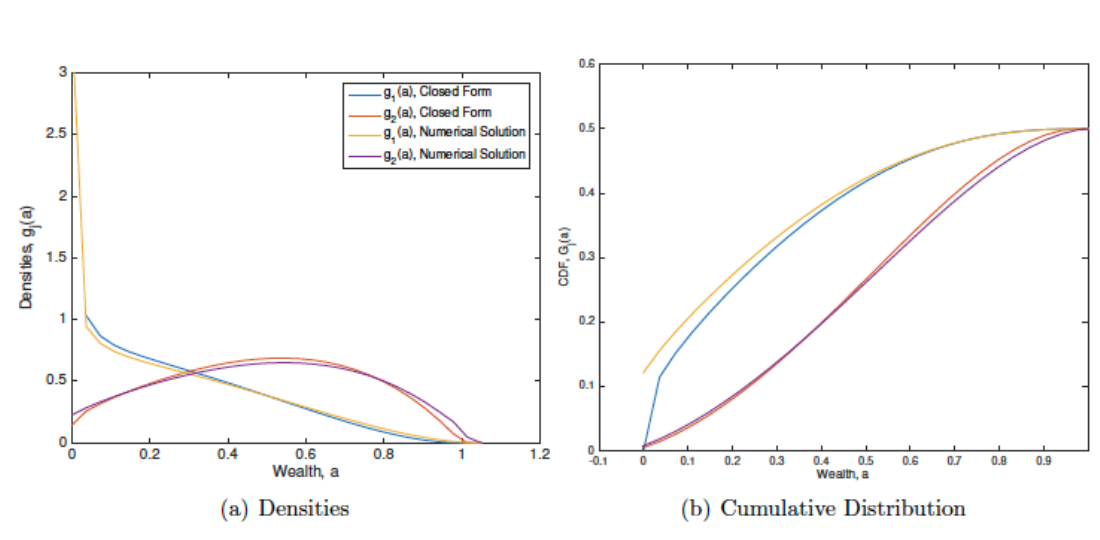} 
  \captionof{figure}{Huggett model: Left: the measures $g_i$  obtained by an explicit formula and by the finite difference method for two grid resolutions with $500$ nodes (Top) and 
$30$ nodes (Bottom). Right: the cumulative distributions  $\int_{\underline{x}}^{x} dm_i$.  }
  \label{fig:macroeco:1}
\end{center}

\subsection{The model of Aiyagari}
\label{sec:model-aiyagari}
There is a continuum of agents which are heterogeneous in  wealth $x$ and productivity $y$. The dynamics of the wealth of a  household is given by 
\begin{displaymath}     
 dx_t = (w y_t+r x_t -c_t ) dt,
    \end{displaymath}
where $r$ is the interest rate, $w$ is the level of wages, and $c_t$ is the consumption (the control variable). 
The dynamics of the productivity $y_t$ is given by the stochastic differential equation in $\R_+$:
\begin{displaymath}
   dy_t = \mu(y_t) dt +\sigma(y_t) dW_t,
\end{displaymath}
where $W_t$ is a standard one dimensional Brownian motion. Note that this models situations in which the noises in the productivity of the agents are all independent  (idiosyncratic noise). Common noise will be briefly discussed below.

 As for 
the Huggett model, there  is  a borrowing constraint $x_t\ge \underline{x}$, 
a  household tries to maximize the utility $     \mathbb{E}\int_0^\infty e^{-\rho t} u(c_t)dt$.  
To determine the interest rate $r$ and the level of wages $w$, Aiyagari considers that the production of the economy is given by the following Cobb-Douglas law:
\[  F(X,Y)=  A X^\alpha Y^{1-\alpha}  ,\]
for some $\alpha\in (0,1)$,
where, if $m(\cdot,\cdot)$ is the ergodic measure, 
\begin{itemize}
  \item $A$ is a  productivity factor
  \item $X= \int_{x\ge \underline  x} \int_{y\in \R_+}  x dm (x,y) $ is the aggregate capital
  \item $Y=   \int_{x\ge \underline  x} \int_{y\in \R_+}  y dm (x,y)$ is  the  aggregate  labor.
  \end{itemize}
 The level of wages $w$ and interest rate $r$ are obtained by  the equilibrium relation
  \begin{displaymath}
    (X,Y)= {\rm{argmax}}    \Bigl( F(X,Y)- (r+\delta ) X -w Y \Bigr),
  \end{displaymath}
  where $\delta$ is the rate of depreciation of the capital. This implies that 
  \begin{displaymath}
  r=\partial _X F (X,Y)-\delta, \quad\quad w = \partial _Y F (X,Y).
\end{displaymath}
\begin{remark}\label{sec:model-aiyagari-1}
Tackling a  common noise 
 is possible in the non-stationary  version of the model, by letting the productivity factor become a random process $A_t$ (for example $A_t$  may be a piecewise constant process with random shocks): this leads to a more complex setting  which is the continuous-time version of the famous Krusell-Smith model, see~\cite{krusell-smith}.
Numerical simulations of Krusell-Smith models have been carried out  by the authors of~\cite{achdou2017income},
 but,  in order to keep the survey reasonnably short, we will not discuss this topic. 
\end{remark}
The Hamiltonian of the problem is $  H( x, y , p)=\max_c \Bigl( u(c) + p ( w y + r x -c) \Bigr)$ and the mean field equilibrium is found by solving the system of partial differential  equations:
\begin{eqnarray}
  \label{eq:macroeco:17}    0=\ds   \frac{\sigma^2(y)}{2}\partial_{yy}v +\mu(y)  \partial_y v  +     H(x,y,\partial_x v )   - \rho v   ,\\
\label{eq:macroeco:18}    0= \ds  - \frac{1}{2} \partial_{yy} \left(\sigma^2(y)m\right)+\partial_y \Bigl(\mu(y)m\Bigr) + \partial_x \Bigl(m H_p(x,y,\partial_x v )  \Bigr) ,
\end{eqnarray}
in $(\underline{x},+\infty)\times \R_+$
with suitable boundary conditions on the line $\{x=\underline x\}$ linked to the state constraints, $ \int_{x\ge \underline  x} \int_{y\in \R_+}  dm(x,y)=1$,
and the  equilibrium condition
\begin{equation}
\label{eq:macroeco:19}
  \begin{array}[c]{ll}
    \ds r=\partial _X F (X,Y)-\delta,\quad\quad &\ds w= \partial _Y F (X,Y),\\
\ds    X=  \ds \int_{x\ge \underline  x} \int_{y\in \R_+}  x dm (x,y) & Y=  \ds\int_{x\ge \underline  x} \int_{y\in \R_+} y dm (x,y).
  \end{array}
\end{equation}
 The boundary condition for the value function can be written
 \begin{equation}
    \label{eq:macroeco:20}   0=\ds  \frac{\sigma^2(y)}{2}\partial_{yy}v +\mu(y)  \partial_y v  +     H^{\uparrow}(\underline 
x,y,\partial_x v )  - \rho v  ,
 \end{equation}
where $p\mapsto H^{\uparrow}(x,y,p)$ is the non-decreasing envelope of $p \mapsto H(x,y,p)$. \\
We are going to look for $m$ as the sum of a measure  which is absolutely continuous with respect to the two dimensional Lebesgue measure
 on $(\underline x,+\infty)\times \R_+$ with density $g$  and of a measure  $\eta$ supported in the line $\{x=\underline x\}$:
 \begin{equation}
   \label{eq:macroeco:22}
    dm(x,y) =  g(x,y)  dxdy +  d\eta( y) ,
 \end{equation}
and for all test function $\phi$:
  \begin{equation}
\label{eq:macroeco:21}
 \left. \begin{array}[c]{r}
    \ds 
\int_{x> \underline  x} \int_{y\in \R_+} g(x,y) \left( \frac{\sigma^2(y)}{2} \partial_{yy} \phi (x,y)  + H_p(x ,y,\partial_x v (x,y))  \partial_x \phi(x,y) +\mu(y)  \partial_y \phi(x,y)    \right)dxdy 
\\
 + \ds
  \int_{y\in \R_+}  \left( \frac{\sigma^2(y)}{2} \partial_{yy} \phi (\underline{x},y)  + H_p^\uparrow (\underline{x} ,y,\partial_x v(\underline{x}_+,y))  \partial_x \phi (\underline{x},y)+\mu(y)  \partial_y \phi  (\underline{x},y)  \right) d\eta(y)
  \end{array}
\right\}=0.
  \end{equation}
Note that it is not possible to find a partial differential equation  for $\eta$ on the line $x=\underline{x}$, nor a local boundary condition for $g$, because it is not possible to express the term 
$\ds  \int_{y\in \R_+} H_p^\uparrow (\underline{x} ,y,\partial_x v(\underline{x}_+,y))  \partial_x \phi (\underline{x},y) d\eta(\underline x,y)$
 as a distribution acting on $\phi(\underline x,\cdot)$.
\\
The finite difference scheme for (\ref{eq:macroeco:17}),(\ref{eq:macroeco:20}),(\ref{eq:macroeco:22}),(\ref{eq:macroeco:21}),(\ref{eq:macroeco:19}) is found exactly in the same spirit as for Huggett model and we omit the details.
In Figure \ref{fig:macroeco:3}, we display the optimal saving policy $(x,y)\mapsto wy+rx-c^*(x,y)$ and the ergodic measure obtained by the above-mentioned finite difference scheme for Aiyagari model with $u(c)=-c^{-1}$. We see that the ergodic measure $m$ (right part of Figure \ref{fig:macroeco:3}) has a singularity  on the line $x=\underline{x}$ for small values of the productivity $y$, and for the same values of
$y$, the density of the absolutely continuous part of $m$ with respect to the two-dimensional Lebesgue measure blows up when $x\to \underline x$. The economic interpretation is thatthe agents with low productivity are drifted to the borrwing limit.   The singular part of the ergodic measure is of the same nature as the Dirac mass that was obtained for $y=y_1$ in the Huggett model.  It corresponds to 
the zone where the optimal drift is negative near  the  borrowing limit (see the left part of Figure \ref{fig:macroeco:3}).
\begin{center}
	\includegraphics[width=0.49\textwidth]{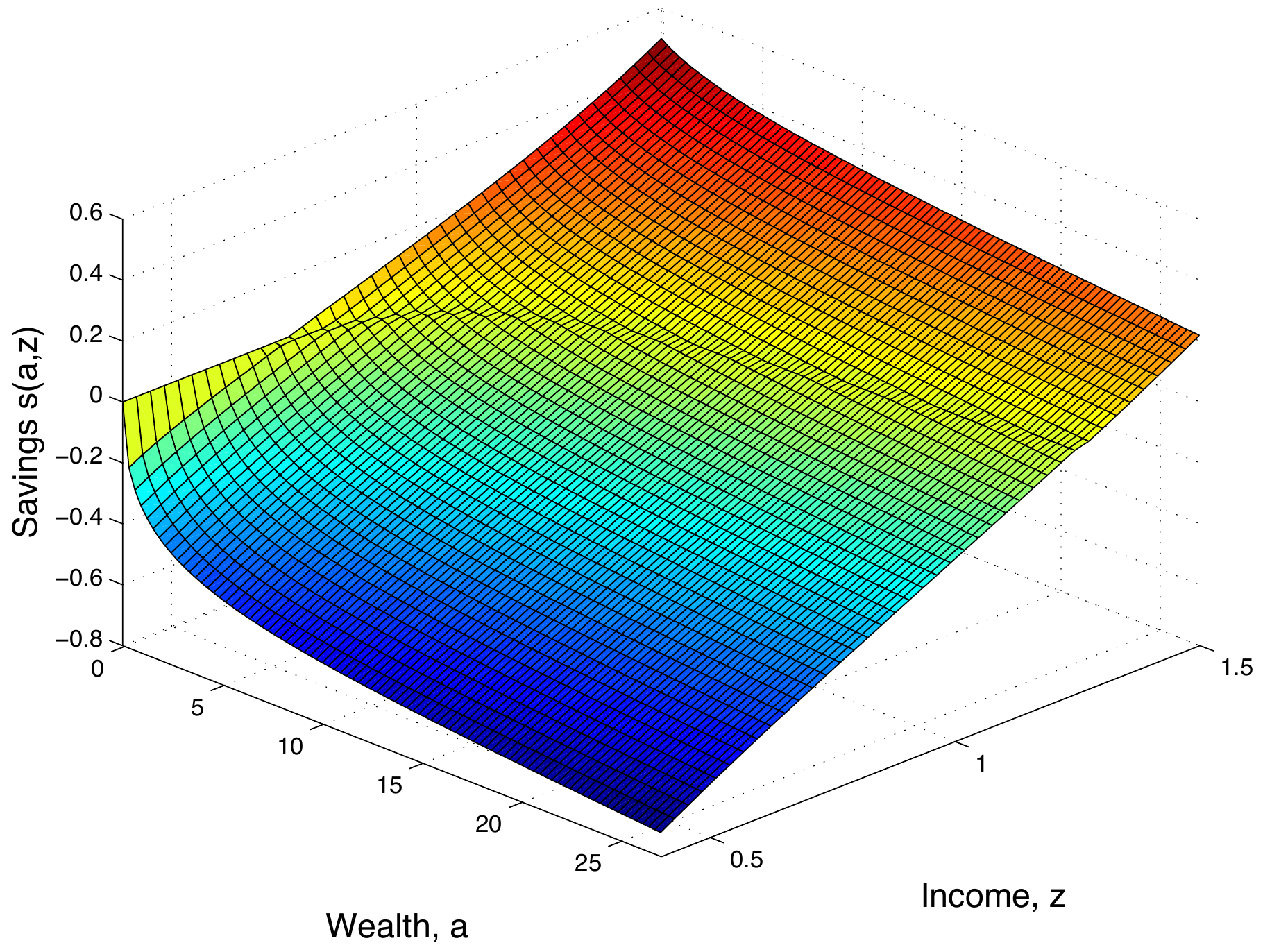}
	\includegraphics[width=0.49\textwidth]{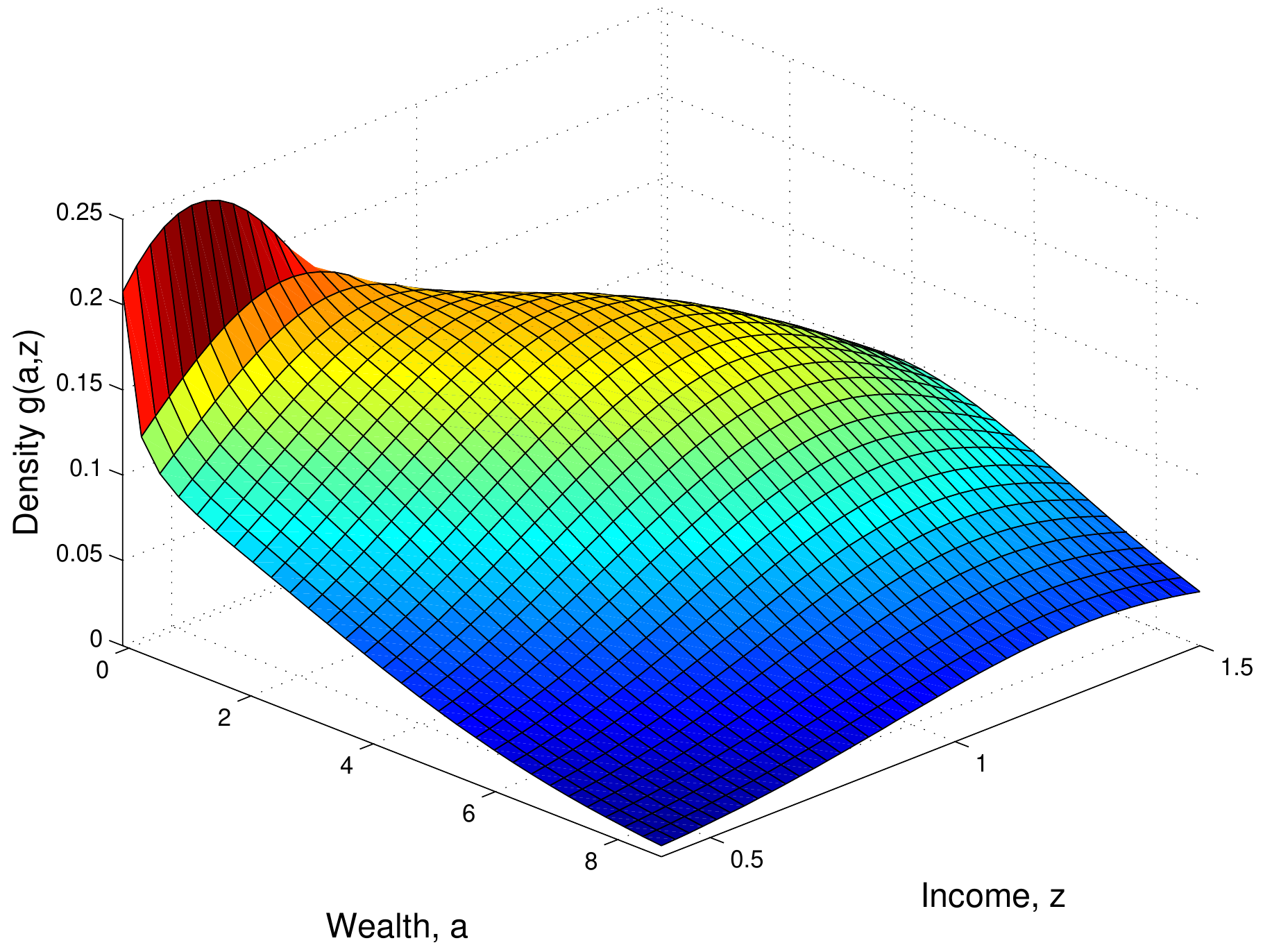}
\captionof{figure}{Numerical simulations of Aiyagari model with $u(c)=-c^{-1}$. Left: the optimal drift (optimal saving policy) $wy+rx-c^*(x,y)$. Right: the part of the ergodic measure which is absolutely continuous  with respect to Lebesgue measure.}
\label{fig:macroeco:3}
\end{center}

\section{Conclusion}

In this survey, we have put stress on finite difference schemes for the systems of forward-backward PDEs that may stem from the theory of mean field games,
and have discussed in particular some variational aspects and  numerical algorithms that may be used for solving the related systems of nonlinear equations.
We have also addressed in details two applications of MFGs to crowd motion and macroeconomics,  and a comparison of MFGs with mean field control: we hope that these  examples 
show well on the one hand, that the theory of MFGs is quite relevant for modeling the behavior of a large number of rational agents,  and 
 on the other hand, that several difficulties must be addressed in order to tackle realistic problems.

To keep  the survey short, we have not addressed the following interesting aspects:

\begin{itemize}
\item  Semi-Lagrangian schemes for the system of forward-backward PDEs.   While semi-Lagrangian schemes for optimal control problems have been extensively studied, 
much less has been done regarding the Fokker-Planck equation. In the context of MFGs,  semi-Lagrangian schemes have been investigated by F.~Camilli and F.~Silva,  E.~Carlini and F.~Silva,
see the references~\cite{MR2928379,MR3148086,MR3392626,MR3828859}.  Arguably, the advantage of such methods is their direct connection to the underlying optimal control problems, and a possible drawback may be the difficulty to address realistic boundary conditions.
\item An efficient algorithm for ergodic MFGs has been proposed by in~\cite{MR3601001,MR3882530}. The approach relies on a finite difference scheme and a least squares formulation, which is then solved using Gauss-Newton iterations. 
\item D. Gomes and his coauthors have proposed gradient flow methods for solving deterministic mean field games in infinite  horizon, see ~\cite{MR3698446,GomesSaude2018}, 
under a monotonicity assumption. Their main idea is that the solution to the system of PDEs can be recast as a zero of a monotone operator, an can thus be found by following the related gradient flow. They focus on the following example:
\begin{subequations}\label{eq:PDE-system-MFG-ergodic}
     \begin{empheq}[left=\empheqlbrace]{align}
     	\displaystyle
	& \lambda  +H_0(\cdot, \grad u) - \log(m) = 0, 
	\qquad \hbox{ in }  \TT,
\\
	\displaystyle
	& - \div\left( m \partial_pH_0 (\cdot, m, \grad u) \right) = 0, 
	\qquad \hbox{ in } \TT,
	\\
	& \int_{\TT} m = 1, \qquad \int_{\TT} u = 0, \qquad m>0 \hbox{ in } \TT,
     \end{empheq}
\end{subequations}
where the ergodic constant $\lambda$ is an unknown.
They consider the following monotone map :
\begin{displaymath}
  	A \begin{pmatrix}
	u
	\\
	m
	\end{pmatrix} 
	= 
	\begin{pmatrix}
	-\div\left( m(\cdot) \partial_p H_0 (\cdot, m(\cdot), \grad u(\cdot)) \right)
\\
	 - H_0(\cdot, \grad u) + \log(m)
	\end{pmatrix}.
      \end{displaymath}     
Thanks to the special choice of the coupling cost $\log(m)$, a  gradient flow method applied to $A$, i.e. 
\begin{displaymath}
  	\frac{d}{d \tau}
	\begin{pmatrix}
	u_\tau
	\\
	m_\tau
	\end{pmatrix} 
	=
	- A \begin{pmatrix}
	u_\tau
	\\
	m_\tau
	\end{pmatrix} - \begin{pmatrix}
	0
	\\
	\lambda_\tau
	\end{pmatrix} ,
\end{displaymath}
preserves the positivity of $m$ (this may not be true with other coupling costs, in which case additional projections may be needed). The real number $\lambda_{\tau}$ is used to enforce
the constraint $\int_\TT m_\tau=1$.  
\\
This idea has been studied on the aforementioned example and some variants but needs to be tested in the stochastic case (i.e., second order MFGs) and with general boundary conditions. The generalization to finite horizon is not obvious.
\item Mean field games related to  impulse control and optimal exit time  have been studied by C.~Bertucci, both from a theoretical and a numerical viewpoint, 
see~\cite{bertucci2018remarkUzawa}. In particular for MFGs related to  impulse control problems,  there remains a lot of difficult open issues.
\item High dimensional problems.  Finite difference schemes can only be used if the dimension $d$ of the state space is not too large, say $d\le 4$. 
Very recently, there have been  attempts to use machine learning methods in order to solve problems in higher dimension or with common noise, see e.g.~\cite{carmona2019convergence1,carmona2019convergence2,carmona2019model,ruthotto2019machine}.   The main difference with the methods discussed in the present survey is that these methods do not rely on a finite-difference scheme but instead use neural networks to approximate functions with a relatively small number of parameters.  Further studies remain necessary before it is possible to really evaluate these methods.
\item  Numerical methods for the master equation  when the state space is finite, see works in progress by the first author and co-workers.

\end{itemize}

To conclude, acknowledging the fact that the theory of mean field games have attracted a lot of interest in the past decade, the authors think that some of the most interesting open problems arise in the actual applications of this theory. Amongst the most fascinating aspects of mean field games are their interactions with social sciences and economics. A few examples of such interactions have been discussed in the present survey, and many more applications remain to be investigated.

\begin{acknowledgement}
  The research of the first author was partially supported by the ANR (Agence Nationale de la Recherche) through
MFG project ANR-16-CE40-0015-01.
\end{acknowledgement}

\input{references}
\end{document}

%% file: references.tex
%
%
\bibliographystyle{plain}
\bibliography{mfg-num-bib}